\UseRawInputEncoding
\documentclass[11pt, reqno]{amsart}
\usepackage{amsfonts,latexsym,enumerate}
\usepackage{amsmath}
\usepackage{amscd}
\usepackage{float,amsmath,amssymb,mathrsfs,bm,multirow,graphics}
\usepackage[dvips]{graphicx}
\usepackage[percent]{overpic}
\usepackage[numbers,sort&compress]{natbib}
\usepackage{todonotes}
\newcommand{\R}{{\Bbb R}}

\newcommand{\C}{{\Bbb C}}
\newcommand{\D}{{\Bbb D}}

\newcommand{\tr}{\text{\upshape tr\,}}
\newcommand{\res}{\text{\upshape Res\,}}
\newcommand{\diag}{\text{\upshape diag\,}}
\newcommand{\re}{\text{\upshape Re\,}}
\newcommand{\im}{\text{\upshape Im\,}}

\DeclareMathOperator{\dist}{dist}
\DeclareMathOperator{\sech}{sech}

\DeclareMathOperator{\rank}{rank}

\usepackage{tikz,graphicx}
\usepackage{tkz-euclide}
\usepackage{pict2e}
\usetikzlibrary{arrows,calc,chains, positioning, shapes.geometric,shapes.symbols,decorations.markings,arrows.meta}
\usetikzlibrary{patterns,angles,quotes}
\usetikzlibrary{decorations.markings}
\tikzset{middlearrow/.style={
			decoration={markings,
				mark= at position 0.6 with {\arrow{#1}} ,
			},
			postaction={decorate}
		}
	}
\tikzset{->-/.style={decoration={
				markings,
				mark=at position #1 with {\arrow{latex}}},postaction={decorate}}}
	
\tikzset{-<-/.style={decoration={
				markings,
				mark=at position #1 with {\arrowreversed{latex}}},postaction={decorate}}}
				
				\tikzset{
	master/.style={
		execute at end picture={
			\coordinate (lower right) at (current bounding box.south east);
			\coordinate (upper left) at (current bounding box.north west);
		}
	},
	slave/.style={
		execute at end picture={
			\pgfresetboundingbox
			\path (upper left) rectangle (lower right);
		}
	}
}
\tikzset{cross/.style={cross out, draw, 
         minimum size=2*(#1-\pgflinewidth), 
         inner sep=0pt, outer sep=0pt}}

\def\XXint#1#2#3{{\setbox0=\hbox{$#1{#2#3}{\int}$}
\vcenter{\hbox{$#2#3$}}\kern-.5\wd0}}

\allowdisplaybreaks

\newtheorem{theorem}{Theorem}[section]
\newtheorem{proposition}[theorem]{Proposition}

\newtheorem{lemma}[theorem]{Lemma}
\newtheorem{definition}[theorem]{Definition}
\newtheorem{assumption}[theorem]{Assumption}
\newtheorem{remark}[theorem]{Remark}

\newtheorem{RHproblem}[theorem]{RH problem}
\newtheorem{figuretext}{Figure}

\numberwithin{equation}{section}

\usepackage[colorlinks=true]{hyperref}
\hypersetup{urlcolor=blue, citecolor=red, linkcolor=blue}

\input epsf
\title[Scattering for the Boussinesq equation]
{Direct and inverse scattering for the Boussinesq equation with solitons}

\author{C. Charlier}
\address{CC: Centre for Mathematical Sciences, Lund University, 221 00 Lund, Sweden.}
\email{christophe.charlier@math.lu.se}

\author{J. Lenells}
\address{JL: Department of Mathematics, KTH Royal Institute of Technology, 100 44 Stockholm, Sweden.}
\email{jlenells@kth.se}

\begin{document}

\begin{abstract}
In a recent paper, we developed an inverse scattering approach to the Boussinesq equation in the case when no solitons are present. In this paper, we extend this approach to include solutions with solitons.
\end{abstract}

\maketitle

\noindent
{\small{\sc AMS Subject Classification (2020)}: 35C08, 35G25, 35Q15, 37K40, 76B15.}
%37K15

\noindent
{\small{\sc Keywords}: Boussinesq equation, solitons, Riemann-Hilbert problem.}

%\tableofcontents

\section{Introduction}

In 1872, Joseph Boussinesq introduced an approximate model for the propagation of dispersive waves of small amplitude in shallow water \cite{B1872}. In dimensionless units, this equation takes the form
\begin{align}\label{badboussinesq}
u_{tt} = u_{xx} + (u^2)_{xx} + u_{xxxx},
\end{align}
where $u(x,t)$ is a real-valued function and subscripts denote partial derivatives. The Boussinesq equation (\ref{badboussinesq}) is linearly unstable and is therefore also referred to as the ``bad" Boussinesq equation. It was understood in the 1970's that (\ref{badboussinesq}) is, at least formally, a completely integrable system. The construction of multi-soliton solutions of (\ref{badboussinesq}) by Hirota \cite{H1973} and the derivation of a Lax pair by Zakharov \cite{Z1974} were the key steps leading to this understanding. Berryman \cite{B1976} later showed that the solitons of (\ref{badboussinesq}) are linearly unstable (but left open the problem of their nonlinear stability). More recently, Bogdanov and Zakharov \cite{BZ2002} studied solitons of \eqref{badboussinesq} using a $\bar{\partial}$-dressing method.

In \cite{CLmain}, we developed an Inverse Scattering Transform (IST) approach to \eqref{badboussinesq} in the case when no solitons are present. In particular, assuming no solitons, we obtained a representation formula for the solution of the initial value problem in terms of the solution of a $1\times 3$ Riemann--Hilbert (RH) problem.
The purpose of this paper is to extend these results to the case when solitons are present. The solitons correspond to zeros of certain spectral functions, and as is usual we will restrict ourselves to the generic case of a finite number of simple zeros. 
Our main results can be summarized as follows:
\begin{enumerate}[$-$]
%\addtolength{\itemindent}{-0.5cm}
\item Theorem \ref{directth} treats the direct problem: it shows how to construct appropriate scattering data from some given initial data $\{u_{0}(x):=u(x,0),u_{1}(x):=u_{t}(x,0)\}$. 

\item Theorem \ref{inverseth} solves the inverse problem: it shows that the solution $u(x,t)$ can be recovered from the scattering data. 
%The presence of solitons leads us to consider a RH problem with poles and corresponding residue conditions. We use a vanishing lemma to prove that the solution of this RH problem exists.

\item Theorem \ref{IVPth} provides the solution of the initial value problem for (\ref{badboussinesq}) in the presence of solitons via inverse scattering.

\item Theorem \ref{blowupth} is a blow-up result and Theorem \ref{globalth} concerns the existence of global solutions with solitons. 
\end{enumerate}

Besides the objective of developing direct and inverse scattering transforms for the Boussinesq equation with solitons, an important motivation for this work is the larger goal of making progress on the soliton resolution conjecture for the Boussinesq equation. In fact, due to the ill-posedness of (\ref{badboussinesq}), the formulation of a soliton resolution conjecture for (\ref{badboussinesq}) requires some care, but roughly speaking, we expect that any generic solution of (\ref{badboussinesq}) within the physically relevant class singled out by Assumption \ref{nounstablemodesassumption} below, eventually decomposes into a non-decaying soliton component superimposed on a decaying radiation component. 

Another key motivation for the present paper is to cast light on the stability of soliton solutions of (\ref{badboussinesq}). We mentioned above that Berryman \cite{B1976} proved that the solitons of (\ref{badboussinesq}) are linearly unstable. Berryman's result is not surprising in light of the ill-posedness of (\ref{badboussinesq}). But it raises the question of whether the solitons of (\ref{badboussinesq}) can actually be observed. Equation (\ref{badboussinesq}) models the evolution of water waves \cite{B1872} (see also e.g. \cite{J1997}) and one would expect it to support soliton solutions that approximate Russell's famous wave of translation, which is characterized by its remarkable stability \cite{R1845}. So how can the solitons be linearly unstable? At the core of this apparent contradiction is the assumption made in the derivation of (\ref{badboussinesq}) that only low-frequency modes are present. Indeed, the ill-posedness of the Boussinesq equation is a consequence of the fact that high-frequency Fourier modes grow or decay exponentially in time. The solutions of (\ref{badboussinesq}) only remain good approximations of water waves as long as these high-frequency modes are suppressed. On the other hand, the nonlinearity of (\ref{badboussinesq}) makes arguments based on Fourier modes too simplistic. To treat the fully nonlinear problem, a more sophisticated analysis that incorporates the nonlinearity is required. The inverse scattering transform can be viewed as the construction of a {\it nonlinear} Fourier transform, and it is therefore perfectly suited to address this problem. 

We intend to come back to the soliton resolution conjecture and to the stability of solitons for the Boussinesq equation in future publications.

The present paper is meant to be read in conjunction with \cite{CLmain}. Because of the similarities with \cite[Sections 3--6]{CLmain}, we will focus on novelties due to the presence of solitons, and we will omit proofs when they are identical (or very similar) to the corresponding proofs in \cite{CLmain}. 
Since the spectral problem associated to (\ref{badboussinesq}) is a third-order equation, the inclusion of solitons is not straightforward but involves several complications. 
%Of particular interest is the fact that the reality of the solutions of (\ref{badboussinesq}) not only imposes reality conditions on the residue constants associated to the solitons, but also restricts the possible pole structures in the RH problem. 
Some of these complications are related to the fact that the basic eigenfunctions of the Lax pair, which we denote by $X$ and $Y$, are not sufficient to construct a RH problem in the third-order case; it is also necessary to use another set of eigenfunctions, which we denote by $X^A$ and $Y^A$. In the case of compactly supported initial data, the eigenfunctions $X,Y,X^A,Y^A$ are defined in the whole complex plane except at isolated points, but in the general case the domains of definitions of their entries are much smaller. Much effort will be spent to handle these restricted domains of definition. In a few places, we will indicate how the arguments can be simplified in the case when the initial data has compact support.

Our main results are stated in Section \ref{mainsec}. The direct and inverse scattering problems for (\ref{badboussinesq}) are treated in Sections \ref{directsec} and \ref{inversesec}, respectively. In the appendix, which is an important part of this work, we consider the regularity properties of the one-soliton solutions and single breather solutions of (\ref{badboussinesq}).

\subsection{Notation}\label{notationsubsec}
We use the following notation throughout the paper.

\begin{enumerate}[$-$]
\item $C>0$ and $c>0$ denote generic constants that may change within a computation.

\item $[A]_1$, $[A]_2$, and $[A]_3$ denote the first, second, and third columns of a $3 \times 3$ matrix $A$.

%\item If $A$ is an $n \times m$ matrix, we define $|A| \ge 0$ by $|A|^2=\Sigma_{i,j}|A_{ij}|^2$. %Note that $|A + B| \leq |A| + |B|$ and $|AB| \leq |A| |B|$.
%For a piecewise smooth contour $\gamma \subset \C$ and $1 \le p \le \infty$, if $|A|$ belongs to $L^p(\gamma)$, we write $A \in L^p(\gamma)$ and define $\|A\|_{L^p(\gamma)} := \| |A|\|_{L^p(\gamma)}$.

\item $\D = \{k \in \C \, | \, |k| < 1\}$ denotes the open unit disk and $\partial \D = \{k \in \C \, | \, |k| = 1\}$ denotes the unit circle. 

%\item $f^*(k):= \overline{f(\bar{k})}$ denotes the Schwarz conjugate of a function $f(k)$.

\item $D_\epsilon(k)$ denotes the open disk of radius $\epsilon$ centered at a point $k \in \C$.

\item $\mathcal{S}(\R)$ denotes the Schwartz space of rapidly decreasing functions on $\R$.

\item $\kappa_{j} = e^{\frac{\pi i(j-1)}{3}}$, $j=1,\ldots,6$, denote the sixth roots of unity, see Figure \ref{fig: Dn}, and we let $\mathcal{Q} := \{\kappa_{j}\}_{j=1}^{6}$ and $\hat{\mathcal{Q}} := \mathcal{Q} \cup \{0\}$.

\item $D_n$, $n = 1, \dots, 6$, denote the open subsets of the complex plane shown in Figure \ref{fig: Dn}.

\item $\Gamma = \cup_{j=1}^9 \Gamma_j$ denotes the contour shown and oriented as in Figure \ref{fig: Dn}, and $\hat{\Gamma}_{j} = \Gamma_{j} \cup \partial \D$ denotes the union of $\Gamma_j$ and the unit circle. 

\end{enumerate}

\section{Main results}\label{mainsec}
Equation \eqref{badboussinesq} can be rewritten as the system
\begin{align}\label{boussinesqsystem}
& \begin{cases}
 v_{t} = u_{x} + (u^2)_{x} + u_{xxx},
 \\
 u_t = v_x.
\end{cases}
\end{align}
The initial value problems for \eqref{badboussinesq} and \eqref{boussinesqsystem} are equivalent, provided that the initial data $u_{1}(x):=u_{t}(x,0)$ satisfies the mass conservation condition $\int_{\R}u_{1}(x)dx=0$. Physically, this condition ensures that the total mass $\int_\R u dx$ is conserved in time.

\begin{figure}
\begin{center}
\begin{tikzpicture}[scale=0.9]
\node at (0,0) {};
\draw[black,line width=0.45 mm,->-=0.4,->-=0.85] (0,0)--(30:4);
\draw[black,line width=0.45 mm,->-=0.4,->-=0.85] (0,0)--(90:4);
\draw[black,line width=0.45 mm,->-=0.4,->-=0.85] (0,0)--(150:4);
\draw[black,line width=0.45 mm,->-=0.4,->-=0.85] (0,0)--(-30:4);
\draw[black,line width=0.45 mm,->-=0.4,->-=0.85] (0,0)--(-90:4);
\draw[black,line width=0.45 mm,->-=0.4,->-=0.85] (0,0)--(-150:4);

\draw[black,line width=0.45 mm] ([shift=(-180:2.5cm)]0,0) arc (-180:180:2.5cm);
\draw[black,arrows={-Triangle[length=0.2cm,width=0.18cm]}]
($(3:2.5)$) --  ++(90:0.001);
\draw[black,arrows={-Triangle[length=0.2cm,width=0.18cm]}]
($(57:2.5)$) --  ++(-30:0.001);
\draw[black,arrows={-Triangle[length=0.2cm,width=0.18cm]}]
($(123:2.5)$) --  ++(210:0.001);
\draw[black,arrows={-Triangle[length=0.2cm,width=0.18cm]}]
($(177:2.5)$) --  ++(90:0.001);
\draw[black,arrows={-Triangle[length=0.2cm,width=0.18cm]}]
($(243:2.5)$) --  ++(330:0.001);
\draw[black,arrows={-Triangle[length=0.2cm,width=0.18cm]}]
($(297:2.5)$) --  ++(210:0.001);

\draw[black,line width=0.15 mm] ([shift=(-30:0.55cm)]0,0) arc (-30:30:0.55cm);

\node at (0.8,0) {$\tiny \frac{\pi}{3}$};

\node at (-5:2.77) {\footnotesize $\Gamma_8$};
\node at (60:2.9) {\footnotesize $\Gamma_9$};
\node at (112:2.73) {\footnotesize $\Gamma_7$};
\node at (181:2.9) {\footnotesize $\Gamma_8$};
\node at (233:2.71) {\footnotesize $\Gamma_9$};
\node at (300:2.83) {\footnotesize $\Gamma_7$};

\node at (77:1.45) {\footnotesize $\Gamma_1$};
\node at (160:1.45) {\footnotesize $\Gamma_2$};
\node at (-163:1.45) {\footnotesize $\Gamma_3$};
\node at (-77:1.45) {\footnotesize $\Gamma_4$};
\node at (-42:1.45) {\footnotesize $\Gamma_5$};
\node at (43:1.45) {\footnotesize $\Gamma_6$};

\node at (84:3.3) {\footnotesize $\Gamma_4$};
\node at (155:3.3) {\footnotesize $\Gamma_5$};
\node at (-156:3.3) {\footnotesize $\Gamma_6$};
\node at (-84:3.3) {\footnotesize $\Gamma_1$};
\node at (-35:3.3) {\footnotesize $\Gamma_2$};
\node at (36:3.3) {\footnotesize $\Gamma_3$};

\end{tikzpicture}
\hspace{0.5cm}
\begin{tikzpicture}[scale=0.9]
\node at (0,0) {};
\draw[black,line width=0.45 mm] (0,0)--(30:4);
\draw[black,line width=0.45 mm] (0,0)--(90:4);
\draw[black,line width=0.45 mm] (0,0)--(150:4);
\draw[black,line width=0.45 mm] (0,0)--(-30:4);
\draw[black,line width=0.45 mm] (0,0)--(-90:4);
\draw[black,line width=0.45 mm] (0,0)--(-150:4);

\draw[black,line width=0.45 mm] ([shift=(-180:2.5cm)]0,0) arc (-180:180:2.5cm);
\draw[black,line width=0.15 mm] ([shift=(-30:0.55cm)]0,0) arc (-30:30:0.55cm);

\node at (120:1.6) {\footnotesize{$D_{1}$}};
\node at (-60:3.7) {\footnotesize{$D_{1}$}};

\node at (180:1.6) {\footnotesize{$D_{2}$}};
\node at (0:3.7) {\footnotesize{$D_{2}$}};

\node at (240:1.6) {\footnotesize{$D_{3}$}};
\node at (60:3.7) {\footnotesize{$D_{3}$}};

\node at (-60:1.6) {\footnotesize{$D_{4}$}};
\node at (120:3.7) {\footnotesize{$D_{4}$}};

\node at (0:1.6) {\footnotesize{$D_{5}$}};
\node at (180:3.7) {\footnotesize{$D_{5}$}};

\node at (60:1.6) {\footnotesize{$D_{6}$}};
\node at (-120:3.7) {\footnotesize{$D_{6}$}};

\node at (0.8,0) {$\tiny \frac{\pi}{3}$};

\draw[fill] (0:2.5) circle (0.1);
\draw[fill] (60:2.5) circle (0.1);
\draw[fill] (120:2.5) circle (0.1);
\draw[fill] (180:2.5) circle (0.1);
\draw[fill] (240:2.5) circle (0.1);
\draw[fill] (300:2.5) circle (0.1);

\node at (0:2.9) {\footnotesize{$\kappa_1$}};
\node at (60:2.85) {\footnotesize{$\kappa_2$}};
\node at (120:2.85) {\footnotesize{$\kappa_3$}};
\node at (180:2.9) {\footnotesize{$\kappa_4$}};
\node at (240:2.85) {\footnotesize{$\kappa_5$}};
\node at (300:2.85) {\footnotesize{$\kappa_6$}};

\draw[dashed] (-4.3,-3.8)--(-4.3,3.8);

\end{tikzpicture}
\end{center}
\begin{figuretext}\label{fig: Dn}
The contour $\Gamma = \cup_{j=1}^9 \Gamma_j$ in the complex $k$-plane (left) and the open sets $D_{n}$, $n=1,\ldots,6$, together with the sixth roots of unity $\kappa_j$, $j = 1, \dots, 6$ (right).
\end{figuretext}
\end{figure}

\subsection{The direct problem}

Let $\omega := e^{\frac{2\pi i}{3}}$ and define $\{l_j(k), z_j(k)\}_{j=1}^3$ by
\begin{align}\label{lmexpressions intro}
& l_{j}(k) = i \frac{\omega^{j}k + (\omega^{j}k)^{-1}}{2\sqrt{3}}, \qquad z_{j}(k) = i \frac{(\omega^{j}k)^{2} + (\omega^{j}k)^{-2}}{4\sqrt{3}}, \qquad k \in \C\setminus \{0\}.
\end{align}
Let $P(k)$ and $\mathsf{U}(x,k)$ be given by
\begin{align*}
P(k) = \begin{pmatrix}
1 & 1 & 1  \\
l_{1}(k) & l_{2}(k) & l_{3}(k) \\
l_{1}(k)^{2} & l_{2}(k)^{2} & l_{3}(k)^{2}
\end{pmatrix}, \quad \mathsf{U}(x,k) = P(k)^{-1} \begin{pmatrix}
0 & 0 & 0 \\
0 & 0 & 0 \\
-\frac{u_{0x}}{4}-\frac{iv_{0}}{4\sqrt{3}} & -\frac{u_{0}}{2} & 0
\end{pmatrix} P(k),
\end{align*} 
where $v_{0}(x) = \int_{-\infty}^{x}u_{1}(x')dx'$. Let $X(x,k), X^A(x,k), Y(x,k)$, and $Y^A(x,k)$ be the unique solutions of the Volterra integral equations
\begin{align}  \nonumber
& X(x,k)  =  I  -  \int_x^{\infty}  e^{(x-x')\widehat{\mathcal{L}(k)}} (\mathsf{U}X)(x',k)dx', 
	\\\nonumber
& X^A(x,k)  =  I  +  \int_x^{\infty}  e^{-(x-x')\widehat{\mathcal{L}(k)}} (\mathsf{U}^T X^A)(x',k) dx',
	 \\\nonumber
& Y(x,k)  =  I  +  \int_{-\infty}^x  e^{(x-x')\widehat{\mathcal{L}(k)}} (\mathsf{U}Y)(x',k) dx', 
	\\ \label{def of X XA Y YA}
& Y^A(x,k)  =  I  -  \int_{-\infty}^x  e^{-(x-x')\widehat{\mathcal{L}(k)}} (\mathsf{U}^T Y^A)(x',k) dx', 
\end{align}
where $\mathcal{L} = \diag(l_1 , l_2 , l_3)$, $e^{\hat{\mathcal{L}}}$ is the operator which acts
on a $3\times 3$ matrix $M$ by $e^{\hat{\mathcal{L}}}M=e^{\mathcal{L}}Me^{-\mathcal{L}}$, and $\mathsf{U}^T$ is the transpose of $\mathsf{U}$. Define the scattering matrices $s(k)$ and $s^A(k)$ by 
\begin{align}\label{def of s sA}
& s(k) = I - \int_\R e^{-x \widehat{\mathcal{L}(k)}}(\mathsf{U}X)(x,k)dx, & & s^A(k) = I + \int_\R e^{x \widehat{\mathcal{L}(k)}}(\mathsf{U}^T X^A)(x,k)dx.
\end{align}
The scattering data for the Boussinesq equation includes two spectral functions $r_{1}(k)$ and $r_{2}(k)$ which are defined by
\begin{align}\label{r1r2def}
\begin{cases}
r_1(k) = \frac{(s(k))_{12}}{(s(k))_{11}}, & k \in \hat{\Gamma}_{1}\setminus \hat{\mathcal{Q}},
	\\ 
r_2(k) = \frac{(s^A(k))_{12}}{(s^A(k))_{11}}, \quad & k \in \hat{\Gamma}_{4}\setminus \hat{\mathcal{Q}},
\end{cases}
\end{align}	
where the set $\hat{\mathcal{Q}}$ and the contours $\hat{\Gamma}_j = \Gamma_{j} \cup \partial \D$ were defined in Section \ref{notationsubsec}.

\begin{figure}
%\vspace{-0.2cm}
\begin{center}
\begin{tikzpicture}[master, scale=0.6]
\node at (0,0) {};
\draw[black,line width=0.45 mm] (30:2.5)--(30:4);
\draw[black,line width=0.45 mm] (0,0)--(90:2.5);
\draw[black,line width=0.45 mm] (150:2.5)--(150:4);
\draw[black,line width=0.45 mm] (0,0)--(-30:2.5);
\draw[black,line width=0.45 mm] (-90:2.5)--(-90:4);
\draw[black,line width=0.45 mm] (0,0)--(-150:2.5);

\draw[black,line width=0.45 mm] ([shift=(-180:2.5cm)]0,0) arc (-180:180:2.5cm);

\node at (30:1.4) {$\omega \mathcal{S}$};
\node at (90:3.3) {$\mathcal{S}$};
\node at (150:1.4) {$\omega^{2}\mathcal{S}$};
\node at (210:3.6) {$\omega \mathcal{S}$};
\node at (270:1.4) {$\mathcal{S}$};
\node at (330:3.6) {$\omega^{2} \mathcal{S}$};

\draw[fill] (0:2.5) circle (0.1);
\draw[fill] (60:2.5) circle (0.1);
\draw[fill] (120:2.5) circle (0.1);
\draw[fill] (180:2.5) circle (0.1);
\draw[fill] (240:2.5) circle (0.1);
\draw[fill] (300:2.5) circle (0.1);

\end{tikzpicture} \hspace{2.5cm} \begin{tikzpicture}[slave, scale=0.6]
\node at (0,0) {};
\draw[black,line width=0.45 mm] (0,0)--(30:2.5);
\draw[black,line width=0.45 mm] (90:2.5)--(90:4);
\draw[black,line width=0.45 mm] (0,0)--(150:2.5);
\draw[black,line width=0.45 mm] (-30:2.5)--(-30:4);
\draw[black,line width=0.45 mm] (0,0)--(-90:2.5);
\draw[black,line width=0.45 mm] (-150:2.5)--(-150:4);

\draw[black,line width=0.45 mm] ([shift=(-180:2.5cm)]0,0) arc (-180:180:2.5cm);

\node at (30:3.6) {$-\omega \mathcal{S}$};
\node at (90:1.4) {$-\mathcal{S}$};
\node at (150:3.8) {$-\omega^{2}\mathcal{S}$};
\node at (210:1.4) {$-\omega \mathcal{S}$};
\node at (270:3.2) {$-\mathcal{S}$};
\node at (330:1.4) {$-\omega^{2} \mathcal{S}$};

\draw[fill] (0:2.5) circle (0.1);
\draw[fill] (60:2.5) circle (0.1);
\draw[fill] (120:2.5) circle (0.1);
\draw[fill] (180:2.5) circle (0.1);
\draw[fill] (240:2.5) circle (0.1);
\draw[fill] (300:2.5) circle (0.1);

\draw[dashed] (-5.7,-3.8)--(-5.7,3.8);
\end{tikzpicture}
\end{center}
\begin{figuretext}\label{fig: S}The regions $\{\pm \omega^{j}\mathcal{S}\}_{j=0,1,2}$. The dots represent $\kappa_{j}$, $j=1,\ldots,6$.
\end{figuretext}
\end{figure}

\subsubsection{Solitons}
Solitons correspond to zeros of the functions $s_{11}(k)$ and $s^A_{11}(k)$ appearing in the denominators in (\ref{r1r2def}).
It follows from \cite[Propositions 3.5 and 3.9]{CLmain} that $s_{11}$ and $s^A_{11}$ have smooth extensions to $\omega^{2}\hat{\mathcal{S}}\setminus \hat{\mathcal{Q}}$ and $-\omega^{2}\hat{\mathcal{S}}\setminus \hat{\mathcal{Q}}$, respectively, where $\hat{\mathcal{S}} := \partial \D \cup \bar{\mathcal{S}}$ and $\mathcal{S}$ is the interior of $\bar{D}_{3}\cup\bar{D}_{4}$ (see Figure \ref{fig: S}).
For simplicity, we will restrict ourselves to the generic case where the zeros of $s_{11}$ and $s^A_{11}$ do not lie on the contour $\Gamma$. The symmetries $s_{11}(k) = s_{11}(\omega/k)$ and $s_{11}^A(k) = \overline{s_{11}(\bar{k}^{-1})}$ (which follow from \cite[Propositions 3.5 and 3.9]{CLmain}) then imply that it is enough to consider the zeros of $s_{11}$ in $D_2$. 

We split the open set $D_{2}$ into three parts: $D_{2} = D_{\mathrm{reg}} \sqcup D_{\mathrm{sing}} \sqcup (D_{2}\cap \R)$, where $D_{2}\cap \R = (-1,0) \cup (1,\infty)$ and
\begin{align*}
& D_{\mathrm{reg}} := D_{2} \cap \big( \{k \,|\, |k|>1, \im k >0\} \cup \{k \,|\, |k|<1, \im k <0\} \big), \\
& D_{\mathrm{sing}} := D_{2} \cap \big( \{k \,|\, |k|>1, \im k <0\} \cup \{k \,|\, |k|<1, \im k >0\} \big).
\end{align*}
Furthermore, we denote the right and left parts of $D_{\mathrm{reg}}$ by $D_{\mathrm{reg}}^{R}$ and $D_{\mathrm{reg}}^{L}$, respectively, and similarly for $D_{\mathrm{sing}}$, see Figure \ref{fig: D2 splitting}. It turns out that 
\begin{enumerate}[$-$]
\item simple zeros of $s_{11}(k)$ in $D_{\mathrm{reg}}^{R}$ correspond to right-moving breather solitons,
\item simple zeros of $s_{11}(k)$ in $D_{\mathrm{reg}}^{L}$ correspond to left-moving breather solitons,
\item simple zeros of $s_{11}(k)$ in $(1,\infty)$ correspond to right-moving solitons, 
\item simple zeros of $s_{11}(k)$ in $(-1,0)$ correspond to left-moving solitons.
\end{enumerate}

One may also consider the solitons of (\ref{badboussinesq}) generated by zeros of $s_{11}$ in $D_{\mathrm{sing}}$; however, we will show in the appendix (see Lemma \ref{breatherlemma}) that the solitons generated by such zeros are singular. More precisely, simple zeros of $s_{11}(k)$ in $D_{\mathrm{sing}}^{R}$ and $D_{\mathrm{sing}}^{L}$ correspond to singular right- and left-moving breathers, respectively, where singular means that the solution is not smooth as a function of $x \in \R$ but has poles on the real axis when viewed as a function of $x$ for any fixed $t$. 

In fact, some of the solitons generated by zeros of $s_{11}(k)$ in $(-1,0) \cup (1,\infty)$ are also singular. In the appendix (see Lemma \ref{onesolitonsingularlemma}), we will derive a condition that exactly characterizes the zeros that give rise to non-singular solitons.

\subsubsection{Scattering data}
In addition to the functions $r_1$  and $r_2$, the scattering data includes the set $\mathsf{Z}$ of zeros of $s_{11}$ and a set of corresponding residue constants $\{c_{k_0}\}_{k_0 \in \mathsf{Z}} \subset \C$. In the case of compactly supported initial data $u_0$ and $v_0$, the residue constants are defined by
 \begin{align}\label{ck0compactdef}
& c_{k_{0}} := \begin{cases} - \frac{s_{13}(k_0)}{\dot{s}_{11}(k_0)}, & k_{0} \in \mathsf{Z}\setminus \mathbb{R}, \\
- \frac{s_{12}(k_0)}{\dot{s}_{11}(k_0)}, & k_{0} \in \mathsf{Z}\cap \mathbb{R}.
\end{cases}
\end{align}

If $u_0$ and $v_0$ do not have compact support, then the entries of the matrix-valued function $s(k)$ have limited domains of definition so the quantities $s_{13}(k_0)$ and $s_{12}(k_0)$ in (\ref{ck0compactdef}) are in general not well-defined. Thus the following more complicated definition of the  $c_{k_{0}}$ is required in the general case: Define the vector-valued function $w(x,k)$ by
\begin{align}\label{wdef}
w = \begin{pmatrix}
Y_{21}^AX_{32}^A - Y_{31}^AX_{22}^A  \\
Y_{31}^AX_{12}^A - Y_{11}^AX_{32}^A  \\
 Y_{11}^AX_{22}^A - Y_{21}^AX_{12}^A 
\end{pmatrix}.
\end{align}
If $k_{0} \in \mathsf{Z}\setminus \mathbb{R}$, then $c_{k_0}$ is defined as the unique complex constant such that
\begin{subequations}\label{ck0def}
\begin{align}\label{ck0def1}
\frac{w(x,k_{0})}{\dot{s}_{11}(k_0)} = c_{k_0} e^{(l_1(k_0)-l_3(k_0))x} [X(x,k_0)]_1 \qquad \text{for all $x \in \R$}.
\end{align}
If $k_{0} \in \mathsf{Z} \cap \mathbb{R}$, then $c_{k_0}$ is defined as the unique complex constant such that
\begin{align}\label{ck0def2}
\frac{[Y(x,k_0)]_2}{\dot{s}_{22}^A(k_0)} = c_{k_0} e^{(l_1(k_0)-l_2(k_0))x} [X(x,k_0)]_1\qquad \text{for all $x \in \R$}.
\end{align}
\end{subequations}
We will show in Theorem \ref{directth} that the $c_{k_0}$ are indeed well-defined by these relations.

\begin{figure}
\begin{center}
\begin{tikzpicture}[scale=0.7]
\node at (0,0) {};
\draw[line width=0.25 mm, dashed,fill=gray!30] (0,0) --  (210:2.5) arc(210:180:2.5) -- cycle;
\draw[line width=0.25 mm, dashed,fill=gray!30] (0,0) --  (30:4) arc(30:0:4) -- cycle;
\draw[line width=0.25 mm, dashed,fill=white!100] (0,0) --  (30:2.5) arc(30:0:2.5) -- cycle;
\draw[white,line width=0.35 mm] ([shift=(0:4cm)]0,0) arc (0:30:4cm);
\draw[white,line width=0.35 mm] (0,0)--(2.5,0);

\draw[black,line width=0.45 mm] (0,0)--(30:4);
\draw[black,line width=0.45 mm] (0,0)--(90:4);
\draw[black,line width=0.45 mm] (0,0)--(150:4);
\draw[black,line width=0.45 mm] (0,0)--(-30:4);
\draw[black,line width=0.45 mm] (0,0)--(-90:4);
\draw[black,line width=0.45 mm] (0,0)--(-150:4);

\draw[black,line width=0.45 mm] ([shift=(-180:2.5cm)]0,0) arc (-180:180:2.5cm);

\node at (193:1.85) {\footnotesize{$D_{\mathrm{reg}}^{L}$}};
\node at (15:3.3) {\footnotesize{$D_{\mathrm{reg}}^{R}$}};

\draw[fill] (0:2.5) circle (0.1);
\draw[fill] (60:2.5) circle (0.1);
\draw[fill] (120:2.5) circle (0.1);
\draw[fill] (180:2.5) circle (0.1);
\draw[fill] (240:2.5) circle (0.1);
\draw[fill] (300:2.5) circle (0.1);

\end{tikzpicture}
\hspace{1.7cm}
\begin{tikzpicture}[scale=0.7]
\node at (0,0) {};
\draw[line width=0.25 mm, dashed,fill=gray!30] (0,0) --  (180:2.5) arc(180:150:2.5) -- cycle;
\draw[line width=0.25 mm, dashed,fill=gray!30] (0,0) --  (0:4) arc(0:-30:4) -- cycle;
\draw[line width=0.25 mm, dashed,fill=white!100] (0,0) --  (0:2.5) arc(0:-30:2.5) -- cycle;
\draw[white,line width=0.35 mm] ([shift=(-30:4cm)]0,0) arc (-30:0:4cm);
\draw[white,line width=0.35 mm] (0,0)--(2.5,0);

\draw[black,line width=0.45 mm] (0,0)--(30:4);
\draw[black,line width=0.45 mm] (0,0)--(90:4);
\draw[black,line width=0.45 mm] (0,0)--(150:4);
\draw[black,line width=0.45 mm] (0,0)--(-30:4);
\draw[black,line width=0.45 mm] (0,0)--(-90:4);
\draw[black,line width=0.45 mm] (0,0)--(-150:4);

\draw[black,line width=0.45 mm] ([shift=(-180:2.5cm)]0,0) arc (-180:180:2.5cm);

\node at (165:1.8) {\footnotesize{$D_{\mathrm{sing}}^{L}$}};
\node at (-15:3.3) {\footnotesize{$D_{\mathrm{sing}}^{R}$}};

\draw[fill] (0:2.5) circle (0.1);
\draw[fill] (60:2.5) circle (0.1);
\draw[fill] (120:2.5) circle (0.1);
\draw[fill] (180:2.5) circle (0.1);
\draw[fill] (240:2.5) circle (0.1);
\draw[fill] (300:2.5) circle (0.1);

\draw[dashed] (-6.3,-3.8)--(-6.3,3.8);

\end{tikzpicture}
\end{center}
\begin{figuretext}\label{fig: D2 splitting}The open regions $D_{\mathrm{reg}}^{R}, D_{\mathrm{reg}}^{L}, D_{\mathrm{sing}}^{R}$, and $D_{\mathrm{sing}}^{L}$.
\end{figuretext}
\end{figure}

\subsubsection{Assumptions}
As mentioned above, we will show in the appendix that the breather solitons generated by zeros of $s_{11}$ in $D_{\mathrm{sing}}$ are singular. We will also show that a zero $k_0$ of $s_{11}(k)$ in $(-1,0) \cup (1,\infty)$ gives rise to a non-singular soliton if and only if the associated residue constant $c_{k_0} \in \C$ obeys the condition $i(\omega^{2}k_{0}^{2} - \omega)c_{k_{0}} \notin (-\infty, 0)$, see Lemma \ref{onesolitonsingularlemma}.

We will restrict ourselves to the case when $s_{11}(k)$ has a finite number of simple zeros corresponding to non-singular solitons. 
More precisely, our results will be valid under the following assumption.

\begin{assumption}[Solitons]\label{solitonassumption}\upshape
Assume that $s_{11}(k)$ has a simple zero at each point in $\mathsf{Z}$, where $\mathsf{Z} \subset D_{\mathrm{reg}} \cup (-1,0) \cup (1,\infty)$ is a set of finite cardinality, and that $s_{11}(k)$ has no other zeros in $k \in (\bar{D}_2 \cup \partial \D) \setminus \hat{\mathcal{Q}}$.
If $k_0 \in \mathsf{Z} \setminus \R$, then we also assume that $s_{22}^A(k_0) \neq 0$. If $k_0 \in \mathsf{Z} \cap \R$, then we also assume that $i(\omega^{2}k_{0}^{2} - \omega)c_{k_{0}} \notin (-\infty, 0)$.
\end{assumption}

The symmetries $s_{11}(k) = s_{11}(\omega/k)$ and $s_{11}^A(k) = \overline{s_{11}(\bar{k}^{-1})}$ imply that if Assumption \ref{solitonassumption} holds then $s_{11}$ is nonzero on $\omega^{2}\hat{\mathcal{S}} \setminus (\hat{\mathcal{Q}}\cup \mathsf{Z} \cup \omega \mathsf{Z}^{-1})$ and $s_{11}^A$ is nonzero on $-\omega^{2}\hat{\mathcal{S}} \setminus ( \hat{\mathcal{Q}} \cup \mathsf{Z}^{-*}\cup \omega \mathsf{Z}^{*})$, where $\mathsf{Z}^{-1}$ is the image of $\mathsf{Z}$ under the map $k\mapsto k^{-1}$, $\mathsf{Z}^{*}$ is the image of $\mathsf{Z}$ under the map $k\mapsto \overline{k}$, and $\mathsf{Z}^{-*}:= (\mathsf{Z}^{*})^{-1}$. 

We further assume that the spectral functions $s(k), s^{A}(k)$ have generic behavior at $k = \pm 1$.

\begin{assumption}[Generic behavior at $k = \pm 1$]\label{originassumption}\upshape
Assume for $k_{\star} =1$ and $k_{\star}=-1$ that
\begin{align*}
& \lim_{k \to k_{\star}} (k-k_{\star}) s(k)_{11} \neq 0, & & \hspace{-0.05cm} \lim_{k \to k_{\star}} (k-k_{\star}) s(k)_{13} \neq 0, & & \hspace{-0.05cm} \lim_{k \to k_{\star}} s(k)_{31} \neq 0, & & \hspace{-0.05cm} \lim_{k \to k_{\star}} s(k)_{33} \neq 0, \\
& \lim_{k \to k_{\star}} (k-k_{\star}) s^{A}(k)_{11} \neq 0, & & \hspace{-0.05cm} \lim_{k \to k_{\star}} (k-k_{\star}) s^{A}(k)_{31} \neq 0, & & \hspace{-0.05cm} \lim_{k \to k_{\star}} s^{A}(k)_{13} \neq 0, & & \hspace{-0.05cm} \lim_{k \to k_{\star}} s^{A}(k)_{33} \neq 0.
\end{align*}
\end{assumption}

\subsubsection{Statement of the theorem}
The following is our main result on the direct scattering problem for (\ref{badboussinesq}).

\begin{theorem}[Direct scattering]\label{directth}
Suppose $u_0,v_0 \in \mathcal{S}(\R)$ are two real-valued functions such that Assumptions \ref{solitonassumption} and \ref{originassumption} hold.
Then the associated scattering data 
$$\{r_1(k), r_2(k), \mathsf{Z}, \{c_{k_0}\}_{k_0 \in \mathsf{Z}}\}$$ 
defined by (\ref{r1r2def}) and (\ref{ck0def}) are well-defined and satisfy the following properties:
\begin{enumerate}[$(i)$]
 \item \label{Theorem2.3itemi}
 $r_1$ and $r_2$ admit extensions such that $r_1 \in C^\infty(\hat{\Gamma}_{1})$\footnote{At $k_\star = \pm i$, the notation $r_1 \in C^\infty(\hat{\Gamma}_{1})$ indicates that $r_1$ has derivatives to all orders along each subcontour of $\hat{\Gamma}_1$ emanating from $k_\star$ and that these derivatives have consistent finite limits as $k \to k_\star$.}
and $r_2 \in C^\infty(\hat{\Gamma}_{4}\setminus \{\omega^{2}, -\omega^{2}\})$.

\item \label{Theorem2.3itemii}
$r_{1}(k)$ is bounded on the whole unit circle, and $r_{1}(\kappa_{j})\neq 0$ for $j=1,\ldots,6$. $r_{2}(k)$ has simple poles at $k=\omega^2$ and $k = -\omega^2$, and simple zeros at $k=\omega$ and $k=-\omega$. Furthermore,
\begin{align}\label{r1r2at0}
r_{1}(1) = r_{1}(-1) = 1, \qquad r_{2}(1) = r_{2}(-1) = -1.
\end{align}

\item \label{Theorem2.3itemiii}
$r_1(k)$ and $r_2(k)$ are rapidly decreasing as $|k| \to \infty$, i.e., for each integer $N \geq 0$,
\begin{align}\label{r1r2rapiddecay}
& \max_{j=0,1,\dots,N}\sup_{k \in \Gamma_{1}} (1+|k|)^N |\partial_k^jr_1(k)| < \infty,  
\qquad
 \max_{j=0,1,\dots,N} \sup_{k \in \Gamma_{4}} (1+|k|)^N|\partial_k^jr_2(k)| < \infty.
\end{align}

\item \label{Theorem2.3itemiv}
For all $k \in \partial \D \setminus \{-\omega,\omega\}$, we have
\begin{align}\label{r1r2 relation on the unit circle}
r_{1}(\tfrac{1}{\omega k}) + r_{2}(\omega k) + r_{1}(\omega^{2} k) r_{2}(\tfrac{1}{k}) = 0.
\end{align}
In fact, \eqref{r1r2 relation on the unit circle} is also equivalent to any of the following two relations:
\begin{align}\label{r1r2 relation on the unit circle new}
r_{2}(k) = \frac{r_{1}( \omega k)r_{1}(\omega^{2} k)-r_{1}(\frac{1}{k})}{1-r_{1}(\omega k)r_{1}(\frac{1}{\omega k})}, \qquad r_{1}(k) = \frac{r_{2}(\omega k)r_{2}(\omega^{2}k)-r_{2}(\frac{1}{k})}{1-r_{2}(\omega k)r_{2}(\frac{1}{\omega k})}.
\end{align}

\item \label{Theorem2.3itemv}
$r_1$ and $r_2$ are related by the symmetry
\begin{align}\label{r1r2 relation with kbar symmetry}
%& r_{1}(k) = \frac{k^{2}-\omega}{\omega k^{2}-1}\overline{r_{2}\big(\bar{k}^{-1}\big)}, & & \mbox{for all } k \in \hat{\Gamma}_{1}\setminus \hat{\mathcal{Q}}, \label{lol1} \\
& r_{2}(k) = \tilde{r}(k) \overline{r_{1}(\bar{k}^{-1})} & & \mbox{for all } k \in \hat{\Gamma}_{4}\setminus \{0,\omega^{2}, -\omega^{2}\}, % \label{lol2} 
\end{align}
where
\begin{align}\label{rtildedef}
\tilde{r}(k):=\frac{\omega^{2}-k^{2}}{1-\omega^{2}k^{2}} & & \mbox{for } k \in \mathbb{C}\setminus \{\omega^{2},-\omega^{2}\}.
\end{align}

\item \label{Theorem2.3itemvi}
$\mathsf{Z}$ is a finite subset of $D_{\mathrm{reg}} \cup (-1,0) \cup (1,\infty)$ and $\{c_{k_0}\}_{k_0 \in \mathsf{Z}} \subset \C$.

\item \label{Theorem2.3itemvii}
For every $k_{0} \in \mathsf{Z}\cap \mathbb{R}$, the residue constant $c_{k_0}$ satisfies the positivity condition
\begin{align}\label{ck0positivitycondition}
i(\omega^{2}k_{0}^{2} - \omega)c_{k_{0}} \geq 0.
\end{align}

\end{enumerate} 
\end{theorem}

The proof of Theorem \ref{directth} is presented in Section \ref{directproofsubsec}.

\subsection{The inverse problem}
We now consider the inverse problem of reconstructing the solution $\{u(x,t), v(x,t)\}$ from the scattering data. Our solution of the inverse problem involves a row-vector RH problem, whose jump matrix is defined on a contour $\Gamma$ in the complex $k$-plane and is expressed in terms of $r_1(k)$ and $r_2(k)$. The solution $n(x,t,k)$ of this RH problem has poles at the points in $\hat{\mathsf{Z}}$, where
\begin{align*}
\hat{\mathsf{Z}} := \mathsf{Z} \cup \mathsf{Z}^{*} \cup \omega \mathsf{Z} \cup \omega \mathsf{Z}^{*} \cup \omega^{2} \mathsf{Z} \cup \omega^{2} \mathsf{Z}^{*} \cup \mathsf{Z}^{-1} \cup \mathsf{Z}^{-*} \cup \omega \mathsf{Z}^{-1} \cup \omega \mathsf{Z}^{-*} \cup \omega^{2} \mathsf{Z}^{-1} \cup \omega^{2} \mathsf{Z}^{-*},
\end{align*}
with $\mathsf{Z}^{-1}=\{k_{0}^{-1} \,|\, k_{0}\in \mathsf{Z}\}$, $\mathsf{Z}^{*}=\{\bar{k}_{0} \,|\, k_{0}\in \mathsf{Z}\}$, and $\mathsf{Z}^{-*}=\{\bar{k}_{0}^{-1} \,|\, k_{0}\in \mathsf{Z}\}$. The residue conditions for $n$ at the points in $\hat{\mathsf{Z}}$ are expressed in terms of the residue constants $\{c_{k_0}\}_{k_0 \in \mathsf{Z}}$. 

To state the RH problem for $n$, let $\theta_{ij}(x,t, k)$ for $1 \leq j<i \leq 3$ be complex-valued functions given by
\begin{align}\label{def of Phi ij}
\theta_{ij}(x,t,k) = (l_{i}(k)-l_{j}(k))x + (z_{i}(k)-z_{j}(k))t,
\end{align}
where $\{l_j(k), z_j(k)\}_{j=1}^3$ are defined by (\ref{lmexpressions intro}). The jump matrix $v(x,t,k)$ is defined for $k \in \Gamma$ by
\begin{align}\nonumber
&  v_1 = %J_{12,1}
  \begin{pmatrix}  
 1 & - r_1(k)e^{-\theta_{21}} & 0 \\
  r_1(\frac{1}{k})e^{\theta_{21}} & 1 - r_1(k)r_1(\frac{1}{k}) & 0 \\
  0 & 0 & 1
  \end{pmatrix},
	\quad  v_2 = %J_{2,3}
  \begin{pmatrix}   
 1 & 0 & 0 \\
 0 & 1 - r_2(\omega k)r_2(\frac{1}{\omega k}) & -r_2(\frac{1}{\omega k})e^{-\theta_{32}} \\
 0 & r_2(\omega k)e^{\theta_{32}} & 1 
    \end{pmatrix},
   	\\ \nonumber
  &v_3 = %J_{4,5}
  \begin{pmatrix} 
 1 - r_1(\omega^2 k)r_1(\frac{1}{\omega^2 k}) & 0 & r_1(\frac{1}{\omega^2 k})e^{-\theta_{31}} \\
 0 & 1 & 0 \\
 -r_1(\omega^2 k)e^{\theta_{31}} & 0 & 1  
  \end{pmatrix},
	\quad   v_4 = %J_{6,7}
  \begin{pmatrix}  
  1 - r_2(k)r_{2}(\frac{1}{k}) & -r_2(\frac{1}{k}) e^{-\theta_{21}} & 0 \\
  r_2(k)e^{\theta_{21}} & 1 & 0 \\
  0 & 0 & 1
   \end{pmatrix},
   	\\ \nonumber
&  v_5 = %J_{8,9}
  \begin{pmatrix}
  1 & 0 & 0 \\
  0 & 1 & -r_1(\omega k)e^{-\theta_{32}} \\
  0 & r_1(\frac{1}{\omega k})e^{\theta_{32}} & 1 - r_1(\omega k)r_1(\frac{1}{\omega k}) 
  \end{pmatrix},
	\quad   v_6 = %J_{10,11}
  \begin{pmatrix} 
  1 & 0 & r_2(\omega^2 k)e^{-\theta_{31}} \\
  0 & 1 & 0 \\
  -r_2(\frac{1}{\omega^2 k})e^{\theta_{31}} & 0 & 1 - r_2(\omega^2 k)r_2(\frac{1}{\omega^2 k})
   \end{pmatrix}, \nonumber \\
& v_{7} = \begin{pmatrix}
1 & -r_{1}(k)e^{-\theta_{21}} & r_{2}(\omega^{2}k)e^{-\theta_{31}} \\
-r_{2}(k)e^{\theta_{21}} & 1+r_{1}(k)r_{2}(k) & \big(r_{2}(\frac{1}{\omega k})-r_{2}(k)r_{2}(\omega^{2}k)\big)e^{-\theta_{32}} \\
r_{1}(\omega^{2}k)e^{\theta_{31}} & \big(r_{1}(\frac{1}{\omega k})-r_{1}(k)r_{1}(\omega^{2}k)\big)e^{\theta_{32}} & f(\omega^{2}k)
\end{pmatrix}, \nonumber \\
& v_{8} = \begin{pmatrix}
f(k) & r_{1}(k)e^{-\theta_{21}} & \big(r_{1}(\frac{1}{\omega^{2} k})-r_{1}(k)r_{1}(\omega k)\big)e^{-\theta_{31}} \\
r_{2}(k)e^{\theta_{21}} & 1 & -r_{1}(\omega k) e^{-\theta_{32}} \\
\big( r_{2}(\frac{1}{\omega^{2}k})-r_{2}(\omega k)r_{2}(k) \big)e^{\theta_{31}} & -r_{2}(\omega k) e^{\theta_{32}} & 1+r_{1}(\omega k)r_{2}(\omega k)
\end{pmatrix}, \nonumber \\
& v_{9} = \begin{pmatrix}
1+r_{1}(\omega^{2}k)r_{2}(\omega^{2}k) & \big( r_{2}(\frac{1}{k})-r_{2}(\omega k)r_{2}(\omega^{2} k) \big)e^{-\theta_{21}} & -r_{2}(\omega^{2}k)e^{-\theta_{31}} \\
\big(r_{1}(\frac{1}{k})-r_{1}(\omega k) r_{1}(\omega^{2} k)\big)e^{\theta_{21}} & f(\omega k) & r_{1}(\omega k)e^{-\theta_{32}} \\
-r_{1}(\omega^{2}k)e^{\theta_{31}} & r_{2}(\omega k) e^{\theta_{32}} & 1
\end{pmatrix}, \label{vdef}
\end{align}
where $v_j$ denotes the restriction of $v$ to $\Gamma_{j}$, and 
\begin{align}\label{def of f}
f(k) := 1+r_{1}(k)r_{2}(k) + r_{1}(\tfrac{1}{\omega^{2}k})r_{2}(\tfrac{1}{\omega^{2}k}), \qquad k \in \partial \mathbb{D}.
%= (1-r_{1}(\omega^{2}k)r_{1}(\tfrac{1}{\omega^{2}k}))(1-r_{2}(k)r_{2}(\tfrac{1}{k})).
\end{align}
Let $\Gamma_\star = \{i\kappa_j\}_{j=1}^6 \cup \{0\}$ be the set of intersection points of $\Gamma$.

\begin{RHproblem}[RH problem for $n$]\label{RHn}
Find a $1 \times 3$-row-vector valued function $n(x,t,k)$ with the following properties:
\begin{enumerate}[$(a)$]
\item\label{RHnitema} $n(x,t,\cdot) : \C \setminus (\Gamma\cup \hat{\mathsf{Z}}) \to \mathbb{C}^{1 \times 3}$ is analytic.

\item\label{RHnitemb} The limits of $n(x,t,k)$ as $k$ approaches $\Gamma \setminus \Gamma_\star$ from the left and right exist, are continuous on $\Gamma \setminus \Gamma_\star$, and are denoted by $n_+$ and $n_-$, respectively. Furthermore, they are related by
\begin{align}\label{njump}
  n_+(x,t,k) = n_-(x, t, k) v(x, t, k) \qquad \text{for} \quad k \in \Gamma \setminus \Gamma_\star.
\end{align}

\item\label{RHnitemc} $n(x,t,k) = O(1)$ as $k \to k_{\star} \in \Gamma_\star$.

\item\label{RHnitemd} For $k \in \C \setminus (\Gamma \cup \hat{\mathsf{Z}})$, $n$ obeys the symmetries
\begin{align}\label{nsymm}
n(x,t,k) = n(x,t,\omega k)\mathcal{A}^{-1} = n(x,t,k^{-1}) \mathcal{B},
\end{align}
where $\mathcal{A}$ and $\mathcal{B}$ are the matrices defined by 
\begin{align}\label{def of Acal and Bcal}
\mathcal{A} := \begin{pmatrix}
0 & 0 & 1 \\
1 & 0 & 0 \\
0 & 1 & 0
\end{pmatrix} \qquad \mbox{ and } \qquad \mathcal{B} := \begin{pmatrix}
0 & 1 & 0 \\
1 & 0 & 0 \\
0 & 0 & 1
\end{pmatrix}.
\end{align}

\item\label{RHniteme} $n(x,t,k) = (1,1,1) + O(k^{-1})$ as $k \to \infty$.

\item\label{RHnitemf} 
At each point of $\hat{\mathsf{Z}}$, two entries of $n$ are analytic while one entry has (at most) a simple pole. The following residue conditions hold at the points in $\hat{\mathsf{Z}}$: for each $k_{0}\in \mathsf{Z}\setminus \mathbb{R}$, 
\begin{align}\nonumber
& \underset{k = k_0}{\res} n_3(k) = c_{k_{0}}e^{-\theta_{31}(k_0)} n_1(k_0), & & \underset{k = \omega k_0}{\res} n_2(k) = \omega c_{k_{0}}e^{-\theta_{31}(k_0)} n_3(\omega k_0), 
	\\ \nonumber
& \underset{k = \omega^2 k_0}{\res} n_1(k) = \omega^2 c_{k_{0}}e^{-\theta_{31}(k_0)} n_2(\omega^2 k_0), & & \underset{k = k_0^{-1}}{\res} n_3(k) = - k_0^{-2} c_{k_{0}}e^{-\theta_{31}(k_0)} n_2(k_0^{-1}),	
	\\\nonumber
&  \underset{k = \omega^2k_0^{-1}}{\res} n_1(k) = -\tfrac{\omega^2}{k_0^{2}} c_{k_{0}}e^{-\theta_{31}(k_0)} n_3(\omega^2 k_0^{-1}), & & \underset{k = \omega k_0^{-1}}{\res} n_2(k) 
= -\tfrac{\omega}{k_0^{2}} c_{k_{0}}e^{-\theta_{31}(k_0)} n_1(\omega k_0^{-1}), 
	 \\ \nonumber
& \underset{k = \bar{k}_0}{\res} n_2(k) = d_{k_0} e^{\theta_{32}(\bar{k}_0)} n_3(\bar{k}_0), & & \underset{k = \omega \bar{k}_0}{\res} n_1(k) = \omega d_{k_0} e^{\theta_{32}(\bar{k}_0)} n_2( \omega \bar{k}_0),
	\\ \nonumber
& \underset{k = \omega^2 \bar{k}_0}{\res} n_3(k) = \omega^2 d_{k_0} e^{\theta_{32}(\bar{k}_0)} n_1(\omega^2 \bar{k}_0), & & \underset{k = \bar{k}_0^{-1}}{\res} n_1(k) = - \bar{k}_0^{-2} d_{k_0} e^{\theta_{32}(\bar{k}_0)} n_3(\bar{k}_0^{-1}),	
	\\ \label{nresiduesk0}
&  \underset{k = \omega^2 \bar{k}_0^{-1}}{\res} n_2(k) = -\tfrac{\omega^2}{\bar{k}_0^{2}} d_{k_0} e^{\theta_{32}(\bar{k}_0)} n_1( \omega^2 \bar{k}_0^{-1}), & & \underset{k = \omega \bar{k}_0^{-1}}{\res} n_3(k) = -\tfrac{\omega}{\bar{k}_0^{2}} d_{k_0} e^{\theta_{32}(\bar{k}_0)} n_2(\omega \bar{k}_0^{-1}), 
\end{align}
and, for each $k_{0}\in \mathsf{Z}\cap \mathbb{R}$, 
\begin{align}\nonumber
& \underset{k = k_0}{\res} n_2(k) = c_{k_0} e^{-\theta_{21}(k_0)} n_1(k_0), & & \underset{k = \omega k_0}{\res} n_1(k) = \omega c_{k_0} e^{-\theta_{21}(k_0)} n_3(\omega k_0), 
	\\  \nonumber
& \underset{k = \omega^2 k_0}{\res} n_3(k) = \omega^2 c_{k_0} e^{-\theta_{21}(k_0)} n_2(\omega^2 k_0), & & \underset{k = k_0^{-1}}{\res} n_1(k) = - k_0^{-2} c_{k_0} e^{-\theta_{21}(k_0)} n_2(k_0^{-1}),		
	\\ \label{nresiduesk0real}
&  \underset{k = \omega^2k_0^{-1}}{\res} n_2(k) = -\tfrac{\omega^2}{k_0^{2}} c_{k_0} e^{-\theta_{21}(k_0)} n_3(\omega^2 k_0^{-1}), & & \underset{k = \omega k_0^{-1}}{\res} n_3(k) 
= -\tfrac{\omega}{k_0^{2}} c_{k_0} e^{-\theta_{21}(k_0)} n_1(\omega k_0^{-1}),
\end{align}
where the $(x,t)$-dependence of $n$ and of $\theta_{ij}$ has been suppressed for brevity and the complex constants $d_{k_0}$ are defined for $k_{0}\in \mathsf{Z}\setminus \mathbb{R}$ by
\begin{align}\label{dk0def}
d_{k_{0}} := \frac{\bar{k}_0^2-1}{\omega^2 (\omega^2 - \bar{k}_0^2)} \bar{c}_{k_0}.
\end{align}
%Each condition in \eqref{nresiduesk0} and \eqref{nresiduesk0real} should be interpreted as stating that exactly one entry of $n$ has a simple pole, while the other two entries remain bounded; for example, in \eqref{res n 1 ab}, $\underset{k = k_0}{\res} n_3(k) = C_{k_{0}}(x,t) n_1(k_0)$ is shorthand notation for
%\begin{align*}
%n(z) = \frac{C_{k_{0}}(x,t)}{k-k_{0}} n(z) \begin{pmatrix}
%0 & 0 & 1 \\
%0 & 0 & 0 \\
%0 & 0 & 0 
%\end{pmatrix} + O(1), \qquad \mbox{as } k \to k_{0}.
%\end{align*}
\end{enumerate}
\end{RHproblem}

For simplicity, we will only deal with Schwartz class solutions in this paper. 

\begin{definition}\label{Schwartzsolutiondef}\upshape
Let $T \in (0, \infty]$. We say that $\{u(x,t), v(x,t)\}$ is a {\it Schwartz class solution of \eqref{boussinesqsystem} on $\R \times [0,T)$ with initial data $u_0, v_0 \in \mathcal{S}(\R)$} if
\begin{enumerate}[$(i)$] 
  \item $u,v$ are smooth real-valued functions of $(x,t) \in \R \times [0,T)$.

\item $u,v$ satisfy \eqref{boussinesqsystem} for $(x,t) \in \R \times [0,T)$ and 
$$u(x,0) = u_0(x), \quad v(x,0) = v_0(x), \qquad x \in \R.$$ 

  \item $u,v$ have rapid decay as $|x| \to \infty$ in the sense that, for each integer $N \geq 1$ and each $\tau \in [0,T)$,
$$\sup_{\substack{x \in \R \\ t \in [0, \tau]}} \sum_{i =0}^N (1+|x|)^N(|\partial_x^i u| + |\partial_x^i v| ) < \infty.$$
\end{enumerate} 
\end{definition}

\subsubsection{Statement of the theorem}

We showed in Theorem \ref{directth} that initial data in the Schwartz class satisfying Assumptions \ref{solitonassumption} and \ref{originassumption} give rise to scattering data with the properties $(\ref{Theorem2.3itemi})$--$(\ref{Theorem2.3itemvii})$ of Theorem \ref{directth}. Our next theorem establishes a converse of this result, namely it constructs a map from the set of scattering data satisfying $(\ref{Theorem2.3itemi})$--$(\ref{Theorem2.3itemvii})$ of Theorem \ref{directth} to the class of Schwartz class solutions $\{u(x,t), v(x,t)\}$ of the Boussinesq system \eqref{boussinesqsystem}.

\begin{theorem}[Inverse scattering]\label{inverseth}
Let $\{r_1(k), r_2(k), \mathsf{Z}, \{c_{k_0}\}_{k_0 \in \mathsf{Z}}\}$ be scattering data satisfying properties $(\ref{Theorem2.3itemi})$--$(\ref{Theorem2.3itemvii})$ of Theorem \ref{directth}.
Define $T \in (0, \infty]$ by 
\begin{align}\label{Tdef}
T := \sup \big\{t \geq 0 \, | \, \text{$e^{\frac{|k|^2t}{4}}r_1(1/k)$ and its derivatives are rapidly decreasing as $\Gamma_1 \ni k \to \infty$}\big\}.
\end{align}

Then RH problem \ref{RHn} has a unique solution $n(x,t,k)$ for each $(x,t) \in \R \times [0,T)$ and the function $n_{3}^{(1)}$ defined by
$$n_{3}^{(1)}(x,t) := \lim_{k\to \infty} k (n_{3}(x,t,k) -1)$$ 
is well-defined and smooth for $(x,t) \in \R \times [0,T)$. Moreover, $\{u(x,t), v(x,t)\}$ defined by
\begin{align}\label{recoveruvn}
\begin{cases}
u(x,t) = -i\sqrt{3}\frac{\partial}{\partial x}n_{3}^{(1)}(x,t),
	\\
v(x,t) = -i\sqrt{3}\frac{\partial}{\partial t}n_{3}^{(1)}(x,t),
\end{cases}
\end{align}
is a Schwartz class solution of (\ref{boussinesqsystem}) on $\R \times [0,T)$.
\end{theorem}

The proof of Theorem \ref{inverseth} is given in Section \ref{inversesec}.

\subsection{Solution of the initial value problem for (\ref{badboussinesq})}
Recalling the relation between (\ref{badboussinesq}) and the system (\ref{boussinesqsystem}), and using the solutions of both the direct and inverse problems, we can solve the initial value problem for the Boussinesq equation (\ref{badboussinesq}).

\begin{definition}\upshape
Let $T \in (0, \infty]$. We say that $u(x,t)$ is a {\it Schwartz class solution of \eqref{badboussinesq} on $\R \times [0,T)$ with initial data $u_0, u_1 \in \mathcal{S}(\R)$} if
\begin{enumerate}[$(i)$] 
  \item $u$ is a smooth real-valued functions of $(x,t) \in \R \times [0,T)$.

\item $u$ satisfies \eqref{badboussinesq} for $(x,t) \in \R \times [0,T)$ and 
$$u(x,0) = u_0(x), \quad u_t(x,0) = u_1(x), \qquad x \in \R.$$ 

  \item $u$ has rapid decay as $|x| \to \infty$ in the sense that, for each integer $N \geq 1$ and each $\tau \in [0,T)$,
$$\sup_{\substack{x \in \R \\ t \in [0, \tau]}} \sum_{i =0}^N (1+|x|)^N |\partial_x^i u(x,t)|  < \infty.$$
\end{enumerate} 
\end{definition}

\begin{theorem}[Solution of (\ref{badboussinesq}) via inverse scattering]\label{IVPth}
Let $u_0,u_1 \in \mathcal{S}(\R)$ be real-valued and suppose that
\begin{align}\label{u1zeromean}
\int_\R u_1(x) dx = 0.
\end{align}
Let $v_0(x) = \int_{-\infty}^x u_1(x')dx'$ and suppose $u_0, v_0$ are such that Assumptions \ref{solitonassumption} and \ref{originassumption} hold. Define the spectral functions $r_j(k)$, $j = 1,2$, in terms of $u_0, v_0$ by (\ref{r1r2def}). 
Define $T \in (0,\infty]$ by (\ref{Tdef}).
Then the initial value problem for (\ref{badboussinesq}) with initial data $u_0, u_1$ has a unique Schwartz class solution $u(x,t)$ on $\R \times [0,T)$. Moreover, $u(x,t)$  can be recovered from the solution $n(x,t,k)$ of RH problem \ref{RHn} via the representation formula (\ref{recoveruvn}) for any $(x,t) \in \R \times [0,T)$.
\end{theorem}
\begin{proof}
The key step is to show that the inverse scattering map of Theorem \ref{inverseth} at $t = 0$ is the inverse of the direct scattering map of Theorem \ref{directth}. The proof of this fact is based on Proposition \ref{prop:construction of n} and is analogous to the proof of \cite[Theorem 2.8]{CLmain}. 
\end{proof}

\subsection{Blow-up}
It follows from Theorem \ref{inverseth} that the solution $u(x,t)$ of (\ref{boussinesqsystem}) exists at least as long as $t < T$ where $T$ is defined by (\ref{Tdef}). Our next theorem, whose proof is identical to the proof of \cite[Theorem 2.9]{CLmain}, reveals that the solution in fact ceases to exist at $t=T$ if $T < \infty$.

\begin{theorem}[Blow-up]\label{blowupth}
Under the assumptions of Theorem \ref{inverseth}, if $T$ defined by (\ref{Tdef}) satisfies $T < \infty$, then the Schwartz class solution $\{u(x,t), v(x,t)\}$ defined in (\ref{recoveruvn}) blows up at time $T$.
\end{theorem}

More precisely, Theorem \ref{blowupth} states that there exists no Schwartz class solution $\{\tilde{u}, \tilde{v}\}$ of (\ref{boussinesqsystem}) on $[0,\tilde{T})$ with $\tilde{T} > T$ which coincides with $\{u, v\}$ for $t \in [0,T)$.

\subsection{Global solutions}
%As explained in Section \ref{globalsubsec}, solutions of (\ref{badboussinesq}) without ``unstable nonlinear Fourier modes'' are of particular interest. 
Global solutions of (\ref{badboussinesq}) are obtained by assuming the following. 

\begin{assumption}\label{nounstablemodesassumption}
The function $r_{1}:\hat{\Gamma}_{1}\setminus \hat{\mathcal{Q}}\to \mathbb{C}$ satisfies $r_{1}(k) = 0$ for all $k \in [0,i]$, where $[0,i]$ denotes the vertical segment from $0$ to $i$.
\end{assumption}

Assumption \ref{nounstablemodesassumption} implies that $T = \infty$ in (\ref{Tdef}) and therefore that the solution exists globally. The following theorem follows directly from Theorem \ref{IVPth}.

\begin{theorem}[Global solutions]\label{globalth}
Let $u_0,u_1 \in \mathcal{S}(\R)$ be such that the assumptions of Theorem \ref{IVPth} are fulfilled. Assume also that Assumption \ref{nounstablemodesassumption} holds. 
Then the initial value problem for (\ref{badboussinesq}) with initial data $u_0, u_1$ has a unique global Schwartz class solution $u(x,t)$. Moreover, $u(x,t)$ can be recovered from the solution $n(x,t,k)$ of RH problem \ref{RHn} via the representation formula (\ref{recoveruvn}) for any $(x,t) \in \R \times [0,\infty)$.
\end{theorem}

Finally, for completeness, we also state the following lemma, which will be used in \cite{CLmainSectorII} to obtain the long-time asymptotics of the solution in the presence of solitons. The proof is identical to \cite[Proof of Lemma 2.13]{CLmain}.

\begin{lemma}[Inequalities satisfied by the spectral functions]\label{inequalitieslemma}
Suppose $u_0,v_0 \in \mathcal{S}(\R)$ are such that Assumptions \ref{solitonassumption} and \ref{originassumption} hold. Let $\{r_j\}_1^2$ be the associated reflection coefficients defined in (\ref{r1r2def}).
\begin{enumerate}[$(i)$]
\item\label{inequalitieslemmaitemi}
 The function $f:\partial \mathbb{D}\to \mathbb{C}$ defined in (\ref{def of f}) satisfies 
\begin{itemize}
\item[$(a)$] $f(k) \geq 0$ for all $k \in \partial \mathbb{D}$,

\item[$(b)$] $f(k)=0$ if and only if $k \in \{\pm 1, \pm \omega\}$, and 

\item[$(c)$] $f(k) \leq 1$ for all $k \in \partial \mathbb{D}$ with $\arg k \in (2\pi/3, \pi) \cup (5\pi/3, 2\pi)$.
\end{itemize}

\item\label{inequalitieslemmaitemii} $1+r_{1}(k)r_{2}(k) > 0$ for all $k \in \partial \mathbb{D}$ with $\arg k \in (\pi/3, \pi) \cup (4\pi/3, 2\pi)$.

\item\label{inequalitieslemmaitemiii} $-\frac{1}{2\pi}\ln(1+r_{1}(k)r_{2}(k)) \geq 0$ for all $k \in \partial \mathbb{D}$ with $\arg k \in (5\pi/3, 2\pi)$.

\item\label{inequalitieslemmaitemiv} The functions $\hat{\nu}_1, \hat{\nu}_{2}:\partial \mathbb{D}\to \mathbb{C}$ defined by
\begin{align}\label{hatnu12def}
\hat{\nu}_1(k) := \nu_3(k) - \nu_1(k), \qquad \hat{\nu}_{2}(k) = \nu_{2}(k) +\nu_3(k) -\nu_{4}(k),
\end{align}
where
\begin{align}\nonumber
& \nu_1(k) := - \frac{1}{2\pi}\ln(1+r_{1}(\omega k)r_{2}(\omega k)), 
&& \nu_2(k) := - \frac{1}{2\pi}\ln(1+r_{1}(\omega^{2} k)r_{2}(\omega^{2} k)), 
	\\ \label{nu12345def}
& \nu_3(k) := - \frac{1}{2\pi}\ln(f(\omega k)),  &&
\nu_4(k) := - \frac{1}{2\pi}\ln(f(\omega^{2} k)).
\end{align}
satisfy $\hat{\nu}_1(k) \geq 0$ for all $k \in \partial \mathbb{D}$ with $\arg k \in (5\pi/3, 2\pi)$, and $\hat{\nu}_{2}(k) \geq 0$ for all $k \in \partial \mathbb{D}$ with $\arg k \in (\pi, 4\pi/3)$.

\end{enumerate} 
\end{lemma}

\section{The direct problem}\label{directsec}
Given initial data $u_0, v_0$ such that Assumptions \ref{solitonassumption} and \ref{originassumption} hold, we will construct a solution $M(x,t,k)$ of a $3 \times 3$-matrix RH problem. Then we will obtain a solution $n(x,t,k)$ of RH problem \ref{RHn} by setting $n=(1,1,1)M$. Our main goals are to prove Theorem \ref{directth} and Proposition \ref{prop:construction of n}; the latter is a key ingredient in the proof of Theorem \ref{IVPth}. 

The construction of $M$ and $n$ was carried out in \cite[Subsection 3.10]{CLmain} in the solitonless case. In what follows, proofs that are only minimally affected by the presence of solitons are omitted. In particular, various properties of $X,Y,X^{A},Y^{A},s,s^{A}$ were established in \cite[Propositions 3.2--3.9]{CLmain} in the case when no solitons are present. These propositions from \cite{CLmain} hold also under Assumption \ref{solitonassumption}, i.e. in the presence of solitons, with the same proofs. To avoid repetition, we will not reproduce the statements of these propositions here.

\subsection{Lax pair}
The system (\ref{boussinesqsystem}) is the compatibility condition of the Lax pair equations
\begin{equation}\label{Xlax}
\begin{cases}
 X_{x} - [\mathcal{L},X] = \mathsf{U} X, \\
 X_{t} - [\mathcal{Z},X] = \mathsf{V} X,
\end{cases}
\end{equation}
where $\mathcal{L} := \mbox{diag}(l_{1},l_{2},l_{3})$, $\mathcal{Z} := \mbox{diag}(z_{1},z_{2},z_{3})$, 
\begin{align}\label{UVLZdef}
& \mathsf{U} := L - \mathcal{L}, \quad \mathsf{V} := Z - \mathcal{Z}, \quad L := P^{-1}\tilde{L}P, \quad Z := P^{-1}\tilde{Z}P, 
	\\ \nonumber
& \tilde{L} := \begin{pmatrix}
0 & 1 & 0 \\
0 & 0 & 1 \\
\frac{\lambda}{12 i \sqrt{3}}-\frac{u_{x}}{4}-\frac{iv}{4\sqrt{3}} & -\frac{1+2u}{4} & 0
\end{pmatrix}, \quad
\tilde{Z} := \begin{pmatrix}
-i \frac{1+2u}{2\sqrt{3}} & 0 & -i \sqrt{3} \\
- \frac{\lambda}{12} - i \frac{u_{x}}{4\sqrt{3}} - \frac{v}{4} & i \frac{1+2u}{4\sqrt{3}} & 0 \\
-i \frac{u_{xx}}{4\sqrt{3}}-\frac{v_{x}}{4} & - \frac{\lambda}{12} + \frac{iu_{x}}{4\sqrt{3}}-\frac{v}{4} & i \frac{1+2u}{4\sqrt{3}}
\end{pmatrix}
\end{align}
and $\lambda = (k^3 + k^{-3})/2$.

\subsection{Construction of $M$ for $t = 0$}
We start by constructing $M$ at time $t = 0$. We let $M_n$ be the restriction of $M$ to the open region $D_n$, $n = 1,\dots, 6$. For each $n = 1, \dots, 6$, we define $M_n(x,0,k)$ for $k \in D_n\setminus \hat{\mathcal{Q}}$ as the $3\times 3$-matrix valued solution of the following system of Fredholm integral equations: 
\begin{align}\label{Mndef}
(M_n)_{ij}(x,0,k) = \delta_{ij} + \int_{\gamma_{ij}^n} \left(e^{(x-x')\widehat{\mathcal{L}(k)}} (\mathsf{U}M_n)(x',0,k) \right)_{ij} dx', \qquad  i,j = 1, 2,3,
\end{align}
where $\gamma^n_{ij}$, $n = 1, \dots, 6$, $i,j = 1,2,3$, are defined by
 \begin{align} \label{gammaijndef}
 \gamma_{ij}^n =  \begin{cases}
 (-\infty,x),  & \re  l_i(k) < \re  l_j(k), 
	\\
(+\infty,x),  \quad & \re  l_i(k) \geq \re  l_j(k),
\end{cases} \quad \text{for} \quad k \in D_n.
\end{align}
The definition (\ref{Mndef}) of $M_n$ is extended by continuity to the boundary of $D_n$. Let $\mathcal{Z}$ be the zero set of the Fredholm determinants associated with \eqref{Mndef}. 

\begin{proposition}\label{Mnprop}
If $u_0, v_0 \in \mathcal{S}(\R)$, then the following hold:
\begin{enumerate}[$(a)$]
\item The function $M_n(x,0, k)$ is defined for $x \in \R$ and $k \in \bar{D}_n \setminus (\hat{\mathcal{Q}}\cup\mathcal{Z})$. For each $k \in \bar{D}_n  \setminus (\hat{\mathcal{Q}}\cup\mathcal{Z})$, $M_n(\cdot,0, k)$ is smooth and satisfies $(M_n)_x - [\mathcal{L}, M_n] = \mathsf{U} M_n$.

\item For each $x \in \R$, $M_n(x,0,\cdot)$ is continuous for $k \in \bar{D}_n \setminus (\hat{\mathcal{Q}}\cup\mathcal{Z})$ and analytic for $k \in D_n\setminus (\hat{\mathcal{Q}}\cup\mathcal{Z})$.

\item For each $\epsilon > 0$, there exists a $C = C(\epsilon)$ such that $|M_n(x,0,k)| \leq C$ for $x \in \R$ and for all $k \in \bar{D}_n$ with $\dist(k, \hat{\mathcal{Q}}\cup\mathcal{Z}) \geq \epsilon$.

\item For each $x \in \R$ and each integer $j \geq 1$, $\partial_k^j M_n(x,0, \cdot)$ has a continuous extension to $\bar{D}_n \setminus (\hat{\mathcal{Q}}\cup\mathcal{Z})$.

\item $\det M_n(x,0,k) = 1$ for $x \in \R$ and $k \in \bar{D}_n \setminus (\hat{\mathcal{Q}}\cup\mathcal{Z})$.

\item For each $x \in \R$, the sectionally meromorphic function $M(x,0,k)$ defined by $M(x,0,k) = M_n(x,0,k)$ for $k \in D_n$ satisfies the symmetries
\begin{subequations}\label{Msymm}
\begin{align}\label{Msymma}
 & M(x,0, k) = \mathcal{A} M(x,0,\omega k)\mathcal{A}^{-1} = \mathcal{B} M(x,0,k^{-1}) \mathcal{B}, \qquad k \in \C \setminus (\hat{\mathcal{Q}}\cup\mathcal{Z}),
	\\
& \overline{(M^{A})(x,0,\bar{k})} = \bigg\{ \frac{u_{0}(x)}{2}\begin{pmatrix}
1 & 1 & 1 \\
1 & 1 & 1 \\
1 & 1 & 1 
\end{pmatrix} + R(k)^{-1} \bigg\}M(x,0,k)R(k), \qquad k \in \C \setminus (\hat{\mathcal{Q}}\cup\mathcal{Z}), \label{MsymmR}
\end{align}
\end{subequations}
where $M^A := (M^{-1})^T$ and 
\begin{align}\label{def of R}
R(k) := -4k^{2} \begin{pmatrix}
0 &  \frac{\omega}{(k^{2}-1)(k^{2}-\omega^{2})} &  0 \\
\frac{\omega^{2}}{(k^{2}-1)(k^{2}-\omega)} & 0 & 0 \\
0 & 0 & \frac{1}{(k^{2}-\omega)(k^{2}-\omega^{2})}
\end{pmatrix}.
\end{align} 
\end{enumerate}
\end{proposition}

In the solitonless case, the behavior of $M$ as $k \to \infty$ and as $k \to \pm 1$ was described in \cite[Lemmas 3.12 and 3.18]{CLmain}. The presence of solitons does not affect these lemmas, which we do not reproduce here.

%Recall that $M(x,k)$ is the sectionally meromorphic function defined by $M(x,k) = M_n(x,k)$ for $k \in D_n$.

\begin{lemma}[Jump condition for $M$]\label{Mjumplemma}
Let $u_0, v_0 \in \mathcal{S}(\R)$.
For each $x \in \R$, $M(x,0,k)$ satisfies the jump condition
\begin{align*}
  M_+(x,0,k) = M_-(x,0, k) v(x, 0, k), \qquad k \in \Gamma \setminus (\hat{\mathcal{Q}}\cup\mathcal{Z}),
\end{align*}
where $v$ is the jump matrix defined in (\ref{vdef}).
\end{lemma}

\begin{lemma}\label{M1XYlemma}
Let $u_0, v_0 \in \mathcal{S}(\R)$.
The functions $M_2$ and $M_2^A = (M_2^{-1})^T$ can be expressed in terms of the entries of $X,Y,X^A, Y^A, s$, and $s^A$ as follows:
\begin{align}\label{M2M2Aexpressions}
M_2 = \begin{pmatrix} 
X_{11} & \frac{Y_{12}}{s_{22}^A} & \frac{Y_{21}^AX_{32}^A - Y_{31}^AX_{22}^A}{s_{11}}  \\
X_{21} & \frac{Y_{22}}{s_{22}^A} & \frac{Y_{31}^AX_{12}^A - Y_{11}^AX_{32}^A}{s_{11}}  \\
X_{31} & \frac{Y_{32}}{s_{22}^A} & \frac{Y_{11}^AX_{22}^A - Y_{21}^AX_{12}^A}{s_{11}} 
\end{pmatrix}, \qquad
M_2^A = \begin{pmatrix} 
\frac{Y_{11}^A}{s_{11}} & X_{12}^A & \frac{X_{21}Y_{32} - X_{31}Y_{22}}{s_{22}^A}  \\
\frac{Y_{21}^A}{s_{11}} & X_{22}^A & \frac{X_{31}Y_{12} - X_{11}Y_{32}}{s_{22}^A}  \\
\frac{Y_{31}^A}{s_{11}} & X_{32}^A  & \frac{X_{11}Y_{22} - X_{21}Y_{12}}{s_{22}^A}
\end{pmatrix},
\end{align}
for all $x \in \R$ and $k \in \bar{D}_2 \setminus (\hat{\mathcal{Q}}\cup\mathcal{Z})$.
\end{lemma}

\begin{lemma}\label{QtildeQlemma}
Suppose $u_0,v_0 \in \mathcal{S}(\R)$ are such that Assumption \ref{solitonassumption} holds. Then the statements of Proposition \ref{Mnprop} and Lemmas \ref{Mjumplemma}-\ref{M1XYlemma} hold with $\mathcal{Z}$ replaced by $\hat{\mathsf{Z}}$.
\end{lemma}
\begin{proof}
Since $u_0, v_0$ satisfy Assumption \ref{solitonassumption},  Lemma \ref{M1XYlemma} implies that $M$ has no singularities apart from $k=\kappa_{j}$, $j=1,\ldots,6$, and from $k \in \mathsf{Z}$. Indeed, $s_{11}^{A}(k) \neq 0$ for $k \in \bar{D}_{5}\setminus (\hat{\mathcal{Q}}\cup \mathsf{Z}^{-*})$ so $s_{22}^{A}(k)=s_{11}^{A}(1/k) \neq 0$ for $k \in \bar{D}_{2}\setminus ( \hat{\mathcal{Q}} \cup \mathsf{Z}^{*} )$ and thus $M_2$ has no singularities in $\bar{D}_2 \setminus (\mathcal{Q}\cup \mathsf{Z} \cup \mathsf{Z}^{*})$. Using the symmetries (\ref{Msymma}), we then infer that $M_n$ has no singularities in $\bar{D}_n \setminus (\mathcal{Q}\cup \hat{\mathsf{Z}})$ for any $n$.
Thus, for any $n$, $M_{n}(x,0,k)$ can be extended to any $k \in (\mathcal{Z} \cap \bar{D}_n)  \setminus (\mathcal{Q}\cup \hat{\mathsf{Z}})$ by continuity, from which the claim follows.
\end{proof}

The following symmetry relation proved in \cite[Proposition 3.9$(h)$]{CLmain} will be useful in the sequel:
\begin{align}\label{sRsymm}
\overline{s^{A}(\bar{k})} = R(k)^{-1}s(k)R(k) \qquad \mbox{for } k \in \begin{pmatrix}
 \omega^2 \hat{\mathcal{S}} & \hat{\Gamma}_{1} & \hat{\Gamma}_{3} \\
 \hat{\Gamma}_{1} & \omega \hat{\mathcal{S}} & \hat{\Gamma}_{5} \\
\hat{\Gamma}_{3} & \hat{\Gamma}_{5} & \hat{\mathcal{S}}
 \end{pmatrix}\setminus \hat{\mathcal{Q}}.
 \end{align}
 More explicitly, we can write the relation in (\ref{sRsymm}) as
$$\begin{pmatrix} 
\overline{m_{11}(s(\bar{k}))} & -\overline{m_{12}(s(\bar{k}))} & \overline{m_{13}(s(\bar{k}))} \\
-\overline{m_{21}(s(\bar{k}))} & \overline{m_{22}(s(\bar{k}))} & -\overline{m_{23}(s(\bar{k}))} \\
\overline{m_{31}(s(\bar{k}))} & -\overline{m_{32}(s(\bar{k}))} & \overline{m_{33}(s(\bar{k}))}
\end{pmatrix} = 
\begin{pmatrix}
 s_{22}(k) & \tilde{r}(\frac{1}{k}) s_{21}(k)  & \tilde{r}(\omega k) s_{23}(k)
   	\\
\tilde{r}(k) s_{12}(k) & s_{11}(k) & \tilde{r}(\frac{1}{\omega^2 k})  s_{13}(k)
   	\\
\tilde{r}(\frac{1}{\omega k})  s_{32}(k) & \tilde{r}(\omega^2 k)  s_{31}(k)  & s_{33}(k)
\end{pmatrix},$$
 where $\tilde{r}(k)$ is given in (\ref{rtildedef}) and $m_{jk}(s(\bar{k}))$ denotes the $(jk)$th minor of the matrix $s(\bar{k})$.

The next lemma will be used in the proof of Theorem \ref{directth}.

\begin{lemma}\label{s13s31lemma}
Suppose $u_0, v_0 \in \mathcal{S}(\R)$ are compactly supported and that $s_{11}$ has a simple zero at some $k_0 \in D_2 \cap \R$. Then $s_{13}(k_0) = s_{31}(k_0) = 0$ and $i(\omega^2 k_0^2 - \omega) \frac{s_{12}(k_0)}{\dot{s}_{11}(k_0)} \in \R$.
\end{lemma}
\begin{proof}
Since  $u_0, v_0$ have compact support, all entries of $s(k)$ are defined and analytic for  $k \in \C \setminus \hat{\mathcal{Q}}$ and the symmetry relation (\ref{sRsymm}) holds for all $k \in \C \setminus \hat{\mathcal{Q}}$.

Since $s_{11}(k_0) = 0$ and $k_0$ is real, we see from (\ref{sRsymm}) that $m_{22}(s(k)) = s_{11}(k) s_{33}(k) - s_{13}(k) s_{31}(k)$ also has a zero at $k_0$. Thus we must have either $s_{13}(k_0) = 0$ or $s_{31}(k_0) = 0$. 
Suppose first that $s_{13}(k_0) = 0$. Then, using (\ref{sRsymm}) again, $m_{23}(k_0) = 0$. But $m_{23} = s_{11}s_{32} - s_{12}s_{31}$, so we conclude that either $s_{12}$ or $s_{31}$ must vanish at $k_0$. If $s_{12}(k_0) = 0$, then the first row of $s(k)$ would vanish at $k_0$, which contradicts the fact that $\det s(k)$ is identically equal to $1$. So we must have $s_{31}(k_0) = 0$.
A similar argument shows that $s_{13}(k_0) = 0$ if $s_{31}(k_0) = 0$. It follows that both $s_{13}$ and $s_{31}$  vanish at $k_0$. 

To prove the last assertion, we employ the $R$-symmetry (\ref{sRsymm}) twice and the fact that $s_{13}(k_0) = s_{31}(k_0) = 0$ to get
\begin{align*}
  \overline{i(\omega^2 k_0^2 - \omega)\frac{s_{12}(k_0)}{\dot{s}_{11}(k_0)} }
&  =  -i (\omega k_0^2 - \omega^2) \frac{\overline{s_{12}(k_0)}}{\overline{\dot{s}_{11}(k_0)}}
  = \frac{i (\omega k_0^2 - \omega^2)}{\overline{\tilde{r}(k_0)}} \frac{m_{21}(s(k_0))}{\partial_k m_{22}(s(k_0))}
	\\
& = i(\omega^2 k_0^2 - \omega) \frac{s_{12}s_{33} - s_{13}s_{32}}{\partial_k (s_{11}s_{33} - s_{13}s_{31})}\bigg|_{k= k_0}
   =i(\omega^2 k_0^2 - \omega) \frac{s_{12}(k_0)}{\dot{s}_{11}(k_0)},
\end{align*} 
which gives the desired conclusion.
\end{proof}

\subsection{Proof of Theorem \ref{directth}}\label{directproofsubsec}
By Assumption \ref{solitonassumption}, the set $\mathsf{Z} \cap \Gamma$ is empty. Hence the proof of properties $(\ref{Theorem2.3itemi})$--$(\ref{Theorem2.3itemv})$ is identical to \cite[Proof of Theorem 2.3 $(i)$--$(v)$]{CLmain}. 

Let us prove ($\ref{Theorem2.3itemvi}$). By Assumption \ref{solitonassumption}, $\mathsf{Z}$ is a finite subset of $D_{\mathrm{reg}} \cup (-1,0) \cup (1,\infty)$. We need to show that the complex constants $\{c_{k_0}\}_{k_0 \in \mathsf{Z}}$ are well-defined by (\ref{ck0def}).

Suppose first that $k_{0} \in \mathsf{Z}\setminus \mathbb{R}$. We will show that if $w$ is defined by (\ref{wdef}) then there is a unique complex constant $c_{k_0}$ such that (\ref{ck0def1}) holds.
Let $\eta \in C_c^\infty(\R)$ be a bump function such that $\eta(x) = 1$ for $|x| \leq 1$ and $\eta(x) = 0$ for $|x| \geq 2$. For $j \geq 1$, let $\eta_j(x) = \eta(x/j)$. 
If $f \in \mathcal{S}(\R)$, then $\eta_jf$ is a sequence in $C_c^\infty(\R)$ converging to $f$ in $\mathcal{S}(\R)$ as $j \to \infty$. 
Let $X^{(i)}(x,k), Y^{(i)}(x,k), X^{A(i)}(x,k), Y^{A(i)}(x,k), s^{(i)}(k), s^{A(i)}(k)$ be the eigenfunctions and scattering matrices associated with 
\begin{align}\label{uvsequence}
(u_0^{(i)}(x), v_0^{(i)}(x)) := (\eta_iu_0, \eta_iv_0) \in \mathcal{S}(\R) \times \mathcal{S}(\R).
\end{align} 
By \cite[Lemma 3.13]{CLmain} with $(u_{0},v_{0})$ replaced by  $(u_0^{(i)}, v_0^{(i)})$, we have, for all $k \in D_2$ except at the zeros of $s_{11}^{(i)}$ and $m_{22}(s^{(i)})$, 
\begin{align}\label{lol12}
M_2^{(i)}(x,k) & = Y^{(i)}(x, k) e^{x\widehat{\mathcal{L}(k)}} S_2^{(i)}(k)	 = X^{(i)}(x, k) e^{x\widehat{\mathcal{L}(k)}} T_2^{(i)}(k),
\end{align}
where 
$$S_2^{(i)}(k) =  \begin{pmatrix}
 s_{11}^{(i)} & 0 & 0 \\
 s_{21}^{(i)} & \frac{1}{m_{22}(s^{(i)})} & \frac{m_{32}(s^{(i)})}{s_{11}^{(i)}} \\
 s_{31}^{(i)} & 0 & \frac{m_{22}(s^{(i)})}{s_{11}^{(i)}} \\
\end{pmatrix},
	\qquad
  T_2^{(i)}(k) =  \begin{pmatrix}
 1 & -\frac{m_{21}(s^{(i)})}{m_{22}(s^{(i)})} & -\frac{s_{13}^{(i)}}{s_{11}^{(i)}} \\
 0 & 1 & 0 \\
 0 & -\frac{m_{23}(s^{(i)})}{m_{22}(s^{(i)})} & 1 
\end{pmatrix}.
$$
The third column of \eqref{lol12} yields, for $k \in D_2$,
\begin{align}\label{m32Ym22sY}
m_{32}(s^{(i)}) e^{x(l_2 - l_3)} [Y^{(i)}]_2 + m_{22}(s^{(i)}) [Y^{(i)}]_3 = -s_{13}^{(i)}e^{x(l_1 - l_3)} [X^{(i)}]_1 + s_{11}^{(i)}[X^{(i)}]_3.
\end{align}
Applying $\det(\cdot, [X^{(i)}]_1, [Y^{(i)}]_2)$ to both sides of (\ref{m32Ym22sY}), we obtain
$$m_{22}(s^{(i)}) \det([Y^{(i)}]_3, [X^{(i)}]_1, [Y^{(i)}]_2) = s_{11}^{(i)}\det([X^{(i)}]_3, [X^{(i)}]_1, [Y^{(i)}]_2).$$
Using the inner product notation $u \cdot v = u_1 v_1 + u_2 v_2 + u_3 v_3$ (no complex conjugate) for two vectors $u = ( u_1, u_2, u_3 )$ and $v =  ( v_1, v_2, v_3 )$ in $\C^3$, we can write this as
\begin{align}\label{m22X1s11Y2}
m_{22}(s^{(i)}) [X^{(i)}]_1 \cdot [Y^{A(i)}]_1 = s_{11}^{(i)} [Y^{(i)}]_2 \cdot [X^{A(i)}]_2.
\end{align}
For $k \in D_2$, all quantities in this equation have well-defined limits as $i \to \infty$, so letting $i \to \infty$ gives
$$s^{A}_{22}(k) [X(x,k)]_1 \cdot [Y^A(x,k)]_1 = s_{11}(k) [Y(x,k)]_2 \cdot [X^A(x,k)]_2, \qquad k \in D_2.$$
Evaluating at $k_0$ and using that $s_{11}(k_0) = 0$ and $s^{A}_{22}(k_0) \neq 0$ by Assumption \ref{solitonassumption}, we obtain
$$[X(x,k_0)]_1 \cdot [Y^A(x,k_0)]_1 = 0.$$
On the other hand, since $\det([X^{(i)}]_1, [X^{(i)}]_1, [X^{(i)}]_3)(x,k_0) = 0$, we have 
$$[X(x,k_0)]_1 \cdot [X^A(x,k_0)]_2 = 0.$$
It follows that $[X(x,k_0)]_1$ lies in the kernel of the linear map $\C^3 \to \C^{2}$ defined by the matrix
\begin{align}\label{sfAdef}
\mathsf{A} = \begin{pmatrix} 
Y^A_{11} & Y^A_{21} & Y^A_{31} \\
X^A_{12} & X^A_{22} & X^A_{32} 
\end{pmatrix}(x,k_{0}).
\end{align}
The vector $w(x,k_{0})$ also lies in the kernel of $\mathsf{A}$ (by a direct computation or by noting that it is the cross product of the two rows of $\mathsf{A}$). 
If $[Y^A(x,k_0)]_1$ and $[X^A(x,k_0)]_2$ are linearly independent for some $x$ then they are linearly independent for all $x$, because $e^{-x l_1(k_0)} [Y^A(x,k_0)]_1$  and $e^{-x l_2(k_0)} [X^A(x,k_0)]_2$ satisfy the same linear ODE in $x$. As $x \to -\infty$, $[Y^A(x,k_0)]_1 \to (1,0,0)$ and, by (\ref{def of X XA Y YA}) and (\ref{def of s sA}), $X_{22}^A(x,k_0) \to s^A_{22}(k_0) \neq 0$. Thus $[Y^A(x,k_0)]_1$ and $[X^A(x,k_0)]_2$ are linearly independent for large negative $x$. Hence $\rank \mathsf{A} = 2$ and so, by the rank-nullity theorem, $\dim \ker \mathsf{A} = 1$. 
Since the vectors $[X(x,k_0)]_1$ and $w(x,k_{0})$ both lie in $\ker \mathsf{A}$, they are linearly dependent for each $x$. Hence there is a function $\mathrm{c}_{k_{0}}(x)$ such that
$$\frac{w(x,k_{0})}{\dot{s}_{11}(k_0)} = \mathrm{c}_{k_{0}}(x)  e^{(l_1(k_0)-l_3(k_0))x} [X(x,k_0)]_1 \qquad \text{for all $x \in \R$}.$$
Since the vectors $e^{x l_3(k_0)} w(x,k_{0})$ and $e^{x l_1(k_0)} [X]_1(x,k_{0})$ satisfy the same linear ODE in $x$, $c_{k_{0}} = c_{k_{0}}(x) \in \C$ is in fact independent of $x$. Since $[X(x,k_0)]_1$ is not identically zero, $c_{k_{0}}$ is uniquely determined.
This shows that $c_{k_0}$ is well-defined in the case when $k_{0} \in \mathsf{Z}\setminus \mathbb{R}$. 

Suppose now that $k_{0} \in \mathsf{Z} \cap \mathbb{R}$. We need to show that there is a unique complex constant $c_{k_0}$ such that (\ref{ck0def2}) holds.
By Assumption \ref{solitonassumption}, $s_{11}$ has a simple zero at $k_0$.
The $R$-symmetry relation (\ref{sRsymm}) gives $\overline{s_{22}^A(\bar{k})} = s_{11}(k)$, so that $s_{22}^A$ also has a simple zero at $k_0$. In particular, $\dot{s}_{22}^A(k_0) \neq 0$.
Assumption \ref{solitonassumption} also shows that there exists an $\epsilon > 0$ such that $\overline{D_\epsilon(k_0)} \subset D_2$ and $s_{11}(k) \neq 0$ for all $k \in \overline{D_\epsilon(k_0)} \setminus \{k_0\}$. 
Since $k_0$ is a simple zero of $s_{11}(k)$, we have
$$\frac{1}{2\pi i} \int_{\partial D_\epsilon(k_0)} \frac{\dot{s}_{11}(k)}{s_{11}(k)} dk = 1 \quad \text{and} \quad
\frac{1}{2\pi i} \int_{\partial D_\epsilon(k_0)} \frac{k\dot{s}_{11}(k)}{s_{11}(k)} dk = k_0.$$
The functions $s^{(i)}_{11}$ converge uniformly to $s_{11}$ on $\partial D_\epsilon(k_0)$ and thus $s_{11}^{(i)}$ is nonzero on $\partial D_\epsilon(k_0)$ for all large enough $i$ and 
\begin{align}\label{ints11limits}
\frac{1}{2\pi i} \int_{\partial D_\epsilon(k_0)} \frac{\dot{s}^{(i)}_{11}(k)}{s_{11}^{(i)}(k)} dk \to 1 \quad \text{and} \quad
\frac{1}{2\pi i} \int_{\partial D_\epsilon(k_0)} \frac{k\dot{s}_{11}^{(i)}(k)}{s_{11}^{(i)}(k)} dk \to k_0
\end{align}
as $i \to \infty$. For all large enough $i$, we thus have $\frac{1}{2\pi i} \int_{\partial D_\epsilon(k_0)} \frac{\dot{s}^{(i)}_{11}(k)}{s_{11}^{(i)}(k)} dk = 1$ because the left-hand side takes only integer values, so the argument principle implies that $s_{11}^{(i)}(k)$ has exactly one zero in $D_\epsilon(k_0)$ and that this zero is simple. Denoting this simple zero by $k_0^{(i)}$, the second limit in (\ref{ints11limits}) shows that $k_0^{(i)} \to k_0$ as $i \to \infty$.

We now distinguish two cases: either there exist arbitrarily large indices $i$ such that $k_0^{(i)} \notin \R$ or there do not. By picking subsequences, we see that it is sufficient to consider the following two cases: 
\begin{enumerate}
\item[Case 1.] $k_0^{(i)} \in \R$ for all $i$.
\item[Case 2.] $k_0^{(i)} \notin \R$ for all $i$.
\end{enumerate}

Let us first consider Case 1. Since $k_0^{(i)} \in \R$ for all $i$, we can deduce from Lemma \ref{s13s31lemma} (applied with $u_0, v_0$ replaced by $u_0^{(i)}, v_0^{(i)}$) that $s_{13}^{(i)}(k_0^{(i)}) = s_{31}^{(i)}(k_0^{(i)}) = 0$, and hence also
\begin{align}\label{m22m230}
m_{22}(s^{(i)}(k_0^{(i)})) = m_{23}(s^{(i)}(k_0^{(i)})) = m_{32}(s^{(i)}(k_0^{(i)})) = 0,
\end{align}
for all $i$. 
On the other hand, by considering the second column of the second equality in \eqref{lol12},
$$[Y^{(i)}]_2 = -m_{21}(s^{(i)}) e^{(l_1 - l_2)x} [X^{(i)}]_1 + m_{22}(s^{(i)}) [X^{(i)}]_2 - m_{23}(s^{(i)}) e^{(l_3 - l_2)x} [X^{(i)}]_3$$
for $k \in D_2$. Evaluating this equation at $k_0^{(i)}$ and using (\ref{m22m230}), we obtain
\begin{align*}
 [Y^{(i)}(x,k_0^{(i)})]_2 = -m_{21}(s^{(i)}(k_0^{(i)})) e^{(l_1(k_0^{(i)}) - l_2(k_0^{(i)}))x} [X^{(i)}(x, k_0^{(i)})]_1.
\end{align*}
Dividing by $\dot{s}_{22}^A(k_0)$, we can write this as
\begin{align}\label{Yi2m21eX1}
 \frac{[Y^{(i)}(x,k_0^{(i)})]_2}{\dot{s}_{22}^A(k_0)} = c_{k_0}^{(i)} e^{(l_1(k_0^{(i)}) - l_2(k_0^{(i)}))x} [X^{(i)}(x, k_0^{(i)})]_1,
\end{align}
where
\begin{align}\label{ck0idef}
c_{k_0}^{(i)} := -\frac{m_{21}(s^{(i)}(k_0^{(i)}))}{\dot{s}_{22}^A(k_0)}.
\end{align}
Since $X_{11}(x,k_0) \to 1$ as $x \to +\infty$, there is an $x_0$ such that $X_{11}(x_0,k_0) \neq 0$. Consequently, the limit 
$$\lim_{i \to \infty} c_{k_0}^{(i)} = \lim_{i\to \infty} \frac{Y_{12}^{(i)}(x_0,k_0^{(i)})e^{-(l_1(k_0^{(i)}) - l_2(k_0^{(i)}))x_{0}}}{\dot{s}_{22}^A(k_0)X_{11}^{(i)}(x_{0}, k_0^{(i)})}
= \frac{Y_{12}(x_0,k_0)e^{-(l_1(k_0) - l_2(k_0))x_{0}}}{\dot{s}_{22}^A(k_0)X_{11}(x_0, k_0)}$$ 
exists as a finite number.
Sending $i \to \infty$ in (\ref{Yi2m21eX1}), we conclude that (\ref{ck0def2}) holds with $c_{k_0} := \lim_{i \to \infty} c_{k_0}^{(i)} \in \C$.

We next consider Case 2. Recall that $k_0^{(i)}$ is the only zero of $s_{11}^{(i)}(k)$ in $D_\epsilon(k_0)$. Since $k_0^{(i)} \notin \R$ and $\overline{s_{22}^{A(i)}(\bar{k})} = s_{11}^{(i)}(k)$, we see that $\overline{k_0^{(i)}}$ is the only zero of $s_{22}^{A(i)}(k)$ in $D_\epsilon(k_0)$, and that $s_{11}^{(i)}(\overline{k_0^{(i)}}) \neq 0$ and $s_{22}^{A(i)}(k_0^{(i)}) \neq 0$. Thus, we may use (\ref{m22X1s11Y2}) evaluated at $k_0^{(i)}$ and at $\overline{k_0^{(i)}}$ to deduce that 
$$[X^{(i)}(x,k_0^{(i)})]_1 \cdot [(Y^A)^{(i)}(x,k_0^{(i)})]_1 = 0, \qquad
[Y^{(i)}(x,\overline{k_0^{(i)}})]_2 \cdot [X^{A(i)}(x,\overline{k_0^{(i)}})]_2 = 0.$$
Taking the limits of these relations as $i \to \infty$, we get
\begin{align}\label{scalarproducts1}
[X(x,k_0)]_1 \cdot [Y^A(x,k_0)]_1 = 0, \qquad
[Y(x,k_0)]_2 \cdot [X^A(x,k_0)]_2 = 0.
\end{align}
On the other hand, taking the limits of the relations $\det([Y^{(i)}]_2, [Y^{(i)}]_2, [Y^{(i)}]_3)(x,k_0) = 0$ and $\det([X^{(i)}]_1, [X^{(i)}]_1, [X^{(i)}]_3)(x,k_0) = 0$ as $i \to \infty$, we obtain 
\begin{align}\label{scalarproducts2}
[Y(x,k_0)]_2 \cdot [Y^A(x,k_0)]_1 = 0, \qquad [X(x,k_0)]_1 \cdot [X^A(x,k_0)]_2 = 0.
\end{align}
The relations (\ref{scalarproducts1}) and (\ref{scalarproducts2}) show that the vectors $[X(x,k_0)]_1$ and $[Y(x,k_0)]_2$ both lie in the kernel of the matrix $\mathsf{A}$ in (\ref{sfAdef}). If $\rank \mathsf{A} = 2$ at some $x_0$, then $[X(x_0,k_0)]_1$ and $[Y(x_0,k_0)]_2$ are linearly dependent at $x_0$. If $\rank \mathsf{A} < 2$ at some $x_0$, then the two rows of $\mathsf{A}$ are linearly dependent, implying that $w(x_0,k_0) =0$. But taking the residue of the third column of (\ref{MsymmR}) at $k_0$, recalling (\ref{M2M2Aexpressions}), and using that $w(x_0,k_0) =0$, we find
$$\begin{pmatrix}
\frac{(X_{21}Y_{32} - X_{31}Y_{22})(x_0,k_0)}{\dot{s}_{22}^A(k_0)}  \\[0.2cm]
\frac{(X_{31}Y_{12} - X_{11}Y_{32})(x_0,k_0)}{\dot{s}_{22}^A(k_0)}  \\[0.2cm]
\frac{(X_{11}Y_{22} - X_{21}Y_{12})(x_0,k_0)}{\dot{s}_{22}^A(k_0)}
\end{pmatrix} = 0,$$
so also in this case the vectors $[X(x_0,k_0)]_1$ and $[Y(x_0,k_0)]_2$ are linearly dependent.
Thus, in either case, the vectors $[X(x,k_0)]_1$ and $[Y(x,k_0)]_2$ are linearly dependent for $x = x_0$.
As above, it follows that they are linearly dependent for all $x$ and that (\ref{ck0def2}) holds for some $c_{k_0} \in \C$ independent of $x$.
This completes the proof of assertion ($\ref{Theorem2.3itemvi}$) of Theorem \ref{directth}.
 
It only remains to prove ($\ref{Theorem2.3itemvii}$). Since $i(\omega^2 k_0^2 - \omega) c_{k_0} \notin (-\infty,0)$ by Assumption \ref{solitonassumption}, ($\ref{Theorem2.3itemvii}$) will follow if we can show that $i(\omega^2 k_0^2 - \omega) c_{k_0}$ is real. In the case of compactly supported initial data, Lemma \ref{s13s31lemma} implies that $s_{13}$ and $s_{31}$ vanish at $k_0$ and hence, using the identity $X(x,k_{0})=Y(x,k_{0})e^{x\widehat{\mathcal{L}(k_{0})}}s(k_{0})$ (proved in \cite[Proposition 3.5]{CLmain}) evaluated at a point $x_{0}$ for which $X(x_{0},k_{0})=I$, we get
\begin{align*}
c_{k_{0}} = \frac{e^{-(l_{1}(k_{0})-l_{2}(k_{0}))x_{0}}Y_{12}(x_{0},k_{0})}{\dot{s}^{A}_{22}(k_{0})X_{11}(x_{0},k_{0})} = \frac{s^{A}_{21}(k_{0})}{\dot{s}^{A}_{22}(k_{0})} = -\frac{s_{12}(k_{0})}{\dot{s}_{11}(k_{0})},
\end{align*}
in agreement with (\ref{ck0compactdef}); thus the reality of $i(\omega^2 k_0^2 - \omega) c_{k_0}$ follows from Lemma \ref{s13s31lemma}.

In the general case, we proceed as follows. 
Considering the terms of $O((k-k_0)^{-1})$ and of $O(1)$ of the first and second columns of (\ref{MsymmR}), respectively, and using (\ref{M2M2Aexpressions}), we find
\begin{align}\label{YAXAQ}
& \overline{[Y^A(x,k_0)]_1} = \mathsf{Q}[Y(x,k_0)]_2,
\qquad
\overline{[X^A(x,k_0)]_2} = \overline{\tilde{r}(k_0)} \mathsf{Q} [X(x,k_0)]_1,
\end{align}
where $\mathsf{Q} = \mathsf{Q}(x,k_0)$ is short-hand notation for
$$\mathsf{Q} := \begin{pmatrix} 0 & 1 & 0 \\ \tilde{r}(k_0) & 0 & 0 \\ 0 & 0 & \tilde{r}(\frac{1}{\omega k_0})  \end{pmatrix} 
- \frac{2 \omega^2 k_0^2 u_0(x)}{(k_0^2 -1)(k_0^2 -\omega)}  \begin{pmatrix} 1 & 1 & 1 \\ 1 & 1 & 1 \\ 1 & 1 & 1 \end{pmatrix}.$$
Taking the complex conjugate of the relations in (\ref{YAXAQ}) and employing (\ref{ck0def2}), we obtain
\begin{align}\label{YAXAlinearlydependent}
\frac{[Y^A(x,k_0)]_1}{\dot{s}_{11}(k_0)} = \bar{c}_{k_0} \tilde{r}(k_0)^{-1} e^{(l_1(k_0)-l_2(k_0))x} [X^A(x,k_0)]_2\qquad \text{for all $x \in \R$},
\end{align}
where we have also used that $\overline{\dot{s}^{A}_{22}(k_{0})} = \dot{s}_{11}(k_{0})$ (which follows from \eqref{sRsymm}).
In particular, $[Y^A(x,k_0)]_1$ and $[X^A(x,k_0)]_2$ are linearly dependent, which implies that
\begin{align}\label{wzeroatk0}
w(x,k_0) = 0\qquad \text{for all $x \in \R$ if } k_0 \in \mathsf{Z} \cap \mathbb{R}.
\end{align}
In light of (\ref{M2M2Aexpressions}), this means that $[M(x,k)]_3$ is analytic at $k_0$. 
Let $u \times v = (u_2v_3 - u_3v_2,  u_3v_1 - u_1v_3, u_1 v_2 - u_2 v_1)$ denote the cross product of two vectors $u = ( u_1, u_2, u_3 ) \in \C^3$ and $v =  ( v_1, v_2, v_3 ) \in \C^3$. Taking the residue at $k_0$ of the identity
$$[M^A(x,k)]_1 = [M(x,k)]_2 \times [M(x,k)]_3$$
and using (\ref{ck0def2}), \eqref{M2M2Aexpressions}, $\dot{s}^{A}_{22}(k_{0}) = \overline{\dot{s}_{11}(k_{0})}$, $[M^A(x,k)]_2 = -[M(x,k)]_1 \times [M(x,k)]_3$, and (\ref{YAXAQ}), we find
\begin{align}\nonumber
& \frac{[Y^A(x,k_0)]_1}{\dot{s}_{11}(k_0)} 
= \frac{[Y(x,k_0)]_2}{\overline{\dot{s}_{11}(k_0)}} \times [M(x,k_0)]_3
= c_{k_0} e^{(l_1(k_0)-l_2(k_0))x} [X(x,k_0)]_1 \times [M(x,k_0)]_3
	\\ \label{YA1dots11}
&\quad = - c_{k_0} e^{(l_1(k_0)-l_2(k_0))x} [X^A(x,k_0)]_2 = -c_{k_0} e^{(l_1(k_0)-l_2(k_0))x} \tilde{r}(k_0) \bar{\mathsf{Q}} \overline{[X(x,k_0)]_1}.
\end{align}
On the other hand, by (\ref{ck0def2}) and (\ref{YAXAQ}),
$$\frac{\overline{[Y^A(x,k_0)]_1}}{\overline{\dot{s}_{11}(k_0)}} = \frac{\mathsf{Q}[Y(x,k_0)]_2}{\overline{\dot{s}_{11}(k_0)}} = \mathsf{Q}  c_{k_0} e^{(l_1(k_0)-l_2(k_0))x} [X(x,k_0)]_1.$$
We also have $l_1(k_0)-l_2(k_0) \in \R$. Thus the complex conjugate of (\ref{YA1dots11}) can be written as
$$\mathsf{Q} c_{k_0} [X(x,k_0)]_1 
= - \bar{c}_{k_0} \overline{\tilde{r}(k_0)} \mathsf{Q} [X(x,k_0)]_1.$$
Since $\det \mathsf{Q} = -\tilde{r}(k_0)\tilde{r}(\frac{1}{\omega k_0}) \neq 0$ and $[X(x,k_0)]_1$ is not identically zero, we conclude that $c_{k_0} =  -\bar{c}_{k_0} \overline{\tilde{r}(k_0)}$, which is equivalent to $i(\omega^2 k_0^2 - \omega) c_{k_0} \in \R$.
This completes the proof of Theorem  \ref{directth}.

\subsection{Construction of $M$ for $t > 0$}
Let us now consider the time-dependence of $M$. Given a Schwartz class solution of (\ref{boussinesqsystem}) on $\R \times [0,T)$,  $\{M_n(x,t,k)\}_{n=1}^6$ is defined by replacing $\mathsf{U}(x,0,k)$ with the time-dependent matrix $\mathsf{U}(x,t,k)$ in the integral equations (\ref{Mndef}). We let $M(x,t,k)$ be the sectionally meromorphic function which equals $M_n(x,t,k)$ for $k \in D_n$. 
We will prove that $M(x,t,k)$ satisfies the following $3 \times 3$-matrix RH problem. 

\begin{RHproblem}[RH problem for $M$]\label{RH problem for M}
Find $M(x,t,k)$ with the following properties:
\begin{enumerate}[$(a)$]
\item \label{RHMitema}
$M(x,t,\cdot) : \mathbb{C}\setminus (\Gamma \cup \hat{\mathsf{Z}}) \to \mathbb{C}^{3 \times 3}$ is analytic.

\item \label{RHMitemb}
The limits of $M(x,t,k)$ as $k$ approaches $\Gamma\setminus (\Gamma_\star \cup \mathcal{Q})$ from the left and right exist, are continuous on $\Gamma\setminus (\Gamma_\star \cup \mathcal{Q})$, and satisfy
\begin{align}\label{Mjumpcondition}
& M_{+}(x,t,k) = M_{-}(x,t,k)v(x,t,k), \qquad k \in \Gamma \setminus (\Gamma_\star \cup \mathcal{Q}),
\end{align}
where $v$ is defined by \eqref{vdef}.

\item \label{RHMitemc}
As $k \to \infty$, 
\begin{align}\label{asymp for M at infty in RH def}
M(x,t,k) = I + \frac{M^{(1)}(x,t)}{k} + \frac{M^{(2)}(x,t)}{k^{2}} + O\bigg(\frac{1}{k^3}\bigg),
\end{align}
where the matrices $M^{(1)}$ and $M^{(2)}$ depend on $x$ and $t$ but not on $k$, and satisfy
\begin{align}\label{singRHMatinftyb}
M_{12}^{(1)} = M_{13}^{(1)} = M_{12}^{(2)} + M_{21}^{(2)} = 0.
\end{align}

\item \label{RHMitemd}
There exist matrices $\{\mathcal{M}_2^{(l)}(x,t),\widetilde{\mathcal{M}}_2^{(l)}(x,t)\}_{l=-1}^{+\infty}$ depending on $x$ and $t$ but not on $k$ such that, for any $N \geq -1$,
\begin{align}\label{singRHMat0}
& M(x,t,k) = \sum_{l=-1}^{N} \mathcal{M}_2^{(l)}(x,t)(k-1)^{l} + O((k-1)^{N+1}) \qquad \text{as}\ k \to 1, \ k \in \bar{D}_2, \\
& M(x,t,k) = \sum_{l=-1}^{N} \widetilde{\mathcal{M}}_2^{(l)}(x,t)(k+1)^{l} + O((k+1)^{N+1}) \qquad \text{as}\ k \to -1, \ k \in \bar{D}_2.
\end{align}
Furthermore, there exist scalar coefficients $\alpha, \beta, \gamma, \tilde{\alpha}, \tilde{\beta}, \tilde{\gamma}$ depending on $x$ and $t$, but not on $k$, such that
\begin{align}\nonumber
& \mathcal{M}_{2}^{(-1)}(x,t) = \begin{pmatrix}
\alpha(x,t) & 0 & \beta(x,t) \\
-\alpha(x,t) & 0 & -\beta(x,t) \\
0 & 0 & 0
\end{pmatrix}, & & \mathcal{M}_{2}^{(0)}(x,t) = \begin{pmatrix}
\star & \gamma(x,t) & \star \\
\star & -\gamma(x,t) & \star \\
\star & 0 & \star
\end{pmatrix}, 
	\\ \label{mathcalMcoefficients}
& \widetilde{\mathcal{M}}_{2}^{(-1)}(x,t) = \begin{pmatrix}
\tilde{\alpha}(x,t) & 0 & \tilde{\beta}(x,t) \\
-\tilde{\alpha}(x,t) & 0 & -\tilde{\beta}(x,t) \\
0 & 0 & 0
\end{pmatrix}, & & \widetilde{\mathcal{M}}_{2}^{(0)}(x,t) = \begin{pmatrix}
\star & \tilde{\gamma}(x,t) & \star \\
\star & -\tilde{\gamma}(x,t) & \star \\
\star & 0 & \star
\end{pmatrix}.
\end{align}

\item \label{RHMiteme}
 $M$ satisfies the symmetries $M(x,t, k) = \mathcal{A} M(x,t,\omega k)\mathcal{A}^{-1} = \mathcal{B} M(x,t,\tfrac{1}{k})\mathcal{B}$ and
\begin{align*}
\overline{(M^{A})(x,t,\bar{k})} = \bigg\{ \frac{u(x,t)}{2}\begin{pmatrix}
1 & 1 & 1 \\
1 & 1 & 1 \\
1 & 1 & 1 
\end{pmatrix} + R(k)^{-1} \bigg\}M(x,t,k)R(k), \qquad k \in \C \setminus \Gamma,
\end{align*}
where $M^{A}=(M^{-1})^{T}$ and $R$ is given by \eqref{def of R}. 

\item \label{RHMitemf}
At each point of $\hat{\mathsf{Z}}$, two columns of $M$ are analytic while one column has (at most) a simple pole. Moreover, for each $k_{0}\in \mathsf{Z}\setminus \mathbb{R}$, 
\begin{align}\label{near k0 notin R}
\underset{k = k_0}{\res} [M(x,t,k)]_3 = c_{k_{0}} e^{-\theta_{31}(x,t,k_0)}[M(x,t,k_0)]_1,
\end{align}
and, for each $k_{0} \in \mathsf{Z}\cap \mathbb{R}$,
\begin{align}\label{near k0 in R}
\underset{k = k_0}{\res} [M(x,t,k)]_2 = c_{k_{0}} e^{-\theta_{21}(x,t,k_0)}[M(x,t,k_0)]_1.
\end{align}
\end{enumerate}
\end{RHproblem}

It can be shown that the conditions (\ref{singRHMatinftyb}) make the solution of RH problem \ref{RH problem for M} unique, but we will not need this fact.

\begin{proposition}\label{RHth}
Suppose $\{u(x,t), v(x,t)\}$ is a Schwartz class solution of (\ref{boussinesqsystem}) on $\R \times [0,T)$ with initial data $u_0, v_0 \in \mathcal{S}(\R)$ for some $T \in (0, +\infty]$ such that Assumptions \ref{solitonassumption} and \ref{originassumption} hold. Define $\{r_j(k)\}_1^2$ and $\{c_{k_0}\}_{k_0 \in \mathsf{Z}}$ in terms of $u_0, v_0$ by (\ref{r1r2def}) and \eqref{ck0def}. Define the sectionally meromorphic function $M$ by $M(x,t,k) = M_n(x,t,k)$ for $k \in D_n\setminus \hat{\mathsf{Z}}$.
Then $M(x,t,k)$ satisfies RH problem \ref{RH problem for M} for each $(x,t) \in \R \times [0,T)$ and the formulas 
\begin{align}\label{recoveruv}
\begin{cases}
 \displaystyle{u(x,t) = -i\sqrt{3}\frac{\partial}{\partial x}\lim_{k\to \infty}k\big[(M(x,t,k))_{33} - 1\big] = \frac{1-\omega}{2} \lim_{k\to \infty}k^{2}(M(x,t,k))_{32},}
	\vspace{.1cm}\\
 \displaystyle{v(x,t) = -i\sqrt{3}\frac{\partial}{\partial t}\lim_{k\to \infty}k\big[(M(x,t,k))_{33} - 1\big],}
\end{cases}
\end{align}
 expressing $\{u(x,t), v(x,t)\}$ in terms of $M$ are valid for all $(x,t) \in \R \times [0,T)$.
\end{proposition}
\begin{proof}
The fact that $M$ satisfies \eqref{recoveruv} and properties $(\ref{RHMitema})$--$(\ref{RHMiteme})$ of RH problem \ref{RH problem for M} can be proved as in \cite[Proof of Proposition 3.21]{CLmain}.

It remains to prove that $M$ satisfies property $(\ref{RHMitemf})$ of RH problem \ref{RH problem for M}. We define $X(x,t,k)$ by replacing $u_{0},v_{0}$ by $u(\cdot,t),v(\cdot,t)$ in the definition \eqref{def of X XA Y YA} of $X(x,k)$, and similarly for $Y(x,t,k), X^{A}(x,t,k), Y^{A}(x,t,k), s(k;t), s^{A}(k;t)$, and $w(x,t,k)$. 

We first establish a relation between $s(k;t)$ and $s(k)$. 
Let $u_0^{(i)}(\cdot, t), v_0^{(i)}(\cdot, t)$ be a sequence converging to $u(\cdot, t), v(\cdot, t)$ as in (\ref{uvsequence}), and let $X^{(i)}(x,t,k)$, $Y^{(i)}(x,t,k)$, $X^{A(i)}(x,t,k)$, $Y^{A(i)}(x,t,k)$, $s^{(i)}(k;t)$, $s^{A(i)}(k;t)$ be the associated eigenfunctions and spectral functions.
By replacing $u_{0},v_{0}$ in the statement of \cite[Proposition 3.5]{CLmain} by $(u^{(i)}(\cdot,t), v^{(i)}(\cdot,t))$, we obtain
\begin{align}\label{XYs t}
X^{(i)}(x,t,k)=Y^{(i)}(x,t,k)e^{x  \widehat{\mathcal{L}(k)}}s^{(i)}(k;t).
\end{align}
On the other hand, since $\{u(x,t), v(x,t)\}$ solves \eqref{boussinesqsystem}, the matrices $L$ and $Z$ in (\ref{UVLZdef}) satisfy the compatibility condition $L_{t}-Z_{x}+[L,Z]=0$. Thus $\partial_{t} \hat{X}^{(i)}-Z \hat{X}^{(i)}$ and $\partial_{t} \hat{Y}^{(i)}-Z \hat{Y}^{(i)}$, where $\hat{X}^{(i)}=X^{(i)}e^{\mathcal{L}x+\mathcal{Z}t}$ and $\hat{Y}^{(i)}=Y^{(i)}e^{\mathcal{L}x+\mathcal{Z}t}$, obey the equation $\hat{X}_{x} - \mathcal{L}\hat{X} = \mathsf{U}\hat{X}$. In other words, 
\begin{align*}
& \chi_{X^{(i)}} := (\partial_{t} \hat{X}^{(i)}-Z \hat{X}^{(i)})e^{-\mathcal{L}x-\mathcal{Z}t} = \partial_{t} X^{(i)} - [\mathcal{Z},X^{(i)}]-\mathsf{V}X^{(i)}, \\
& \chi_{Y^{(i)}} := (\partial_{t} \hat{Y}^{(i)}-Z \hat{Y}^{(i)})e^{-\mathcal{L}x-\mathcal{Z}t} = \partial_{t} Y^{(i)} - [\mathcal{Z},Y^{(i)}]-\mathsf{V}Y^{(i)}
\end{align*}
obey the $x$-part in (\ref{Xlax}). Since $\mathsf{V}$ has fast decay as $x \to \pm \infty$, $\chi_{X^{(i)}} \to 0$ as $x \to \infty$ and $\chi_{Y^{(i)}} \to 0$ as $x \to -\infty$. Hence $\chi_{X^{(i)}}$ and $\chi_{Y^{(i)}}$ satisfy the homogeneous versions of the first and third Volterra equations in \eqref{def of X XA Y YA}, respectively. Thus $\chi_{X^{(i)}}=\chi_{Y^{(i)}}=0$, i.e., $X^{(i)}$ and $Y^{(i)}$ also satisfy the $t$-part in (\ref{Xlax}). It follows that
\begin{align*}
X^{(i)}(x,t,k)e^{\mathcal{L}x+\mathcal{Z}t} = Y^{(i)}(x,t,k)e^{\mathcal{L}x+\mathcal{Z}t}Y^{(i)}(0,0,k)^{-1}X^{(i)}(0,0,k) = Y^{(i)}(x,t,k)e^{\mathcal{L}x+\mathcal{Z}t}s^{(i)}(k),
\end{align*}
where we have used \eqref{XYs t} with $x=t=0$ for the last equality. Combining the above with \eqref{XYs t}, we conclude that 
\begin{align*}
s^{(i)}(k;t) = e^{t  \widehat{\mathcal{Z}(k)}}s^{(i)}(k), \qquad k \in \mathbb{C}\setminus \hat{\mathcal{Q}}.
\end{align*}
Taking $i \to \infty$ (and using \cite[Proposition 3.5]{CLmain} to get the domains of definition of the entries of $s$), we arrive at
\begin{align}\label{time evolution of s}
s(k;t) = e^{t  \widehat{\mathcal{Z}(k)}}s(k) \qquad \mbox{for } k \in \begin{pmatrix}
 \omega^2 \hat{\mathcal{S}} & \hat{\Gamma}_{1} & \hat{\Gamma}_{3} \\
 \hat{\Gamma}_{1} & \omega \hat{\mathcal{S}} & \hat{\Gamma}_{5} \\
\hat{\Gamma}_{3} & \hat{\Gamma}_{5} & \hat{\mathcal{S}}
 \end{pmatrix}\setminus \hat{\mathcal{Q}}.
\end{align}
We can show similarly that 
\begin{align}\label{time evolution of sA}
s^{A}(k;t) = e^{t  \widehat{\mathcal{Z}(k)}}s^{A}(k)\qquad \mbox{for } k \in
 \begin{pmatrix}
 -\omega^2 \hat{\mathcal{S}} & \hat{\Gamma}_{4} & \hat{\Gamma}_{6} \\
\hat{\Gamma}_{4} & -\omega \hat{\mathcal{S}} & \hat{\Gamma}_{2} \\
\hat{\Gamma}_{6} & \hat{\Gamma}_{2} & -\hat{\mathcal{S}}
 \end{pmatrix}\setminus \hat{\mathcal{Q}}.
 \end{align}

If $k_0 \in \mathsf{Z}$, then (\ref{time evolution of s}) implies that $s_{11}(k_0;t) = s_{11}(k_0) = 0$ for all $t \in [0,T)$.
On the other hand, by Lemma \ref{M1XYlemma} (with $(u_{0},v_{0})$ replaced by $(u(\cdot,t), v(\cdot,t))$), we have 
\begin{align}\label{M2 at t}
M_2(x,t,k) = \begin{pmatrix} 
X_{11}(x,t,k) & \frac{Y_{12}(x,t,k)}{s_{22}^A(k;t)} & \frac{w_{1}(x,t,k)}{s_{11}(k;t)}  \\
X_{21}(x,t,k) & \frac{Y_{22}(x,t,k)}{s_{22}^A(k;t)} & \frac{w_{2}(x,t,k)}{s_{11}(k;t)}  \\
X_{31}(x,t,k) & \frac{Y_{32}(x,t,k)}{s_{22}^A(k;t)} & \frac{w_{3}(x,t,k)}{s_{11}(k;t)} 
\end{pmatrix}.
\end{align}
Since $[X]_1, [Y]_2, [X^A]_2, [Y^A]_1$ are analytic in $D_2$, this shows that $[M]_1$ is analytic at $k_0$ and that $[M]_{3}$ has (at most) a simple pole at each $k_0 \in \mathsf{Z}$. 

Suppose that $k_0 \in \mathsf{Z} \setminus \R$. Then (\ref{time evolution of sA}) and Assumption \ref{solitonassumption} imply that $s_{22}^A(k_0; t) = s_{22}^A(k_0) \neq 0$, and hence $[M]_2$ is analytic at $k_0$. To prove \eqref{near k0 notin R}, it is therefore enough to show that $\frac{w(x,t,k_{0})}{\dot{s}_{11}(k_0)} = c_{k_0} e^{-\theta_{31}(x,t,k_0)} [X(x,t,k_0)]_1$. This relation holds for $t = 0$ by the definition (\ref{ck0def1}) of $c_{k_0}$. Since the vectors $e^{x l_3(k_0) + t z_3(k_0)} w(x,t,k_{0})$ and $e^{x l_1(k_0)+ t z_1(k_0)} [X]_1(x,t,k_{0})$ satisfy the same ODEs in $x$ and $t$, it holds for all $x$ and $t$. This completes the proof of \eqref{near k0 notin R}.

Suppose now that $k_0 \in \mathsf{Z} \cap \R$. By virtue of (\ref{wzeroatk0}), $[M]_3$ is analytic at $k_0$. To prove \eqref{near k0 in R}, it is therefore enough to show that $\frac{[Y(x,t,k_0)]_2}{\dot{s}_{22}^A(k_0)} = c_{k_0} e^{-\theta_{21}(x,t,k_0)} [X(x,t,k_0)]_1$ for all $(x,t) \in \R \times [0, T)$. This relation holds for $t = 0$ by the definition (\ref{ck0def2}) of $c_{k_0}$. Since the vectors $e^{x l_2(k_0) + t z_2(k_0)} [Y(x,t,k_0)]_2$ and $e^{x l_1(k_0)+ t z_1(k_0)} [X(x,t,k_{0})]_1$ satisfy the same ODEs in $x$ and $t$, it holds for all $x$ and $t$. This completes the proof of \eqref{near k0 in R}.
\end{proof}

\subsection{Construction of $n$}\label{sec: n at t=0}

The following proposition can be proved in the same way as \cite[Proposition 3.22]{CLmain}.

\begin{proposition}[Time evolution of the scattering data]\label{reflectionprop}
Let $T \in (0, \infty]$ and suppose $\{u, v\}$ is a Schwartz class solution of \eqref{boussinesqsystem} on $\R \times [0,T)$ with initial data $u_0, v_0 \in \mathcal{S}(\R)$ such that Assumptions \ref{solitonassumption} and \ref{originassumption} hold.
Let $\{r_j(k)\}_1^2$ be the reflection coefficients associated to $\{u_0, v_0\}$ via (\ref{r1r2def}) and let $\{r_j(k; t)\}_1^2$ be the reflection coefficients associated to $\{u(\cdot,t), v(\cdot, t)\}$. Then
\begin{align}\label{rel between rj at different time}
r_1(k; t) = r_1(k) e^{-\theta_{21}(0,t,k)} \quad \text{and} \quad r_2(k; t) = r_2(k) e^{\theta_{21}(0,t,k)}.
\end{align}
The first and second identities in \eqref{rel between rj at different time} are valid for $k \in \hat{\Gamma}_{1}$ and $k \in \hat{\Gamma}_{4}\setminus \{\omega^{2}, -\omega^{2}\}$, respectively.
\end{proposition}

Define $n =(1,1,1)M$. It follows from (\ref{mathcalMcoefficients}) and the $\mathcal{A}$- and $\mathcal{B}$-symmetries $M(x,t, k) = \mathcal{A} M(x,t,\omega k)\mathcal{A}^{-1} = \mathcal{B} M(x,t,\tfrac{1}{k})\mathcal{B}$ that $n$ has no singularities at the points $\kappa_j$.

\begin{proposition}\label{prop:construction of n}
Let $T \in (0, \infty]$ and suppose $\{u(x,t), v(x,t)\}$ is a Schwartz class solution of \eqref{boussinesqsystem} on $\R \times [0,T)$ with initial data $u_0, v_0 \in \mathcal{S}(\R)$  such that Assumptions \ref{solitonassumption} and \ref{originassumption} hold.
Let $M(x,t,k)$ be defined as in the statement of Proposition \ref{RHth}, and let $n(x,t,k):=(1,1,1)M(x,t,k)$. Then $n$ solves RH problem \ref{RHn}. Moreover,
\begin{align}\nonumber
& u(x,t) := -i\sqrt{3} \frac{\partial}{\partial x} n_3^{(1)}(x,t),
	\\ \label{uvdef}
& v(x,t) := -i\sqrt{3} \frac{\partial}{\partial x}\bigg(n_{3}^{(2)}(x,t) + u(x,t) + \frac{1}{6} \bigg( \int_{\infty}^{x}u(x',t)dx' \bigg)^{2}\bigg),
\end{align}
where $n^{(1)}$ and $n^{(2)}$ are defined through the expansion 
\begin{align}\label{expansion of n at inf}
n(x,t,k) = (1,1,1) + n^{(1)}(x,t) \, k^{-1} + n^{(2)}(x,t) \, k^{-2} + O(k^{-3}) \qquad \mbox{as } k \to \infty.
\end{align}
\end{proposition}
\begin{proof}
The proof is based on Proposition \ref{RHth} and is similar to \cite[Proof of Proposition 3.23]{CLmain}.
\end{proof}

\section{The inverse problem}\label{inversesec}
Instead of working directly with $n$, we will work with an equivalent RH problem, whose solution we denote by $\tilde{n}$. The transformation $n \to \tilde{n}$ replaces the residue conditions \eqref{near k0 notin R} and \eqref{near k0 in R} with jump conditions on small circles. We present the RH problem for $\tilde{n}$ in Section \ref{subsec:RHP ntilde}. In Section \ref{vanishinglemmasubsec}, we establish a vanishing lemma for the RH problem for $\tilde{n}$.  In Section \ref{Dsingsubsec}, we explain where the proof of the vanishing lemma breaks down if $s_{11}$ has zeros in $D_{\mathrm{sing}}$.
The proof of Theorem \ref{inverseth} is then given in Section \ref{inversethsubsec}.

\subsection{The $n \to \tilde{n}$ transformation}\label{subsec:RHP ntilde} The $n \to \tilde{n}$ transformation is defined as follows. For each $k_{0}\in \mathsf{Z}$, we let $D_{\epsilon}(k_{0})$ be a small open disk centered at $k_0$ of radius $\epsilon>0$. We let $D_{\epsilon}(k_{0})^{-1}$, $D_{\epsilon}(k_{0})^*$, and $D_{\epsilon}(k_{0})^{-*}$ be the images of $D_{\epsilon}(k_{0})$ under the maps $k \mapsto k^{-1}$, $k \mapsto \bar{k}$, and $k \mapsto \bar{k}^{-1}$, respectively. If $k_{0} \in \mathbb{R}$, then $D_{\epsilon}(k_{0})^*=D_{\epsilon}(k_{0})$. We assume that $\partial D_{\epsilon}(k_{0})$, $\partial D_{\epsilon}(k_{0})^*$ are oriented counterclockwise, and that $\partial D_{\epsilon}(k_{0})^{-1}$, $\partial D_{\epsilon}(k_{0})^{-*}$ are oriented clockwise. Let 
\begin{align}
& \mathcal{D} = \bigcup_{k_{0}\in \mathsf{Z}}\bigcup_{j=0,1,2} \bigg( \omega^{j} D_{\epsilon}(k_{0}) \cup \omega^{j}  D_{\epsilon}(k_{0})^{-1} \cup \omega^{j} D_{\epsilon}(k_{0})^{*} \cup \omega^{j}  D_{\epsilon}(k_{0})^{-*}\bigg), \label{def of Dcal} \\
& \partial \mathcal{D} = \bigcup_{k_{0}\in \mathsf{Z}}\bigcup_{j=0,1,2} \bigg( \omega^{j} \partial D_{\epsilon}(k_{0}) \cup \omega^{j}  \partial D_{\epsilon}(k_{0})^{-1} \cup \omega^{j} \partial D_{\epsilon}(k_{0})^{*} \cup \omega^{j}  \partial D_{\epsilon}(k_{0})^{-*}\bigg). \label{def of partial Dcal}
\end{align}
Since $k \mapsto k^{-1}$ is a M\"obius transformation, it maps the circle $\partial D_{\epsilon}(k_{0})$ to a circle; hence $\partial \mathcal{D}$ is the union of $6 |\mathsf{Z}\cap \mathbb{R}| + 12 |\mathsf{Z}\setminus \mathbb{R}|$ small circles.
 We choose $\epsilon>0$ sufficiently small such that none of these circles intersect each other, and such that they do not intersect $\Gamma$. Define the contour $\tilde{\Gamma}$ by (see Figure \ref{fig: Gammat})
\begin{align*}
& \tilde{\Gamma} =  \Gamma \cup \partial \mathcal{D}.
\end{align*}
\begin{figure}
\begin{center}
\begin{tikzpicture}[scale=0.9]
\node at (0,0) {};
\draw[black,line width=0.45 mm,->-=0.4,->-=0.85] (0,0)--(30:4);
\draw[black,line width=0.45 mm,->-=0.4,->-=0.85] (0,0)--(90:4);
\draw[black,line width=0.45 mm,->-=0.4,->-=0.85] (0,0)--(150:4);
\draw[black,line width=0.45 mm,->-=0.4,->-=0.85] (0,0)--(-30:4);
\draw[black,line width=0.45 mm,->-=0.4,->-=0.85] (0,0)--(-90:4);
\draw[black,line width=0.45 mm,->-=0.4,->-=0.85] (0,0)--(-150:4);

\draw[black,line width=0.45 mm] ([shift=(-180:2.5cm)]0,0) arc (-180:180:2.5cm);
\draw[black,arrows={-Triangle[length=0.2cm,width=0.18cm]}]
($(3:2.5)$) --  ++(90:0.001);
\draw[black,arrows={-Triangle[length=0.2cm,width=0.18cm]}]
($(57:2.5)$) --  ++(-30:0.001);
\draw[black,arrows={-Triangle[length=0.2cm,width=0.18cm]}]
($(123:2.5)$) --  ++(210:0.001);
\draw[black,arrows={-Triangle[length=0.2cm,width=0.18cm]}]
($(177:2.5)$) --  ++(90:0.001);
\draw[black,arrows={-Triangle[length=0.2cm,width=0.18cm]}]
($(243:2.5)$) --  ++(330:0.001);
\draw[black,arrows={-Triangle[length=0.2cm,width=0.18cm]}]
($(297:2.5)$) --  ++(210:0.001);

\draw[black,line width=0.15 mm] ([shift=(-30:0.55cm)]0,0) arc (-30:30:0.55cm);

\node at (0.8,0) {$\tiny \frac{\pi}{3}$};

\node at (-5.5:2.77) {\footnotesize $\Gamma_8$};
\node at (71:2.72) {\footnotesize $\Gamma_9$};
\node at (112:2.73) {\footnotesize $\Gamma_7$};
\node at (193:2.71) {\footnotesize $\Gamma_8$};
\node at (234:2.71) {\footnotesize $\Gamma_9$};
\node at (311:2.75) {\footnotesize $\Gamma_7$};

\node at (77:1.45) {\footnotesize $\Gamma_1$};
\node at (160:1.45) {\footnotesize $\Gamma_2$};
\node at (-162:1.45) {\footnotesize $\Gamma_3$};
\node at (-77:1.45) {\footnotesize $\Gamma_4$};
\node at (-40:1.45) {\footnotesize $\Gamma_5$};
\node at (43:1.45) {\footnotesize $\Gamma_6$};

\node at (84:3.3) {\footnotesize $\Gamma_4$};
\node at (155:3.3) {\footnotesize $\Gamma_5$};
\node at (-156:3.3) {\footnotesize $\Gamma_6$};
\node at (-84:3.3) {\footnotesize $\Gamma_1$};
\node at (-35:3.3) {\footnotesize $\Gamma_2$};
\node at (35:3.3) {\footnotesize $\Gamma_3$};

% Here we draw 6 circles corresponding to k0 on the real-line:
% ---------------------------start-----------------------------

\draw[blue, line width=0.45 mm] (0:3.3) circle (0.4);
\draw[blue,arrows={-Triangle[length=0.2cm,width=0.18cm]}]
($(0:3.7)+(0,0.15)$) --  ++(90:0.001);

\draw[blue, line width=0.45 mm] (120:3.3) circle (0.4);
\draw[blue,arrows={-Triangle[length=0.2cm,width=0.18cm]}]
($(120:3.7)+(-0.866*0.15,-0.5*0.15)$) --  ++(120+90:0.001);

\draw[blue, line width=0.45 mm] (-120:3.3) circle (0.4);
\draw[blue,arrows={-Triangle[length=0.2cm,width=0.18cm]}]
($(-120:3.7)+(0.866*0.15,-0.5*0.15)$) --  ++(-120+90:0.001);

% Here is a working example taken from internet: DO NOT ERASE
%\draw[red, thick] plot[samples=400,variable=\t,domain=-pi:pi,smooth,line width=0.45 mm] ({cos(\t r)+\t*sin(\t r)},{sin(\t r)-\t*cos(\t r)});

% This is the true inverse of the circle. DO NOT ERASE
%\draw[red,line width=0.45 mm] plot[samples=400,variable=\t,domain=-pi:pi,smooth] ({2.5^2*(3.3+0.4*cos(-\t r))/((3.3+0.4*cos(\t r))^2+(0.4*sin(\t r))^2)},{2.5^2*(0.4*sin(-\t r))/((3.3+0.4*cos(\t r))^2+(0.4*sin(\t r))^2 ) });

% To speed things up, since the above is very close to a circle, we replace it by a circle
\draw[blue, line width=0.45 mm] (0:2.5^2/3.3+0.025) circle (0.25);
\draw[blue,arrows={-Triangle[length=0.2cm,width=0.18cm]}]
($(0:2.5^2/3.3+0.025-0.25)+(0.015,0.12)$) --  ++(91:0.001);

\draw[blue, line width=0.45 mm] (120:2.5^2/3.3+0.025) circle (0.25);
\draw[blue,arrows={-Triangle[length=0.2cm,width=0.18cm]}]
($(120:2.5^2/3.3+0.025-0.25)+(-0.866*0.12,-0.5*0.12)$) --  ++(120+92:0.001);

\draw[blue, line width=0.45 mm] (-120:2.5^2/3.3+0.025) circle (0.25);
\draw[blue,arrows={-Triangle[length=0.2cm,width=0.18cm]}]
($(-120:2.5^2/3.3+0.025-0.25)+(0.866*0.12,-0.5*0.12)$) --  ++(-120+92:0.001);

% ---------------------------end-----------------------------
% We have just finished to draw the 6 circles

% Here we draw 12 circles corresponding to k0 not on the real-line:
% ---------------------------start-----------------------------

\draw[green, line width=0.45 mm] (19:3) circle (0.4);
\draw[green,arrows={-Triangle[length=0.2cm,width=0.18cm]}]
($(19:3)+(0.4,0.15)$) --  ++(90:0.001);
\draw[green, line width=0.45 mm] (-19:3) circle (0.4);
\draw[green,arrows={-Triangle[length=0.2cm,width=0.18cm]}]
($(-19:3)+(0.4,0.15)$) --  ++(90:0.001);

\draw[green, line width=0.45 mm] (120+19:3) circle (0.4);
\draw[green,arrows={-Triangle[length=0.2cm,width=0.18cm]}]
($(120+19:3)+(120:0.4)+(-0.866*0.15,-0.5*0.15)$) --  ++(120+90:0.001);
\draw[green, line width=0.45 mm] (120-19:3) circle (0.4);
\draw[green,arrows={-Triangle[length=0.2cm,width=0.18cm]}]
($(120-19:3)+(120:0.4)+(-0.866*0.15,-0.5*0.15)$) --  ++(120+90:0.001);

\draw[green, line width=0.45 mm] (-120+19:3) circle (0.4);
\draw[green,arrows={-Triangle[length=0.2cm,width=0.18cm]}]
($(-120+19:3)+(-120:0.4)+(0.866*0.15,-0.5*0.15)$) --  ++(-120+90:0.001);
\draw[green, line width=0.45 mm] (-120-19:3) circle (0.4);
\draw[green,arrows={-Triangle[length=0.2cm,width=0.18cm]}]
($(-120-19:3)+(-120:0.4)+(0.866*0.15,-0.5*0.15)$) --  ++(-120+90:0.001);

% This is the true inverse of the circle. DO NOT ERASE
%\draw[red,line width=0.15 mm] plot[samples=400,variable=\t,domain=-pi:pi,smooth] ({2.5^2*(3*cos(-19)+0.4*cos(-\t r))/((3*cos(19)+0.4*cos(\t r))^2+(3*sin(19)+0.4*sin(\t r))^2)},{2.5^2*(3*sin(-19)+0.4*sin(-\t r))/((3*cos(19)+0.4*cos(\t r))^2+(3*sin(19)+0.4*sin(\t r))^2) });

% To speed things up, since the above is very close to a circle, we replace it by a circle

\draw[green, line width=0.45 mm] (-19:2.12) circle (0.29);
\draw[green,arrows={-Triangle[length=0.2cm,width=0.18cm]}]
($(-19:2.12)-(0.28,0)+(0,0.15)$) --  ++(90:0.001);
\draw[green, line width=0.45 mm] (19:2.12) circle (0.29);
\draw[green,arrows={-Triangle[length=0.2cm,width=0.18cm]}]
($(19:2.12)-(0.28,0)+(0,0.15)$) --  ++(90:0.001);

\draw[green, line width=0.45 mm] (-19+120:2.12) circle (0.29);
\draw[green,arrows={-Triangle[length=0.2cm,width=0.18cm]}]
($(-19+120:2.12)-(120:0.28)+(-0.866*0.15,-0.5*0.15)$) --  ++(90+120:0.001);
\draw[green, line width=0.45 mm] (19+120:2.12) circle (0.29);
\draw[green,arrows={-Triangle[length=0.2cm,width=0.18cm]}]
($(19+120:2.12)-(120:0.28)+(-0.866*0.15,-0.5*0.15)$) --  ++(90+120:0.001);

\draw[green, line width=0.45 mm] (-19-120:2.12) circle (0.29);
\draw[green,arrows={-Triangle[length=0.2cm,width=0.18cm]}]
($(-19-120:2.12)-(-120:0.28)+(0.866*0.15,-0.5*0.15)$) --  ++(90-120:0.001);
\draw[green, line width=0.45 mm] (19-120:2.12) circle (0.29);
\draw[green,arrows={-Triangle[length=0.2cm,width=0.18cm]}]
($(19-120:2.12)-(-120:0.28)+(0.866*0.15,-0.5*0.15)$) --  ++(90-120:0.001);

% ---------------------------end-----------------------------
% We have just finished to draw the 12 circles

% Here we draw 6 circles corresponding to k0 on the real-line with k0 < 0:
% ---------------------------start-----------------------------

\draw[red, line width=0.45 mm] (180:1.9) circle (0.4);
\draw[red,arrows={-Triangle[length=0.2cm,width=0.18cm]}]
($(180:1.5)+(0,0.15)$) --  ++(90:0.001);

\draw[red, line width=0.45 mm] (180+120:1.9) circle (0.4);
\draw[red,arrows={-Triangle[length=0.2cm,width=0.18cm]}]
($(180+120:1.5)+(-0.866*0.15,-0.5*0.15)$) --  ++(120+90:0.001);

\draw[red, line width=0.45 mm] (180-120:1.9) circle (0.4);
\draw[red,arrows={-Triangle[length=0.2cm,width=0.18cm]}]
($(180-120:1.5)+(0.866*0.15,-0.5*0.15)$) --  ++(-120+90:0.001);

% Here is a working example taken from internet: DO NOT ERASE
%\draw[red, thick] plot[samples=400,variable=\t,domain=-pi:pi,smooth,line width=0.45 mm] ({cos(\t r)+\t*sin(\t r)},{sin(\t r)-\t*cos(\t r)});

% This is the true inverse of the circle. DO NOT ERASE
%\draw[red,line width=0.45 mm] plot[samples=400,variable=\t,domain=-pi:pi,smooth] ({-2.5^2*(1.9+0.4*cos(-\t r))/((1.9+0.4*cos(\t r))^2+(0.4*sin(\t r))^2)},{2.5^2*(0.4*sin(-\t r))/((1.9+0.4*cos(\t r))^2+(0.4*sin(\t r))^2 ) });

% To speed things up, since the above is very close to a circle, we replace it by a circle
% center = 2.5*(1.9/2.5)^(-1)/(1 - 0.4^2/1.9^2) = 3.44203
% radius = 2.5^2 (0.4/1.9^2)/Abs[1 - 0.4^2/1.9^2] = 0.724638
\draw[red, line width=0.45 mm] (180:3.44203) circle (0.724638);
\draw[red,arrows={-Triangle[length=0.2cm,width=0.18cm]}]
($(180:4.1667)+(0,0.1)$) --  ++(91:0.001);

\draw[red, line width=0.45 mm] (180+120:3.44203) circle (0.724638);
\draw[red,arrows={-Triangle[length=0.2cm,width=0.18cm]}]
($(180+120:4.1667)+(-0.866*0.1,-0.5*0.1)$) --  ++(120+91:0.001);

\draw[red, line width=0.45 mm] (180-120:3.44203) circle (0.724638);
\draw[red,arrows={-Triangle[length=0.2cm,width=0.18cm]}]
($(180-120:4.1667)+(0.866*0.1,-0.5*0.1)$) --  ++(-120+91:0.001);

% ---------------------------end-----------------------------
% We have just finished to draw the 6 circles
\end{tikzpicture}
\end{center}
\begin{figuretext}\label{fig: Gammat}
The contour $\tilde{\Gamma} = \Gamma \cup \partial \mathcal{D}$ in the complex $k$-plane, in a case where $|\mathsf{Z}\cap (1,\infty)| =1$ (i.e. one right-moving soliton, corresponding to blue circles), $|\mathsf{Z}\cap (-1,0)| = 1$ (i.e. one left-moving soliton, corresponding to red circles), $|\mathsf{Z}\cap D_{\mathrm{reg}}^{R}|=1$ (i.e. one right-moving breather, corresponding to green circles), and $|\mathsf{Z}\cap D_{\mathrm{reg}}^{L}| =0$ (no left-moving breather).
\end{figuretext}
\end{figure}
  
The function $\tilde{n}(x,t,k)$ differs from $n(x,t,k)$ only for $k \in \mathcal{D}$, and for $k \in \mathcal{D}$ we define $\tilde{n}(x,t,k)$ as follows. If $k_{0} \in \mathsf{Z}\setminus \mathbb{R}$, we define $\tilde{n}$ for $k \in D_{\epsilon}(k_{0}) \cup D_{\epsilon}(k_{0})^*$ by
\begin{align}\label{def of ntildeC}
\tilde{n}(x,t,k) = \begin{cases} 
n(x,t,k) Q_1(x,t,k), \quad & k \in D_{\epsilon}(k_{0}), \; k_{0} \in \mathsf{Z}\setminus \mathbb{R},
	\\
n(x,t,k) Q_7(x,t,k), & k \in D_{\epsilon}(k_{0})^*, \; k_{0} \in \mathsf{Z}\setminus \mathbb{R},
\end{cases}
\end{align}
where
\begin{align}\label{Q1def}
& Q_1 := \begin{pmatrix} 1 & 0 & -\frac{c_{k_0} e^{-\theta_{31}(x,t,k_0)}}{k- k_0} \frac{k^2-\omega}{k_0^2-\omega} \\ 0 & 1 & 0 \\ 0 & 0 & 1 \end{pmatrix}, \quad
 Q_{7} := \begin{pmatrix} 1 & 0 & 0 \\ 0 & 1 & 0 \\ 0 & -\frac{\bar{c}_{k_0} e^{\theta_{32}(x,t,\bar{k}_0)}}{k- \bar{k}_0} \frac{k^2-1}{\omega^{2}(\omega^{2}-\bar{k}_0^2)}  & 1 \end{pmatrix}, 
\end{align}
and if $k_{0} \in \mathsf{Z} \cap \mathbb{R}$, we define $\tilde{n}$ for $k \in D_{\epsilon}(k_{0})$ by
\begin{align}\label{def of ntildeR}
\tilde{n}(x,t,k) = n(x,t,k) P_1(x,t,k), \qquad k \in D_{\epsilon}(k_{0}), \; k_{0} \in \mathsf{Z}\cap \mathbb{R},
\end{align}
where
\begin{align}\label{P1def}
& P_1(x,t,k) := \begin{pmatrix} 1 & -\frac{c_{k_0} e^{-\theta_{21}(x,t,k_0)}}{k- k_0} \frac{k^2-\omega}{k_0^2-\omega} & 0 \\ 0 & 1 & 0 \\ 0 & 0 & 1 \end{pmatrix}.
\end{align}
We then extend $\tilde{n}$ to all of $\mathcal{D}$ by means of the symmetries
\begin{align}\label{ntsymm}
\tilde{n}(x,t, k) = \tilde{n}(x,t,\omega k)\mathcal{A}^{-1} = \tilde{n}(x,t,k^{-1})\mathcal{B}.
\end{align}

The function $n$ satisfies RH problem \ref{RHn} if and only if $\tilde{n}$ satisfies the following RH problem.
\begin{RHproblem}[RH problem for $\tilde{n}$]\label{RH problem for ntilde}
Find $\tilde{n}(x,t,k)$ with the following properties:
\begin{enumerate}[(a)]
\item $\tilde{n}(x,t,\cdot) : \mathbb{C}\setminus \tilde{\Gamma} \to \mathbb{C}^{1 \times 3}$ is analytic.

\item The limits of $\tilde{n}(x,t,k)$ as $k$ approaches $\tilde{\Gamma}\setminus (\Gamma_\star \cup \mathcal{Q})$ from the left and right exist, are continuous on $\tilde{\Gamma}\setminus (\Gamma_\star \cup \mathcal{Q})$, and satisfy
\begin{align}\label{Mtjumpcondition}
& \tilde{n}_{+}(x,t,k) = \tilde{n}_{-}(x,t,k)v(x,t,k), \qquad k \in \tilde{\Gamma} \setminus (\Gamma_\star \cup \mathcal{Q}),
\end{align}
where the jump matrix $v$ is defined by \eqref{vdef} for $k \in \Gamma$, and $v$ is defined on $\partial \mathcal{D}$ by setting
\begin{align}\label{vregdef}
v(x,t,k) = \begin{cases}
Q_1(x,t,k), & k \in \partial D_{\epsilon}(k_{0}), \; k_{0} \in \mathsf{Z}\setminus \mathbb{R},
    \\
Q_7(x,t,k), & k \in \partial D_{\epsilon}(k_{0})^*, \; k_{0} \in \mathsf{Z}\setminus \mathbb{R},
	\\
P_1(x,t,k), & k \in \partial D_{\epsilon}(k_{0}), \; k_{0} \in \mathsf{Z}\cap \mathbb{R},
\end{cases}
\end{align}
and then extending it to all of $\partial \mathcal{D}$ by means of the symmetries
$$v(x,t,k) = \mathcal{A} v(x,t,\omega k) \mathcal{A}^{-1}
= \mathcal{B}v(x,t,k^{-1})^{-1}\mathcal{B}.$$

\item $\tilde{n}(x,t,k) = O(1)$ as $k\to k_{*}\in \Gamma_{*}$.

\item For $k \in \mathbb{C}\setminus \tilde{\Gamma}$, $\tilde{n}$ obeys the symmetries (\ref{ntsymm}).

\item $\tilde{n}(x,t,k) = (1,1,1) + O(k^{-1})$ as $k\to \infty$.
\end{enumerate}
\end{RHproblem}

\begin{remark}
A more naive choice for the jump matrix $P_1$, which would also remove the pole at $k_{0} \in \mathsf{Z}\cap \mathbb{R}$, would be
$$P_1(x,t,k) =  \begin{pmatrix} 1 &  -\frac{c_{k_0} e^{-\theta_{21}(x,t,k_0)}}{k - k_0} &  0 \\ 0 &  1 &  0 \\ 0 &  0 &  1 \end{pmatrix}.$$
However, the resulting jump matrix $v$ would not obey the $R$-symmetry of Lemma \ref{vsymmlemma} below.
\end{remark}

\subsection{A vanishing lemma}\label{vanishinglemmasubsec}
In the rest of Section \ref{inversesec}, we let $\{r_1(k), r_2(k), \mathsf{Z}, \{c_{k_0}\}_{k_0 \in \mathsf{Z}}\}$ be scattering data satisfying properties $(\ref{Theorem2.3itemi})$--$(\ref{Theorem2.3itemvii})$ of Theorem \ref{directth}. We assume that $T \in (0, \infty]$ is defined by (\ref{Tdef}) and that $v$ is the jump matrix given by (\ref{vdef}) and \eqref{vregdef}.

\begin{lemma}[Complex conjugation symmetry of $v$]\label{vsymmlemma}
The jump matrix $v$ defined in (\ref{vdef}) and \eqref{vregdef} satisfies $(\overline{v(\bar{k})^{-1}})^T = R(k)^{-1}  v(k) R(k)$ for $k \in \partial \mathcal{D} \cup \partial \D \setminus \Gamma_\star$ and $(\overline{v(\bar{k})^{-1}})^T = R(k)^{-1}  v(k)^{-1} R(k)$ for $k \in \Gamma \setminus (\partial \mathcal{D} \cup \partial \D)$. 
\end{lemma}
\begin{proof}
By (\ref{ck0positivitycondition}), we  have that $i(\omega^{2}k_{0}^{2} - \omega)c_{k_{0}} \in \R$ for each $k_0 \in \mathsf{Z}\cap \mathbb{R}$, or, equivalently, that $\bar{c}_{k_0} = -\tilde{r}(k_0) c_{k_0}$ for each $k_{0} \in \mathsf{Z}\cap \mathbb{R}$.
Recalling that $R(k)$ is given by (\ref{def of R}) and that $\partial \mathcal{D}$ is oriented as explained at the beginning of Section \ref{subsec:RHP ntilde}, the statement follows from direct calculations using (\ref{r1r2 relation on the unit circle}) and (\ref{r1r2 relation with kbar symmetry}). 
\end{proof}

The following vanishing lemma will be used to show existence of a solution $n(x,t,k)$ of RH problem \ref{RHn}.

\begin{lemma}[Vanishing lemma]\label{vanishinglemma}
Let $(x,t) \in \R \times [0,T)$. 
Suppose $\tilde{n}^h(x,t,k)$ is a solution of the homogeneous version of RH problem \ref{RH problem for ntilde}, that is, suppose $\tilde{n}^h$ satisfies $(a)$--$(d)$ of RH problem \ref{RH problem for ntilde} together with the homogeneous condition $\tilde{n}^h(x,t,k) = O(k^{-1})$ as $k \to \infty$.
Then $\tilde{n}^h$ vanishes identically. 
\end{lemma}
\begin{proof}
The strategy of the proof is the same as \cite[Proof of Lemma 4.2]{CLmain}, but the presence of solitons makes it technically more complicated. As in \cite[Proof of Lemma 4.2]{CLmain}, we write $\tilde{n}^h(x,t,k) \equiv \mathbf{u}(k) = (u_1(k), u_2(k), u_3(k))$ and $\lambda = (k^3 + k^{-3})/2$. The mapping $k \mapsto \lambda$ is represented in Figure \ref{klambdamapfig}; it is six-to-one except at isolated points.
\begin{figure}
\begin{center}
\begin{tikzpicture}[scale=0.8]
\node at (0,0) {};
\draw[blue,line width=0.65 mm,-<-=0.2,->-=0.8] (0,0)--(0:4);
\draw[red,line width=0.65 mm,->-=0.3,-<-=0.7] (0,0)--(60:4);
\draw[blue,line width=0.65 mm,-<-=0.2,->-=0.8] (0,0)--(120:4);
\draw[red,line width=0.65 mm,->-=0.3,-<-=0.7] (0,0)--(180:4);
\draw[blue,line width=0.65 mm,-<-=0.2,->-=0.8] (0,0)--(-120:4);
\draw[red,line width=0.65 mm,->-=0.3,-<-=0.7] (0,0)--(-60:4);

\draw[black,line width=0.65 mm] ([shift=(-180:2cm)]0,0) arc (-180:-150:2cm);
\draw[green,line width=0.65 mm] ([shift=(-150:2cm)]0,0) arc (-150:-120:2cm);
\draw[green,line width=0.65 mm] ([shift=(-120:2cm)]0,0) arc (-120:-90:2cm);
\draw[black,line width=0.65 mm] ([shift=(-90:2cm)]0,0) arc (-90:-60:2cm);
\draw[black,line width=0.65 mm] ([shift=(-60:2cm)]0,0) arc (-60:-30:2cm);
\draw[green,line width=0.65 mm] ([shift=(-30:2cm)]0,0) arc (-30:0:2cm);
\draw[green,line width=0.65 mm] ([shift=(0:2cm)]0,0) arc (0:30:2cm);
\draw[black,line width=0.65 mm] ([shift=(30:2cm)]0,0) arc (30:60:2cm);
\draw[black,line width=0.65 mm] ([shift=(60:2cm)]0,0) arc (60:90:2cm);
\draw[green,line width=0.65 mm] ([shift=(90:2cm)]0,0) arc (90:120:2cm);
\draw[green,line width=0.65 mm] ([shift=(120:2cm)]0,0) arc (120:150:2cm);
\draw[black,line width=0.65 mm] ([shift=(150:2cm)]0,0) arc (150:180:2cm);
\draw[green,arrows={-Triangle[length=0.27cm,width=0.22cm]}]
($(10:2)$) --  ++(-75:0.001);
\draw[black,arrows={-Triangle[length=0.27cm,width=0.22cm]}]
($(40:2)$) --  ++(-45:0.001);
\draw[black,arrows={-Triangle[length=0.27cm,width=0.22cm]}]
($(80:2)$) --  ++(165:0.001);
\draw[green,arrows={-Triangle[length=0.27cm,width=0.22cm]}]
($(110:2)$) --  ++(195:0.001);
\draw[green,arrows={-Triangle[length=0.27cm,width=0.22cm]}]
($(130:2)$) --  ++(45:0.001);
\draw[black,arrows={-Triangle[length=0.27cm,width=0.22cm]}]
($(160:2)$) --  ++(75:0.001);
\draw[black,arrows={-Triangle[length=0.27cm,width=0.22cm]}]
($(200:2)$) --  ++(-75:0.001);
\draw[green,arrows={-Triangle[length=0.27cm,width=0.22cm]}]
($(230:2)$) --  ++(-45:0.001);
\draw[green,arrows={-Triangle[length=0.27cm,width=0.22cm]}]
($(250:2)$) --  ++(165:0.001);
\draw[black,arrows={-Triangle[length=0.27cm,width=0.22cm]}]
($(280:2)$) --  ++(195:0.001);
\draw[black,arrows={-Triangle[length=0.27cm,width=0.22cm]}]
($(320:2)$) --  ++(45:0.001);
\draw[green,arrows={-Triangle[length=0.27cm,width=0.22cm]}]
($(350:2)$) --  ++(75:0.001);

\node at (45:2.35) {\small{$10$}};

\node at (15:2.3) {\small{$9$}};

\node at (-15:2.3) {\small{$8$}};

\node at (-45:2.3) {\small{$7$}};

\node at (-5:3) {\small $1$};
\node at (55:3) {\small $2$};
\node at (115:3) {\small $3$};
\node at (175:3) {\small $4$};
\node at (235:3) {\small $5$};
\node at (295:3) {\small $6$};

\node at (30:3.5) {\small $U_1$};
\node at (90:3.5) {\small $U_2$};
\node at (150.5:3.5) {\small $U_3$};
\node at (210:3.5) {\small $U_4$};
\node at (-90:3.5) {\small $U_5$};
\node at (-30:3.5) {\small $U_6$};

\node at (30:1.3) {\small $U_7$};
\node at (90:1.3) {\small $U_8$};
\node at (151.3:1.3) {\small $U_9$};
\node at (210:1.3) {\small $U_{10}$};
\node at (-90:1.3) {\small $U_{11}$};
\node at (-30:1.3) {\small $U_{12}$};

\node at (0.6:4.5) {\small $\Gamma^h$};

\end{tikzpicture}
\\
\vspace{0.5cm}
\begin{tikzpicture}
\node at (0,0) {};
\draw[red,line width=0.65 mm,->-=0.6] (-4,0)--(-1,0);
\draw[black,line width=0.65 mm,->-=0.7] (-1,0)--(0,0);
\draw[green,line width=0.65 mm,->-=0.7] (0,0)--(1,0);
\draw[blue,line width=0.65 mm,->-=0.6] (1,0)--(4,0);

\node at (-1,-0.3) {\small $-1$};
\node at (0,-0.3) {\small $0$};
\node at (1,-0.3) {\small $1$};
\node at (4.3,0) {\small $\R$};
\end{tikzpicture}

\end{center}
\begin{figuretext}\label{klambdamapfig}
The map $k \mapsto \lambda = \frac{k^3 + k^{-3}}{2}$ maps the contour $\Gamma^h$ to the real axis. In particular, the unit circle is mapped to the interval $[-1, 1]$. 
\end{figuretext}
\end{figure}
Let $\Gamma^h$ be the contour in Figure \ref{klambdamapfig}, and let $\Gamma^h_j$ be the subcontour labeled by $j$ in Figure \ref{klambdamapfig}.
For $\lambda < -1$, let $k_2^h(\lambda) \in \Gamma_2^h$ be the unique solution of $\lambda = \frac{k^3 + k^{-3}}{2}$ in $\Gamma_2^h$.
Let $\delta_3:\C \setminus (-\infty,-1] \to \C$ be given by
\begin{align}\label{d3def}
\delta_3(\lambda) = e^{\frac{1}{2\pi i} \int_{-\infty}^{-1} \ln\big(\omega^2 \tilde{r}(\omega k_2^h(\lambda'))\big) \frac{d\lambda'}{\lambda' - \lambda}}, \qquad \lambda \in \C \setminus (-\infty,-1],
\end{align}
where the branch of the logarithm is fixed by the requirement that
\begin{align}\label{d3integrandpositive}
\frac{1}{2\pi i} \ln(\omega^2 \tilde{r}(\omega k_2^h(\lambda))) = \frac{1}{2\pi} \arctan\bigg(\frac{\sqrt{3} \left(2 r^2+1\right)}{2 r^4+2 r^2-1}\bigg) \in (0, 1/6)
\end{align}
for $\lambda < -1$.
Let $\delta_2: \C \setminus \Gamma^{h} \to \C$ and $F:(U_1 \cup U_6)\setminus (\Gamma\cup \partial \mathcal{D}) \to \C$ by
\begin{align}\label{d2def}
& \delta_2(k) := \tilde{r}(\omega^2 k) \delta_3(\lambda), \qquad k \in \C \setminus \Gamma^{h},
	\\ \nonumber
& F(k) := \begin{cases}
\mathbf{u}(k) \Delta_3(k) \overline{\mathbf{u}(\bar{k})}^T, \qquad k \in U_1 \setminus (\Gamma\cup \partial \mathcal{D}),
	\\
\mathbf{u}(k) \Delta_2(k) \overline{\mathbf{u}(\bar{k})}^T, \qquad k \in U_6 \setminus (\Gamma\cup \partial \mathcal{D}),
\end{cases}
\end{align}
where $U_1$ and $U_6$ are the open regions displayed in Figure \ref{klambdamapfig} and
\begin{align}\label{D2D3def}
\Delta_2(k) = \begin{pmatrix} 0 & \delta_2(k) & 0 \\ 0 & 0 & 0 \\ 0 & 0 & 0 \end{pmatrix}, \qquad 
\Delta_3(k) = \begin{pmatrix} 0 & \delta_2(k) & 0 \\ 0 & 0 & 0 \\ 0 & 0 & \delta_3(\lambda) \end{pmatrix}.
\end{align}
In the same way as in \cite[Proof of Lemma 4.2]{CLmain}, we can prove that $F$ extends to an analytic function $F:(U_1 \cup U_6)\setminus \partial \mathcal{D} \to \C$. 
We extend $F$ to $\C \setminus (\Gamma^h\cup \partial \mathcal{D})$ by means of the symmetries
\begin{align}\label{symm of F}
F(k) = F(\omega k), \qquad F(k) = F(1/k).
\end{align}
These symmetries imply that $F$ can be viewed as a function of $\lambda$ for $\lambda \in \C \setminus (\R \cup \lambda(\partial \mathcal{D}))$. Note that $\lambda(\partial \mathcal{D})$ is the union of $|\mathsf{Z}\cap \mathbb{R}| + 2 |\mathsf{Z}\setminus \mathbb{R}|$ small closed loops, and that $\lambda(\partial D_{\epsilon}(k_{0}))$ is oriented counterclockwise for all $k_{0}\in \mathsf{Z}$. Define the analytic function $G:\C \setminus \R \to \C$ by
$$G(\lambda) := \frac{F(k)}{(\lambda + 1)^{2/3}}, \qquad \lambda \in \C \setminus \R,$$
where the principal branch is taken for the root. The same computations as in \cite[Proof of Lemma 4.2]{CLmain} show that
\begin{align}
& G_+(\lambda) - G_-(\lambda) = \frac{\delta_3(\lambda)}{|\lambda + 1|^{2/3}} |u_3(k)|^2, & & \lambda >1, \; k=k_{1}^{h}(\lambda) \in \Gamma_{1}^{h}, \label{Gjump1} \\
& G_+(\lambda) - G_-(\lambda) = \frac{ \omega^2 \delta_{2+}(k) }{|\lambda + 1|^{2/3}} |u_1(k)|^2, & & \lambda <-1, \; k=k_{2}^{h}(\lambda) \in \Gamma_{2}^{h}, \label{Gjump2} \\
& G_+(\lambda) - G_-(\lambda) = \frac{1}{|\lambda + 1|^{2/3}} \mathbf{u}^2(k) Q_1(k) \overline{\mathbf{u}^2(k)}^T, & & \lambda \in (0,1), \; k=k_{9}^{h}(\lambda) \in \Gamma_{9}^{h}, \label{Gjump3} \\
& G_+(\lambda) - G_-(\lambda) = \frac{1}{|\lambda +1|^{2/3}} \mathbf{u}^6(k) Q_2(k) \overline{\mathbf{u}^6(k)}^T, & & \lambda \in (-1,0), \; k=k_{10}^{h}(\lambda) \in \Gamma_{10}^{h}, \label{Gjump4} 
\end{align}
where $Q_1(k) := \Delta_3(k) R(k)^{-1}  v_8(k) R(k) \mathcal{B}
- v_8(k) \mathcal{B} \Delta_2(1/k)$ is positive definite for $k\in \Gamma_{9}^{h}$ and $Q_2(k) := v_9(k) \Delta_3(k) \mathcal{B} - \mathcal{B} \Delta_2(1/k) \mathcal{B}R(k)^{-1}  v_9(k) R(k)\mathcal{B}$ is positive definite for $k\in \Gamma_{10}^{h}$. Furthermore, $\delta_3(\lambda)>0$ for $\lambda>1$ and $\omega^2 \delta_{2+}(k)>0$ for $k \in \Gamma_{2}^{h}$. In particular the right-hand sides of \eqref{Gjump1}--\eqref{Gjump4} are $\geq 0$. 

It remains to compute the jumps of $G$ on $\partial \mathcal{D}$.
We first consider the jumps associated with $k_{0}\in \mathsf{Z}\setminus \R$. We will use that
\begin{align}
v(x,t,k) = \begin{cases}
Q_1(x,t,k), & k \in \partial D_{\epsilon}(k_{0}), \; k_{0} \in \mathsf{Z}\setminus \mathbb{R},
	\\
Q_2(x,t,k), & k \in \omega \partial D_{\epsilon}(k_{0}), \; k_{0} \in \mathsf{Z}\setminus \mathbb{R},
	\\
Q_{5}(x,t,k), & k \in \partial(\omega D_{\epsilon}(k_{0}))^{-1}, \; k_{0} \in \mathsf{Z}\setminus \mathbb{R},
    \\
Q_7(x,t,k), & k \in \partial D_{\epsilon}(k_{0})^*, \; k_{0} \in \mathsf{Z}\setminus \mathbb{R},
	\\
Q_{11}(x,t,k), & k \in (\partial(\omega D_{\epsilon}(k_{0}))^{-1})^*, \; k_{0} \in \mathsf{Z}\setminus \mathbb{R},
\end{cases}
\end{align}
where $Q_1, Q_7$ are given by (\ref{Q1def}) and
\begin{align*}
& Q_2(x,t,k) := \mathcal{A}^{-1} Q_1(x,t,\omega^2 k)\mathcal{A}
= \begin{pmatrix} 1 & 0 & 0 \\ 0 & 1 & 0 \\ 0 & -\frac{c_{k_0} e^{\theta_{32}(x,t,\omega k_0)}}{k- \omega k_0} \frac{\omega^2(k^2-1)}{k_0^2-\omega} & 1 \end{pmatrix},
	\\
& Q_5(x,t,k)  := \mathcal{B} Q_2(x,t,k^{-1})^{-1}\mathcal{B}
= \begin{pmatrix} 1 &  0 &  0 \\ 0 &  1 &  0 \\ -\frac{k_{0}^{-1}k c_{k_{0}} e^{\theta_{31}(x,t,\omega^2 k_0^{-1})}}{k- \omega^2 k_0^{-1}} \frac{\omega k^{-2}-\omega}{k_0^2-\omega} &  0 &  1 \end{pmatrix},
	\\
& Q_{11}(x,t,k)  := R(k) (\overline{Q_5(x,t,\bar{k})^{-1}})^{T}R(k)^{-1}
= \begin{pmatrix} 1 &  0 &  0 \\ 0 &  1 &  \frac{\bar{k}_{0}^{-1}k\bar{c}_{k_0} e^{-\theta_{32}(x,t,\omega \bar{k}_0^{-1})}}{k- \omega\bar{k}_0^{-1}} \frac{1-\omega^{2}k^{-2}}{\omega^{2}(\omega^{2}-\bar{k}_0^2)} \\ 0 &  0  &  1 \end{pmatrix}.
\end{align*}

\subsubsection*{Jump for $\lambda \in \lambda(\partial D_{\epsilon}(k_{0}))$, $k_{0}\in \mathsf{Z}\cap D_{\mathrm{reg}}^{R}$}
Recall that both $\partial D_{\epsilon}(k_{0})$ and $\partial D_{\epsilon}(k_{0})^{*}$ are oriented counterclockwise. Letting $k \in \partial D_{\epsilon}(k_{0})$ correspond to $\lambda \in \lambda(\partial D_{\epsilon}(k_{0}))$, we find
\begin{align}
G_+(\lambda) - G_-(\lambda) 
& = \frac{1}{(\lambda+1)^{2/3}}\Big(\mathbf{u}_+(k) \Delta_3(k) \overline{\mathbf{u}_+(\bar{k})}^T
- \mathbf{u}_-(k) \Delta_3(k) \overline{\mathbf{u}_-(\bar{k})}^T \Big) \nonumber \\
& = \frac{\mathbf{u}_-(k)}{(\lambda+1)^{2/3}} \Big(Q_{1}(k) \Delta_3(k) \overline{Q_{7}(\bar{k})}^T  
- \Delta_3(k) \Big)\overline{\mathbf{u}_-(\bar{k})}^T = 0, \label{lol14}
\end{align}
where we have used \eqref{d2def} for the last equality.
\subsubsection*{Jump for $\lambda \in \lambda(\partial D_{\epsilon}(k_{0})^{*})$, $k_{0}\in \mathsf{Z}\cap D_{\mathrm{reg}}^{R}$} Letting $k \in \partial D_{\epsilon}(k_{0})^{*}$ correspond to $\lambda \in \lambda(\partial D_{\epsilon}(k_{0})^{*})$, we find
\begin{align}
G_+(\lambda) - G_-(\lambda) 
& = \frac{1}{(\lambda+1)^{2/3}}\Big(\mathbf{u}_+(k) \Delta_2(k) \overline{\mathbf{u}_+(\bar{k})}^T
- \mathbf{u}_-(k) \Delta_2(k) \overline{\mathbf{u}_-(\bar{k})}^T\Big) \nonumber \\
& = \frac{\mathbf{u}_-(k)}{(\lambda+1)^{2/3}} \Big(Q_{7}(k) \Delta_2(k) \overline{Q_{1}(\bar{k})}^T  
- \Delta_2(k) \Big)\overline{\mathbf{u}_-(\bar{k})}^T = 0. \label{lol15}
\end{align}
\subsubsection*{Jump for $\lambda \in \lambda(\partial D_{\epsilon}(k_{0}))$, $k_{0}\in \mathsf{Z}\cap D_{\mathrm{reg}}^{L}$}
Recall that $\partial D_{\epsilon}(k_{0})$ is oriented counterclockwise while $\omega^{2}\partial D_{\epsilon}(k_{0})^{-1}$ is oriented clockwise. Letting $k \in \partial D_{\epsilon}(k_{0})$ correspond to $\lambda \in \lambda(\partial D_{\epsilon}(k_{0}))$, we find
\begin{align*}
G_+(\lambda) - G_-(\lambda) & = \frac{F_{+}(k)-F_{-}(k)}{(\lambda + 1)^{2/3}} = \frac{F_{-}(\frac{1}{\omega k})-F_{+}(\frac{1}{\omega k})}{(\lambda + 1)^{2/3}} \\
& = \frac{1}{(\lambda+1)^{2/3}}\Big(\mathbf{u}_-(\tfrac{1}{\omega k}) \Delta_3(\tfrac{1}{\omega k}) \overline{\mathbf{u}_-(\tfrac{1}{\,\overline{\omega k}\,})}^T
- \mathbf{u}_+(\tfrac{1}{\omega k}) \Delta_3(\tfrac{1}{\omega k}) \overline{\mathbf{u}_+(\tfrac{1}{\,\overline{\omega k}\,})}^T \Big) \nonumber \\
& = \frac{\mathbf{u}_-(\tfrac{1}{\omega k})}{(\lambda+1)^{2/3}} \Big( \Delta_3(\tfrac{1}{\omega k}) 
- Q_{5}(\tfrac{1}{\omega k})\Delta_3(\tfrac{1}{\omega k})\overline{Q_{11}(\tfrac{1}{\,\overline{\omega k}\,})}^T   \Big)\overline{\mathbf{u}_-(\tfrac{1}{\,\overline{\omega k}\,})}^T = 0,
\end{align*}
where for the last equality we have used \eqref{d2def}.
\subsubsection*{Jump for $\lambda \in \lambda(\partial D_{\epsilon}(k_{0})^{*})$, $k_{0}\in \mathsf{Z}\cap D_{\mathrm{reg}}^{L}$} Letting $k \in \partial D_{\epsilon}(k_{0})^{*}$ correspond to $\lambda \in \lambda(\partial D_{\epsilon}(k_{0})^{*})$, we find
\begin{align*}
G_+(\lambda) - G_-(\lambda) & = \frac{F_{+}(k)-F_{-}(k)}{(\lambda + 1)^{2/3}} = \frac{F_{-}(\frac{1}{\omega k})-F_{+}(\frac{1}{\omega k})}{(\lambda + 1)^{2/3}} \\
& = \frac{1}{(\lambda+1)^{2/3}}\Big(\mathbf{u}_-(\tfrac{1}{\omega k}) \Delta_2(\tfrac{1}{\omega k}) \overline{\mathbf{u}_-(\tfrac{1}{\,\overline{\omega k}\,})}^T
- \mathbf{u}_+(\tfrac{1}{\omega k}) \Delta_2(\tfrac{1}{\omega k}) \overline{\mathbf{u}_+(\tfrac{1}{\,\overline{\omega k}\,})}^T\Big) \nonumber \\
& = \frac{\mathbf{u}_-(\frac{1}{\omega k})}{(\lambda+1)^{2/3}} \Big(\Delta_2(\tfrac{1}{\omega k})   
- Q_{11}(\tfrac{1}{\omega k})\Delta_2(\tfrac{1}{\omega k})\overline{Q_{5}(\tfrac{1}{\,\overline{\omega k}\,})}^T \Big)\overline{\mathbf{u}_-(\tfrac{1}{\,\overline{\omega k}\,})}^T = 0.
\end{align*}

We now consider the jumps associated with $k_{0}\in \mathsf{Z}\cap \R$. We will use that
\begin{align}
v(x,t,k) = \begin{cases}
P_1(x,t,k), & k \in \partial D_{\epsilon}(k_{0}), \; k_{0} \in \mathsf{Z}\cap \mathbb{R},
	\\
P_{5}(x,t,k), & k \in \partial(\omega D_{\epsilon}(k_{0}))^{-1}, \; k_{0} \in \mathsf{Z}\cap \mathbb{R},
	\\
P_{6}(x,t,k), & k \in \partial(\omega^{2}D_{\epsilon}(k_{0}))^{-1}, \; k_{0} \in \mathsf{Z}\cap \mathbb{R}, \end{cases}
\end{align}
where $P_1$ is given by (\ref{P1def}) and
\begin{align*}
& P_{5}(x,t,k) := \mathcal{B} \mathcal{A}^{-1} P_1\bigg(x,t,\frac{1}{\omega k}\bigg)^{-1} \mathcal{A} \mathcal{B}
= \begin{pmatrix} 1 &  0 &  0 \\ 0 &  1 &  0 \\ 0 &  -\frac{k_0^{-1} k c_{k_0} e^{\theta_{32}(x,t,\omega^2 k_0^{-1})}}{k- \omega^{2} k_0^{-1}} \frac{\omega k^{-2}-\omega}{k_0^2-\omega} &  1 \end{pmatrix}, 
	\\
& P_{6}(x,t,k) := \mathcal{B} \mathcal{A} P_1\bigg(x,t,\frac{1}{\omega^2 k}\bigg)^{-1} \mathcal{A}^{-1} \mathcal{B}
= \begin{pmatrix} 1 &  0 &  -\frac{k_0^{-1} k c_{k_0} e^{-\theta_{31}(x,t,\omega k_0^{-1})}}{k- \omega k_0^{-1}} \frac{\omega^{2} k^{-2}-\omega}{k_0^2-\omega} \\ 0 &  1 &  0 \\ 0 &  0 &  1 \end{pmatrix}.
\end{align*}
We use the notation $E_{k_{0}}(x,t)  := c_{k_0} e^{-\theta_{21}(x,t,k_0)}$ and note that $e^{-\theta_{21}(x,t,k_0)} > 0$ for $k_0 \in \R$.

\subsubsection*{Jump for $\lambda \in \lambda(\partial D_{\epsilon}(k_{0}))$, $k_{0}\in \mathsf{Z}\cap (1,\infty)$, $\im \lambda >0$}
Recall that $\partial D_{\epsilon}(k_{0})$ is oriented counterclockwise. Hence, letting $k \in \partial D_{\epsilon}(k_{0})$, $\im k >0$ correspond to $\lambda \in \lambda(\partial D_{\epsilon}(k_{0}))$, $\im \lambda >0$, we find
\begin{align}
G_+(\lambda) - G_-(\lambda) & = \frac{1}{(\lambda+1)^{2/3}}\Big(\mathbf{u}_+(k) \Delta_3(k) \overline{\mathbf{u}_+(\bar{k})}^T
- \mathbf{u}_-(k) \Delta_3(k) \overline{\mathbf{u}_-(\bar{k})}^T \Big)
	\nonumber \\
& = \frac{\mathbf{u}_-(k)}{(\lambda+1)^{2/3}} \Big(P_{1}(k) \Delta_3(k) \overline{P_{1}(\bar{k})}^T  
- \Delta_3(k) \Big)\overline{\mathbf{u}_-(\bar{k})}^T \nonumber \\
& = \frac{-E_{k_{0}}(x,t) \delta_{2}(k)(k^{2}-\omega^{2})}{(\lambda+1)^{2/3}(k-k_{0})(1-\omega^{2}k_{0}^{2})}u_{1}(k)\overline{u_{1}(\overline{k})}, \label{lol8}
\end{align}
where for the last equality we have used \eqref{d2def} and the fact that $u_{1,-}=u_{1,+}$ on $\partial D_{\epsilon}(k_{0})$.

\subsubsection*{Jump for $\lambda \in \lambda(\partial D_{\epsilon}(k_{0}))$, $k_{0}\in \mathsf{Z}\cap (1,\infty)$, $\im \lambda <0$}
For $k \in \partial D_{\epsilon}(k_{0})$, $\im k<0$,
\begin{align}
G_+(\lambda) - G_-(\lambda) 
& = \frac{1}{(\lambda+1)^{2/3}}\Big(\mathbf{u}_+(k) \Delta_2(k) \overline{\mathbf{u}_+(\bar{k})}^T
- \mathbf{u}_-(k) \Delta_2(k) \overline{\mathbf{u}_-(\bar{k})}^T\Big) \nonumber \\
& = \frac{\mathbf{u}_-(k)}{(\lambda+1)^{2/3}} \Big(P_{1}(k) \Delta_2(k) \overline{P_{1}(\bar{k})}^T  
- \Delta_2(k) \Big)\overline{\mathbf{u}_-(\bar{k})}^T \nonumber \\
& = \frac{-E_{k_{0}}(x,t) \delta_{2}(k)(k^{2}-\omega^{2})}{(\lambda+1)^{2/3}(k-k_{0})(1-\omega^{2}k_{0}^{2})}u_{1}(k)\overline{u_{1}(\overline{k})}. \label{lol9}
\end{align}

\subsubsection*{Jump for $\lambda \in \lambda(\partial D_{\epsilon}(k_{0}))$, $k_{0}\in \mathsf{Z}\cap (-1,0)$, $\im \lambda >0$}
Letting $k \in \partial D_{\epsilon}(k_{0})$, $\im k<0$, correspond to $\lambda \in \lambda(\partial D_{\epsilon}(k_{0}))$, $\im \lambda >0$, we find
\begin{align}
G_+(\lambda) - G_-(\lambda) & = \frac{F_{+}(k)-F_{-}(k)}{(\lambda + 1)^{2/3}} = \frac{F_{-}(\frac{1}{\omega k})-F_{+}(\frac{1}{\omega k})}{(\lambda + 1)^{2/3}} \nonumber \\
& = \frac{1}{(\lambda+1)^{2/3}}\Big(\mathbf{u}_-(\tfrac{1}{\omega k}) \Delta_3(\tfrac{1}{\omega k}) \overline{\mathbf{u}_-(\tfrac{1}{\,\overline{\omega k}\,})}^T - \mathbf{u}_+(\tfrac{1}{\omega k}) \Delta_3(\tfrac{1}{\omega k}) \overline{\mathbf{u}_+(\tfrac{1}{\,\overline{\omega k}\,})}^T \Big) \nonumber \\
& = \frac{\mathbf{u}_-(\tfrac{1}{\omega k})}{(\lambda+1)^{2/3}} \Big( \Delta_3(\tfrac{1}{\omega k})   
- P_{5}(\tfrac{1}{\omega k})\Delta_3(\tfrac{1}{\omega k})\overline{P_{6}(\tfrac{1}{\,\overline{\omega k}\,})}^T \Big)\overline{\mathbf{u}_-(\tfrac{1}{\,\overline{\omega k}\,})}^T \nonumber \\
& = \frac{-E_{k_{0}}(x,t) \delta_{2}(\frac{1}{\omega k})(k^{2}-\omega^{2})(\omega - k^{2})}{(\lambda+1)^{2/3}(k-k_{0})(1-\omega^{2}k_{0}^{2})\omega^{2}(1-k^{2})}u_{3}(\tfrac{1}{\omega k})\overline{u_{1}(\tfrac{1}{\,\overline{\omega k}\,})}, \label{lol8 bis}
\end{align}
where in the last step we have used that $u_{3,-}=u_{3,+}$ on $(\omega\partial D_{\epsilon}(k_{0}))^{-1}$ and that $u_{1,-}=u_{1,+}$ on $(\omega^{2}\partial D_{\epsilon}(k_{0}))^{-1}$.

\subsubsection*{Jump for $\lambda \in \lambda(\partial D_{\epsilon}(k_{0}))$, $k_{0}\in \mathsf{Z}\cap (-1,0)$, $\im \lambda <0$}
Letting $k \in \partial D_{\epsilon}(k_{0})$, $\im k >0$ correspond to $\lambda \in \lambda(\partial D_{\epsilon}(k_{0}))$, $\im \lambda <0$, we find
\begin{align}
G_+(\lambda) - G_-(\lambda) & = \frac{F_{+}(k)-F_{-}(k)}{(\lambda + 1)^{2/3}} = \frac{F_{-}(\frac{1}{\omega^{2} k})-F_{+}(\frac{1}{\omega^{2} k})}{(\lambda + 1)^{2/3}} \nonumber \\
& = \frac{1}{(\lambda+1)^{2/3}}\Big(\mathbf{u}_-(\tfrac{1}{\omega^{2} k}) \Delta_2(\tfrac{1}{\omega^{2} k}) \overline{\mathbf{u}_-(\tfrac{1}{\,\overline{\omega^{2} k}\,})}^T
- \mathbf{u}_+(\tfrac{1}{\omega^{2} k}) \Delta_2(\tfrac{1}{\omega^{2} k}) \overline{\mathbf{u}_+(\tfrac{1}{\,\overline{\omega^{2} k}\,})}^T \Big) \nonumber \\
& = \frac{\mathbf{u}_-(\tfrac{1}{\omega^{2} k})}{(\lambda+1)^{2/3}} \Big( \Delta_2(\tfrac{1}{\omega^{2} k})   
- P_{6}(\tfrac{1}{\omega^{2} k})\Delta_2(\tfrac{1}{\omega^{2} k})\overline{P_{5}(\tfrac{1}{\,\overline{\omega^{2} k}\,})}^T \Big)\overline{\mathbf{u}_-(\tfrac{1}{\,\overline{\omega^{2} k}\,})}^T \nonumber \\
& = \frac{-E_{k_{0}}(x,t) \delta_{2}(\tfrac{1}{\omega^{2} k})(k^{2} - \omega^{2})}{(\lambda+1)^{2/3}(k-k_{0})(1-\omega^{2}k_{0}^{2})}u_{1}(\tfrac{1}{\omega^{2} k})\overline{u_{3}(\tfrac{1}{\,\overline{\omega^{2} k}\,})}, \label{lol9 bis}
\end{align}
where for the last equality we have used \eqref{d2def} and the fact that $u_{1,-}=u_{1,+}$ on $(\omega^{2}\partial D_{\epsilon}(k_{0}))^{-1}$ and that $u_{3,-}=u_{3,+}$ on $(\omega\partial D_{\epsilon}(k_{0}))^{-1}$.

Equations \eqref{lol8 bis} and \eqref{lol9 bis} can be further simplified. Using the $\mathcal{A}$-symmetry in \eqref{ntsymm}, we infer that
\begin{align*}
\overline{u_{3}(\tfrac{1}{\,\overline{\omega^{2} k}\,})} = \overline{u_{1}(\tfrac{1}{\, \omega^{2} \overline{k}\,})}, \qquad u_{3}(\tfrac{1}{\,\omega k\,}) = u_{1}(\tfrac{1}{\, \omega^{2} k\,}),
\end{align*}
and therefore \eqref{lol8 bis} and \eqref{lol9 bis} become
\begin{align}
& G_+(\lambda) - G_-(\lambda) = \frac{-E_{k_{0}}(x,t) \delta_{2}(\frac{1}{\omega k})(k^{2}-\omega^{2})(\omega - k^{2})}{(\lambda+1)^{2/3}(k-k_{0})(1-\omega^{2}k_{0}^{2})\omega^{2}(1-k^{2})}u_{1}(\tfrac{1}{\, \omega^{2} k\,})\overline{u_{1}(\tfrac{1}{\,\omega^{2} \overline{k}\,})},  \label{lol8 bis bis} \\
& G_+(\lambda) - G_-(\lambda) = \frac{-E_{k_{0}}(x,t) \delta_2(\tfrac{1}{\omega^2 k})(k^{2} - \omega^{2})}{(\lambda+1)^{2/3}(k-k_{0})(1-\omega^{2}k_{0}^{2})}u_{1}(\tfrac{1}{\omega^{2} k})\overline{u_{1}(\tfrac{1}{\, \omega^{2} \overline{k}\,})}, \label{lol9 bis bis}
\end{align}
respectively. Recall that \eqref{lol8 bis bis} is valid for $\lambda \in \lambda(\partial D_{\epsilon}(k_{0}))$ with $\im \lambda >0$, while \eqref{lol9 bis bis} is valid for $\lambda \in \lambda(\partial D_{\epsilon}(k_{0}))$ with $\im \lambda <0$. One can check that the right-hand side of \eqref{lol9 bis bis} is the analytic continuation of the right-hand side of \eqref{lol8 bis bis}. Indeed, suppose that $k \in (-1,0)$ corresponds to $\lambda \in (-\infty, -1)$. Then $k_{2}^{h}(\lambda)=\frac{1}{\omega k}$, so (\ref{d3def}) implies that
$$\frac{\delta_{3,+}(\lambda)}{\delta_{3,-}(\lambda)} = \omega^2 \tilde{r}(\omega k_2^h(\lambda))
= \omega^2 \tilde{r}(k^{-1}).$$
Consequently, using (\ref{d2def}), \eqref{lol8 bis bis}, and the fact that $\tilde{r}(\frac{1}{\omega^{2} k}) = \frac{1-k^2}{\omega(\omega - k^2)}$, we obtain 
\begin{align*}
\bigg(\frac{\delta_{2}(\frac{1}{\omega k})(\omega - k^{2})}{(\lambda+1)^{2/3}\omega^{2}(1-k^{2})}\bigg)_+
= \frac{\delta_{3,+}(\lambda)\tilde{r}(\frac{1}{\omega^{2} k})(\omega - k^{2})}{\omega |\lambda+1|^{2/3}\omega^{2}(1-k^{2})} = \frac{\delta_{3,-}(\lambda) \tilde{r}(\tfrac{1}{k})}{\omega^2 |\lambda+1|^{2/3}}
= \bigg(\frac{\delta_{2}(\tfrac{1}{\omega^2 k})}{(\lambda+1)^{2/3}}\bigg)_-.
\end{align*}

\subsubsection*{Final steps}
The function $G(\lambda)$ is analytic for $\lambda \in \C \setminus (\R \cup \lambda(\partial \mathcal{D}))$ and has continuous boundary values on $\lambda(\partial \mathcal{D}) \cup \R \setminus \{-1,0,1\}$. 
Moreover, we saw above that $G$ has no jump across $\lambda[\partial D_{\epsilon}(k_{0})] \cup \lambda[\partial D_{\epsilon}(k_{0})^*]$ for $k_{0}\in \mathsf{Z}\setminus \R$. Consequently, since $G(\lambda) = O(\lambda^{-4/3})$ as $\lambda \to \infty$, Cauchy's theorem yields
\begin{align}
& \int_{\mathbb{R}\setminus \cup_{k_{0}\in \mathsf{Z}\cap \mathbb{R}}\lambda[(k_{0}-\epsilon,k_{0}+\epsilon)]} G_{+}(\lambda)d\lambda - \sum_{k_{0}\in \mathsf{Z}\cap \mathbb{R}}\int_{\lambda[\partial D_{\epsilon}(k_{0})] \cap \{\im \lambda > 0\}} G_{-}(\lambda)d\lambda = 0, \label{lol1} \\
&- \int_{\mathbb{R}\setminus \cup_{k_{0}\in \mathsf{Z}\cap \mathbb{R}}\lambda[(k_{0}-\epsilon,k_{0}+\epsilon)]} G_{-}(\lambda)d\lambda 
- \sum_{k_{0}\in \mathsf{Z}\cap \mathbb{R}}\int_{\lambda[\partial D_{\epsilon}(k_{0})] \cap \{\im \lambda < 0\}} G_{-}(\lambda)d\lambda = 0, \label{lol2}\\
& \int_{\cup_{k_{0}\in \mathsf{Z}\cap \mathbb{R}}\lambda[(k_{0}-\epsilon,k_{0}+\epsilon)]} G_{+}(\lambda)d\lambda + \sum_{k_{0}\in \mathsf{Z}\cap \mathbb{R}}\int_{\lambda[\partial D_{\epsilon}(k_{0})] \cap \{\im \lambda > 0\}} G_{+}(\lambda)d\lambda = 0, \label{lol3} \\
& -\int_{\cup_{k_{0}\in \mathsf{Z}\cap \mathbb{R}}\lambda[(k_{0}-\epsilon,k_{0}+\epsilon)]} G_{-}(\lambda)d\lambda + \sum_{k_{0}\in \mathsf{Z}\cap \mathbb{R}}\int_{\lambda[\partial D_{\epsilon}(k_{0})] \cap \{\im \lambda < 0\}} G_{+}(\lambda)d\lambda = 0. \label{lol4}
\end{align}
Adding the four equations \eqref{lol1}--\eqref{lol4}, we get
\begin{align*}
\int_{\mathbb{R}} (G_{+}(\lambda)-G_{-}(\lambda))d\lambda + \sum_{k_{0}\in \mathsf{Z}\cap \mathbb{R}}\int_{\lambda[\partial D_{\epsilon}(k_{0})] } (G_{+}(\lambda)-G_{-}(\lambda)) d\lambda  = 0,
\end{align*}
or, changing variables in the integral over $\lambda[\partial D_{\epsilon}(k_{0})]$, 
\begin{align}\label{intGplusGminus}
\int_{\mathbb{R}} (G_{+}(\lambda)-G_{-}(\lambda))d\lambda + \sum_{k_{0}\in \mathsf{Z}\cap \mathbb{R}}\int_{\partial D_{\epsilon}(k_{0}) } (G_{+}(\lambda)-G_{-}(\lambda)) \frac{d\lambda}{dk} dk  = 0.
\end{align}
Using \eqref{lol8}, \eqref{lol9}, \eqref{lol8 bis bis}, \eqref{lol9 bis bis}, and the residue theorem to compute the integrals over $\partial D_{\epsilon}(k_{0})$, we can write (\ref{intGplusGminus}) as
\begin{align}
& \int_{\mathbb{R}} (G_{+}(\lambda)-G_{-}(\lambda))d\lambda + \sum_{k_{0}\in \mathsf{Z}\cap (1,\infty)}\frac{-2\pi i E_{k_{0}}(x,t) \delta_{2}(k_{0})(k_{0}^{2}-\omega^{2})}{(\lambda(k_{0})+1)^{2/3}(1-\omega^{2}k_{0}^{2})}|u_{1}(k_{0})|^{2} \lambda'(k_0) \nonumber 
	\\
& + \sum_{k_{0}\in \mathsf{Z}\cap (-1,0)} \frac{-2\pi iE_{k_{0}}(x,t) \delta_{2,+}(\frac{1}{\omega k_{0}}) \omega (k_{0}^{2}-\omega^{2})}{|\lambda(k_{0})+1|^{2/3}(1-k_{0}^{2})}|u_{1}(\tfrac{1}{\omega^{2} k_{0}})|^{2} \lambda'(k_0)= 0. \label{lol10}
\end{align}
For $k_{0} \in \mathsf{Z}\cap (1,\infty)$, we have $\lambda'(k_0)>0$, and thus
\begin{align*}
& \frac{-2\pi i E_{k_{0}}(x,t) \delta_{2}(k_{0})(k_{0}^{2}-\omega^{2})}{(\lambda(k_{0})+1)^{2/3}(1-\omega^{2}k_{0}^{2})}|u_{1}(k_{0})|^{2}\lambda'(k_0) 
	\\
& \qquad = i(\omega^{2}k_{0}^{2} - \omega) E_{k_{0}}(x,t) \frac{2\pi|\delta_{3}(\lambda(k_{0}))||u_{1}(k_{0})|^{2}}{|\lambda(k_{0})+1|^{2/3}|k_{0}^{2}-1|} |\lambda'(k_0)| \geq 0,
\end{align*}
where we have used \eqref{ck0positivitycondition} for the last inequality. 
Recall from \cite[(4.8)]{CLmain} that $\omega^{2}\delta_{2+}(\frac{1}{\omega k_{0}})>0$ for $k_0 \in (-1,0)$. For $k_0 \in (-1,0)$, it also holds that $\lambda'(k_0) < 0$.
Thus, for $k_{0} \in \mathsf{Z}\cap (-1,0)$, we have
\begin{align*}
& \frac{-2\pi iE_{k_{0}}(x,t) \delta_{2,+}(\frac{1}{\omega k_{0}}) \omega (k_{0}^{2}-\omega^{2})}{|\lambda(k_{0})+1|^{2/3}(1-k_{0}^{2})}|u_{1}(\tfrac{1}{\omega^{2} k_{0}})|^{2}
\lambda'(k_0)
	 \\
& \qquad = i (\omega^{2}k_{0}^{2} - \omega) E_{k_{0}}(x,t) \frac{2\pi \omega^{2}\delta_{2,+}(\frac{1}{\omega k_{0}}) |u_{1}(\tfrac{1}{\omega^{2} k_{0}})|^{2}}{|\lambda(k_{0})+1|^{2/3}(1-k_{0}^{2})} |\lambda'(k_0)| \geq 0,
\end{align*}
where we have used \eqref{ck0positivitycondition} for the last inequality. 
It follows that \eqref{lol10} can be rewritten as
\begin{align}
& \int_{\mathbb{R}} (G_{+}(\lambda)-G_{-}(\lambda))d\lambda 
+ \sum_{k_{0}\in \mathsf{Z}\cap (1,\infty)}i(\omega^{2}k_{0}^{2} - \omega) E_{k_{0}}(x,t)\frac{2\pi|\delta_{3}(\lambda(k_{0}))||u_{1}(k_{0})|^{2}}{|\lambda(k_{0})+1|^{2/3}|k_{0}^{2}-1|} |\lambda'(k_0)|
	\nonumber \\
& + \sum_{k_{0}\in \mathsf{Z}\cap (-1,0)} i (\omega^{2}k_{0}^{2} - \omega) E_{k_{0}}(x,t) \frac{2\pi \omega^{2}\delta_{2,+}(\frac{1}{\omega k_{0}}) |u_{1}(\tfrac{1}{\omega^{2} k_{0}})|^{2}}{|\lambda(k_{0})+1|^{2/3}(1-k_{0}^{2})}|\lambda'(k_0)| = 0 \label{lol10 bis}
\end{align}
But we have shown that $G_+ - G_- \geq 0$ on $\R \setminus \{-1,0,1\}$ and that all elements in the above two sums are $\geq 0$, so recalling the formulas (\ref{Gjump1})--(\ref{Gjump4}) for $G_+ - G_-$, we find that 
\begin{align*}
& u_3(k) = 0 \; \text{for $k \in \Gamma_1^h$};
\;\;\, u_1(k) = 0 \; \text{for $k \in \Gamma_2^h$};
\;\;\,  \mathbf{u}^2(k) = 0 \; \text{for $k \in \Gamma_9^h$}; \;\;\, \mathbf{u}^6(k) = 0 \; \text{for $k \in \Gamma_{10}^h$}.
\end{align*}	
In the same way as in the final steps of \cite[Proof of Lemma 4.2]{CLmain}, we conclude from these relations that $\mathbf{u}$ is identically zero for all $k \in \C \setminus (\Gamma\cup \partial \mathcal{D})$. 
\end{proof}

\subsection{A remark about zeros in $D_{\mathrm{sing}}$}\label{Dsingsubsec}
As mentioned in the introduction, zeros of $s_{11}$ in $D_{\mathrm{sing}}$ correspond to singular breather solitons. Such breathers are not global solutions of \eqref{badboussinesq} (since they have singularities), and therefore one cannot guarantee existence of $\tilde{n}$ for all $x$ and $t$ if such zeros are allowed. In this subsection, we explain where the proof of Lemma \ref{vanishinglemma} breaks down if $s_{11}$ is allowed to have zeros in $D_{\mathrm{sing}}$. 

Suppose that $k_{0} \in D_{\mathrm{sing}}^{R}$ is a simple zero of $s_{11}$, and define $\tilde{n}$ as in \eqref{def of ntildeC}--\eqref{def of ntildeR}. Following the proof of Lemma \ref{vanishinglemma}, we are led to compute the jumps of $G$ for $\lambda \in \lambda(\partial D_{\epsilon}(k_{0}))$, $k_{0} \in D_{\mathrm{sing}}^{R}$. Letting $k\in \partial D_{\epsilon}(k_{0})$ correspond to $\lambda \in \lambda(\partial D_{\epsilon}(k_{0}))$, we find
\begin{align}
G_+(\lambda) - G_-(\lambda) & = \frac{1}{(\lambda+1)^{2/3}}\Big(\mathbf{u}_+(k) \Delta_2(k) \overline{\mathbf{u}_+(\bar{k})}^T
- \mathbf{u}_-(k) \Delta_2(k) \overline{\mathbf{u}_-(\bar{k})}^T \Big) \nonumber \\
& = \frac{\mathbf{u}_-(k)}{(\lambda+1)^{2/3}} \Big(Q_{1}(k) \Delta_2(k) \overline{Q_{7}(\bar{k})}^T  
- \Delta_2(k) \Big)\overline{\mathbf{u}_-(\bar{k})}^T \nonumber \\
& = \frac{c_{k_0} e^{-\theta_{31}(x,t,k_0)} \delta_{3}(\lambda)(k^{2}-\omega)}{(\lambda+1)^{2/3}(k-k_{0})(k_{0}^{2}-\omega)}u_{1}(k) \overline{u_{3}(\overline{k})}, \label{lol6}
\end{align}
where for the last equality we have used \eqref{d2def} and the fact that $u_{1,+}=u_{1,-}$, $u_{3,+}=u_{3,-}$ on $\partial D_{\epsilon}(k_{0})$. Similarly, letting $k\in \partial D_{\epsilon}(k_{0})^{*}$ correspond to $\lambda \in \lambda(\partial D_{\epsilon}(k_{0})^{*})$, we find
\begin{align}
G_+(\lambda) - G_-(\lambda) & = \frac{1}{(\lambda+1)^{2/3}}\Big(\mathbf{u}_+(k) \Delta_3(k) \overline{\mathbf{u}_+(\bar{k})}^T - \mathbf{u}_-(k) \Delta_3(k) \overline{\mathbf{u}_-(\bar{k})}^T \Big) \nonumber \\
& = \frac{\mathbf{u}_-(k)}{(\lambda+1)^{2/3}} \Big(Q_{7}(k) \Delta_3(k) \overline{Q_{1}(\bar{k})}^T  
- \Delta_3(k) \Big)\overline{\mathbf{u}_-(\bar{k})}^T \nonumber  \\
& = \frac{-\bar{c}_{k_0} e^{-\overline{\theta_{31}(x,t,k_0)}} \delta_{3}(\lambda)(k^{2}-\omega^{2})}{(\lambda+1)^{2/3}(k-\bar{k}_0)(\bar{k}_0^{2}-\omega^{2})}u_{3}(k) \overline{u_{1}(\overline{k})}. \label{lol7}
\end{align}
To proceed, we would therefore have to add
\begin{align*}
& - \sum_{\substack{k_{0}\in \mathsf{Z}\setminus \mathbb{R} \\ \im k_{0}<0}} 4\pi \, \im \bigg( \frac{c_{k_0} e^{-\theta_{31}(x,t,k_0)} \delta_{3}(\lambda(k_{0}))(k_{0}^{2}-\omega)}{(\lambda(k_{0})+1)^{2/3}(k_{0}^{2}-\omega)} \lambda'(k_0) u_{1}(k_{0}) \overline{u_{3}(\bar{k}_0)}     \bigg)  \nonumber
\end{align*}
to the left-hand side of \eqref{lol10 bis}. The issue is that one cannot guarantee in general that
\begin{align*}
- 4\pi \, \im \bigg( \frac{c_{k_0} e^{-\theta_{31}(x,t,k_0)} \delta_{3}(\lambda(k_{0}))(k_{0}^{2}-\omega)}{(\lambda(k_{0})+1)^{2/3}(k_{0}^{2}-\omega)} \lambda'(k_0) u_{1}(k_{0}) \overline{u_{3}(\bar{k}_0)}     \bigg) \geq 0.
\end{align*}
There is a similar issue if $k_{0}\in D_{\mathrm{sing}}^{L}$ is a simple zero of $s_{11}$.

\subsection{Proof of Theorem \ref{inverseth}}\label{inversethsubsec}
The open set $\C \setminus \tilde{\Gamma} = \Omega_+ \cup \Omega_-$ is the disjoint union of $\Omega_+$ and $\Omega_-$, where
\begin{align*}
& \Omega_+ \hspace{-0.05cm} = \hspace{-0.05cm} \bigg( \hspace{-0.05cm} D_1 \hspace{-0.05cm} \cup \hspace{-0.05cm} D_3 \hspace{-0.05cm} \cup \hspace{-0.05cm} D_5 \hspace{-0.05cm} \cup \bigcup_{\substack{j=0,1,2 \\ k_{0}\in \mathsf{Z}}} \hspace{-0.05cm} \big(\omega^{j} D_{\epsilon}(k_{0}) \cup \omega^{j} D_{\epsilon}(k_{0})^{*}\big) \bigg) \setminus  \hspace{-0.05cm} \bigcup_{\substack{j=0,1,2 \\ k_{0}\in \mathsf{Z}}} \big( \omega^{j} D_{\epsilon}(k_{0})^{-1} \cup \omega^{j} D_{\epsilon}(k_{0})^{-*} \big),  \\
& \Omega_- \hspace{-0.05cm} = \hspace{-0.05cm} \bigg( \hspace{-0.05cm} D_2 \hspace{-0.05cm} \cup \hspace{-0.05cm} D_4 \hspace{-0.05cm} \cup \hspace{-0.05cm} D_6 \hspace{-0.05cm} \cup \hspace{-0.05cm} \bigcup_{\substack{j=0,1,2 \\ k_{0}\in \mathsf{Z}}} \big( \omega^{j} D_{\epsilon}(k_{0})^{-1} \cup \omega^{j} D_{\epsilon}(k_{0})^{-*} \big) \bigg) \setminus \hspace{-0.05cm} \bigcup_{\substack{j=0,1,2 \\ k_{0}\in \mathsf{Z}}} \big(\omega^{j} D_{\epsilon}(k_{0}) \cup \omega^{j} D_{\epsilon}(k_{0})^{*}\big).
\end{align*}
In this proof, we orient $\Gamma$ so that $\Omega_{+}$ lies on
the left and $\Omega_{-}$ lies on the right of $\Gamma$; this orientation differs from Figure \ref{fig: Dn} only in that $\Gamma_{2},\Gamma_{4},\Gamma_{6}$ are now oriented towards $0$. 
Let the Sobolev space $H^N(\Gamma)$ consist of all $f \in L^2(\Gamma)$ with $N$ weak derivatives in $L^2$. Let $H_z^N(\Gamma)$ be the space of all $f \in H^N(\Gamma)$ such that $f$ satisfies the $(N-1)$th-order zero-sum condition; see e.g. \cite[Definition 2.47]{TO2016} for the formulation of this condition. Let 
$$H_\pm^N(\Gamma) := \{f \in L^2(\Gamma) \, | \, \text{$f \in H_z^N(\partial D)$ for every connected component $D$ of $\Omega_\pm$}\}.$$

\begin{lemma}\label{vpmlemma}
For any integer $N \geq 1$, there exist $3 \times 3$-matrix valued functions $v^\pm$ such that
\begin{enumerate}[$(a)$]
\item $v = (v^-)^{-1}v^+$ on $\Gamma  \setminus \Gamma_\star$, 
\item $v^\pm, (v^\pm)^{-1} \in I + H_{\pm}^N(\Gamma)$,
\item $v^\pm(k) = \mathcal{A} v^\pm(\omega k) \mathcal{A}^{-1}
= \mathcal{B}v^\mp(k^{-1})\mathcal{B}$ for $k \in \Gamma$, and 
\item $w^+ := v^+ - I$ and $w^{-} := I - v^{-}$ are nilpotent.
\end{enumerate}
\end{lemma}
\begin{proof}
Let $v_{j}^{\pm}$ be the restriction of $v^{\pm}$ to $\Gamma_{j}$. The matrices $\{v_{j}^{\pm}\}_{j=1}^{9}$ can be chosen as in \cite[Proof of Lemma 4.3]{CLmain}. For $k \in \partial \mathcal{D}$, we define
\begin{align*}
& v^{+}(k)=v(k), \quad v^{-}(k) = I, & & \mbox{for } k \in \bigcup_{\substack{j=0,1,2 \\ k_{0}\in \mathsf{Z}}} \big(\omega^{j} \partial D_{\epsilon}(k_{0}) \cup \omega^{j} \partial D_{\epsilon}(k_{0})^{*}\big), \\
& v^{-}(k)=v(k)^{-1}, \quad v^{+}(k) = I, & & \mbox{for } k \in \bigcup_{\substack{j=0,1,2 \\ k_{0}\in \mathsf{Z}}} \big( \omega^{j} \partial D_{\epsilon}(k_{0})^{-1} \cup \omega^{j} \partial D_{\epsilon}(k_{0})^{-*} \big).
\end{align*}
The rest of the proof proceeds as in \cite[Proof of Lemma 4.3]{CLmain}.
\end{proof}

Since $\Gamma_{2},\Gamma_{4},\Gamma_{6}$ are now oriented towards $0$, Lemma \ref{vsymmlemma} implies that
\begin{align*}
(\overline{v(\bar{k})^{-1}})^T = R(k)^{-1}  v(k) R(k)
\end{align*}
on all of $\tilde{\Gamma} = \Gamma \cup \partial \mathcal{D}$.
Hence the rest of the proof of Theorem \ref{inverseth} follows in the same way as in \cite[Section 4.2]{CLmain}.

\appendix

\section{Pure soliton solutions}
In this appendix, we derive exact formulas for the one-solitons and the single breather solitons for (\ref{badboussinesq}) and study their regularity properties. Pure soliton solutions are constructed by solving RH problem \ref{RH problem for M} in the case when $r_{1}=0$, $r_{2}=0$, and $\alpha = \beta = \tilde{\alpha} = \tilde{\beta}=0$. In this case $v=I$ on $\Gamma$ and the poles of $M$ at $\{\kappa_{j}\}_{j=1}^{6}$ are removable. 

Let $\{\mathsf{Z}, \{c_{k_0}\}_{k_0 \in \mathsf{Z}}\}$ be scattering data satisfying properties $(\ref{Theorem2.3itemvi})$--$(\ref{Theorem2.3itemvii})$ of Theorem \ref{directth}, and define $\{d_{k_0}\}_{k_0 \in \mathsf{Z}}$ as in (\ref{dk0def}). We are led to consider the following RH problem. 

\begin{RHproblem}[Pure soliton RH problem]\label{RH problem for M sol}
Find $M(x,t,k)$ with the following properties:
\begin{enumerate}[(a)]
\item $M(x,t,\cdot) : \mathbb{C}\setminus \hat{\mathsf{Z}} \to \mathbb{C}^{3 \times 3}$ is analytic.

\item As $k \to \infty$, $M(x,t,k) = I + O(k^{-1})$. 

\item $M$ satisfies the symmetries $M(x,t, k) = \mathcal{A} M(x,t,\omega k)\mathcal{A}^{-1} = \mathcal{B} M(x,t,\tfrac{1}{k})\mathcal{B}$.

\item 
At each point of $\hat{\mathsf{Z}}$, two columns of $M$ are analytic while one column has (at most) a simple pole. Moreover, for each $k_{0}\in \mathsf{Z}\setminus \mathbb{R}$, 
\begin{align}
& \underset{k = k_0}{\res} [M(x,t,k)]_3 = c_{k_{0}} e^{-\theta_{31}(x,t,k_0)}[M(x,t,k_0)]_1, \nonumber \\
& \underset{k = \bar{k}_0}{\res} [M(x,t,k)]_2 = d_{k_0} e^{\theta_{32}(x,t,\bar{k}_0)} [M(x,t,\bar{k}_0)]_3, \label{Mk0notinR}
\end{align}
and, for each $k_{0} \in \mathsf{Z}\cap \mathbb{R}$,
\begin{align}\label{Mk0inR}
\underset{k = k_0}{\res} [M(x,t,k)]_2 = c_{k_{0}} e^{-\theta_{21}(x,t,k_0)}[M(x,t,k_0)]_1.
\end{align}

\end{enumerate}
\end{RHproblem}

Define
$$C_{k_{0}}(x,t) := c_{k_{0}} e^{-\theta_{31}(x,t,k_0)}, \qquad
 D_{k_{0}}(x,t) := d_{k_0} e^{\theta_{32}(x,t,\bar{k}_0)}, \qquad
 E_{k_{0}}(x,t) := c_{k_0} e^{-\theta_{21}(x,t,k_0)}.$$
Using the symmetries $\mathcal{A}$- and $\mathcal{B}$-symmetries together with (\ref{Mk0notinR}) and (\ref{Mk0inR}), we can write the residue conditions for $M$ at all of its poles as follows: For each $k_{0}\in \mathsf{Z}\setminus \mathbb{R}$, 
\begin{subequations}\label{Mresiduesk0 sol}
\begin{align}
& \underset{k = k_0}{\res} [M(k)]_{3} = C_{k_{0}} [M(k_0)]_{1}, & & \underset{k = \omega k_0}{\res} [M(k)]_{2} = \omega C_{k_{0}} [M(\omega k_0)]_{3}, \label{res M 1 ab sol}
	\\ 
& \underset{k = \omega^2 k_0}{\res} [M(k)]_{1} = \omega^2 C_{k_{0}} [M(\omega^2 k_0)]_{2}, & & \underset{k = k_0^{-1}}{\res} [M(k)]_{3} = - k_0^{-2} C_{k_{0}} [M(k_0^{-1})]_{2},	\label{res M 1 cd sol}
	\\
&  \underset{k = \omega^2k_0^{-1}}{\res}  [M(k)]_{1} = -\tfrac{\omega^2}{k_0^{2}} C_{k_{0}} [M(\omega^2 k_0^{-1})]_{3}, & & \underset{k = \omega k_0^{-1}}{\res}  [M(k)]_{2} = -\tfrac{\omega}{k_0^{2}} C_{k_{0}} [M(\omega k_0^{-1})]_{1}, \label{res M 1 ef sol} \\
& \underset{k = \bar{k}_0}{\res} [M(k)]_{2} = D_{k_{0}} [M(\bar{k}_0)]_{3}, & & \underset{k = \omega \bar{k}_0}{\res} [M(k)]_{1} = \omega D_{k_{0}} [M( \omega \bar{k}_0)]_{2}, \label{res M 2 ab sol}
	\\ 
& \underset{k = \omega^2 \bar{k}_0}{\res} [M(k)]_{3} = \omega^2 D_{k_{0}} [M(\omega^2 \bar{k}_0)]_{1}, & & \underset{k = \bar{k}_0^{-1}}{\res} [M(k)]_{1} = - \bar{k}_0^{-2} D_{k_{0}} [M(\bar{k}_0^{-1})]_{3},	 \label{res M 2 cd sol}
	\\
&  \underset{k = \omega^2 \bar{k}_0^{-1}}{\res}  [M(k)]_{2} = -\tfrac{\omega^2}{\bar{k}_0^{2}} D_{k_{0}} [M( \omega^2 \bar{k}_0^{-1})]_{1}, & & \underset{k = \omega \bar{k}_0^{-1}}{\res}  [M(k)]_{3} = -\tfrac{\omega}{\bar{k}_0^{2}} D_{k_{0}} [M(\omega \bar{k}_0^{-1})]_{2}, \label{res M 2 ef sol}
\end{align}
\end{subequations}
and, for each $k_{0}\in \mathsf{Z}\cap \mathbb{R}$, 
\begin{subequations}\label{Mresiduesk0real sol}
\begin{align}
& \underset{k = k_0}{\res} [M(k)]_{2} = E_{k_{0}} [M(k_0)]_{1}, & & \underset{k = \omega k_0}{\res} [M(k)]_{1} = \omega E_{k_{0}} [M(\omega k_0)]_{3}, \label{res M 3 ab sol}
	\\ 
& \underset{k = \omega^2 k_0}{\res} [M(k)]_{3} = \omega^2 E_{k_{0}} [M(\omega^2 k_0)]_{2}, & & \underset{k = k_0^{-1}}{\res} [M(k)]_{1} = - k_0^{-2} E_{k_{0}} [M(k_0^{-1})]_{2},	\label{res M 3 cd sol}
	\\
&  \underset{k = \omega^2k_0^{-1}}{\res} [M(k)]_{2} = -\tfrac{\omega^2}{k_0^{2}} E_{k_{0}} [M(\omega^2 k_0^{-1})]_{3}, & & \underset{k = \omega k_0^{-1}}{\res} [M(k)]_{3} = -\tfrac{\omega}{k_0^{2}} E_{k_{0}} [M(\omega k_0^{-1})]_{1}, \label{res M 3 ef sol}
\end{align}
\end{subequations}
where the $(x,t)$-dependence has been omitted for conciseness.

The following lemma establishes existence and uniqueness of a solution of RH problem \ref{RH problem for M sol}. This result is not only relevant for the construction of pure soliton solutions, but it also plays a role in the calculation of the long-time asymptotics for (\ref{badboussinesq}) in the presence of solitons \cite{CLmainSectorII}.

\begin{lemma}
For each $(x,t) \in \R \times [0,\infty)$, the solution $M$ of RH problem \ref{RH problem for M sol} exists and is unique.
\end{lemma}
\begin{proof}
Let $n$ be the solution of RH problem \ref{RHn} with $r_{1}=0$ and $r_{2}=0$. It follows from Theorem \ref{inverseth} that $n$ exists. Motivated by \cite[Lemma 4.15]{CLmain}, we define $M$ by
\begin{align}\label{Mdef}
M(x,t,k) := P(k)^{-1} \begin{pmatrix}
ne^{\mathcal{L}x + \mathcal{Z}t}  \\
(ne^{\mathcal{L}x + \mathcal{Z}t})_x \\
(ne^{\mathcal{L}x + \mathcal{Z}t})_{xx}
\end{pmatrix} e^{-(\mathcal{L}x+\mathcal{Z}t)}
= P(k)^{-1} \begin{pmatrix}
n  \\
n_x + n\mathcal{L} \\
n_{xx} + 2n_x\mathcal{L} + n\mathcal{L}^2
\end{pmatrix}.
\end{align}
Clearly, $M(x,t,\cdot) : \mathbb{C}\setminus \hat{\mathsf{Z}} \to \mathbb{C}^{3 \times 3}$ is analytic. The fact that $M(x,t,k)=I+O(k^{-1})$ as $k\to \infty$ can be proved as in \cite[Lemma 4.15]{CLmain}. Using the symmetries $n(k)=n(\omega k)\mathcal{A}^{-1}=n(k^{-1})\mathcal{B}$, $\mathcal{L}(k)=\mathcal{A}\mathcal{L}(\omega k)\mathcal{A}^{-1}=\mathcal{B}\mathcal{L}(k^{-1})\mathcal{B}$, $\mathcal{Z}(k)=\mathcal{A}\mathcal{Z}(\omega k)\mathcal{A}^{-1}=\mathcal{B}\mathcal{Z}(k^{-1})\mathcal{B}$, and $P(k)=P(\omega k)\mathcal{A}^{-1}=P(k^{-1})\mathcal{B}$, it is easy to check that $M$ verifies condition $(c)$ of RH problem \ref{RH problem for M sol}. Let $k_{0}\in \mathsf{Z}\setminus \R$. Since $n_{1}$ and $n_{2}$ are analytic at $k_{0}$ and $n_{3}$ has at most a simple pole at $k_{0}$, it follows that the first two columns of $M$ are analytic at $k_{0}$ and that the third column of $M$ has at most a simple pole. Moreover,
\begin{align*}
\underset{k = k_0}{\res} [M(k)]_{3} = P(k_{0})^{-1} \begin{pmatrix}
\underset{k = k_0}{\res}n_{3}(k)  \\
\partial_{x} \underset{k = k_0}{\res}n_{3}(k) + l_{3}(k_{0})\underset{k = k_0}{\res}n_{3}(k) \\
\partial_{xx}\underset{k = k_0}{\res}n_{3}(k) + 2 l_{3}(k_{0}) \partial_{x} \underset{k = k_0}{\res}n_{3}(k) + l_{3}(k_{0})^{2} \underset{k = k_0}{\res}n_{3}(k)
\end{pmatrix}.
\end{align*}
Since $\underset{k = k_0}{\res} n_3(k) = c_{k_{0}}e^{-\theta_{31}(x,t,k_0)} n_1(k_0)$, it follows that
\begin{align*}
\underset{k = k_0}{\res} [M(k)]_{3} & = c_{k_{0}}e^{-\theta_{31}(x,t,k_0)} P(k_{0})^{-1} \begin{pmatrix}
n_{1}(k_{0})  \\
\partial_{x} n_{1}(k_{0}) + l_{1}(k_{0}) n_{1}(k_{0}) \\
\partial_{xx} n_{1}(k_{0}) + 2 l_{1}(k_{0}) \partial_{x} n_{1}(k_{0}) + l_{1}(k_{0})^{2} n_{1}(k_{0})
\end{pmatrix} \\
& = c_{k_{0}}e^{-\theta_{31}(x,t,k_0)}[M(k_{0})]_{1}.
\end{align*}
The residues of $M$ at $k=\bar{k}_{0}$, $k_{0}\in \mathsf{Z}\setminus \R$, and at $k=k_{0}$, $k_{0}\in \mathsf{Z}\cap\R$, can be computed similarly, and we find that $M$ satisfies condition $(d)$ of RH problem \ref{RH problem for M sol}. Thus we have proved that $M$ exists.

The uniqueness of $M$ can be proved by noting that if $M$ and $\tilde{M}$ are two different solutions of RH problem \ref{RH problem for M sol}, then all the poles of $M \tilde{M}^{-1}$ at points in $\hat{\mathsf{Z}}$ are removable by long but straightforward calculations which use the fact that $\det \tilde{M} = 1$. Since $M \tilde{M}^{-1} \to 1$ as $k \to \infty$, it follows that $M \tilde{M}^{-1} = I$ for all $k$ by Liouville's theorem. An alternative proof transforms RH problem \ref{RH problem for M sol} into an equivalent RH problem in which the poles are replaced by jumps on small circles (this transformation is similar to the transformation described in Section \ref{subsec:RHP ntilde}).
The solution of the transformed RH problem (and hence also of RH problem \ref{RH problem for M sol}) is unique by standard arguments because the jump matrix has unit determinant.
\end{proof}

\subsection{One-soliton}\label{onesolitonsubsec}
Let us consider the case when $s_{11}$ has a single simple zero $k_0$ in $(-1,0) \cup (1, \infty)$. In this case, $\mathsf{Z} = \{k_0\}$ where $k_0 \in (-1,0)\cup(1,\infty)$, and the residue conditions \eqref{Mresiduesk0real sol} together with the normalization condition $M(x,t,k) = I + O(k^{-1})$ as $k \to \infty$ imply that
\begin{align*}
& [M(k)]_{1} = [I]_{1} + \frac{ \omega E_{k_{0}} [M(\omega k_0)]_{3}}{k - \omega k_0} + \frac{- k_0^{-2} E_{k_{0}} [M(k_0^{-1})]_{2}}{k - k_0^{-1}},
	\\
& [M(k)]_{2} = [I]_{2} + \frac{E_{k_{0}} [M(k_0)]_{1}}{k - k_0} + \frac{-\omega^2 k_0^{-2} E_{k_{0}} [M(\omega^2 k_0^{-1})]_{3}}{k - \omega^2k_0^{-1}},
	\\
& [M(k)]_{3} = [I]_{3} + \frac{\omega^2 E_{k_{0}} [M(\omega^2 k_0)]_{2}}{k -  \omega^2 k_0} + \frac{-\omega k_0^{-2} E_{k_{0}} [M(\omega k_0^{-1})]_{1}}{k - \omega k_0^{-1}}.
\end{align*}
We evaluate the first equation at $k = k_0$ and $k = \omega k_0^{-1}$, the second equation at $k = \omega^2 k_0$ and $k = k_0^{-1}$, and the third equation at $k = \omega k_0$ and $k = \omega^2 k_0^{-1}$. For each $j \in \{1,2,3\}$, this gives six algebraic equations for the six unknowns $M_{j1}(k_0)$, $M_{j1}(\omega k_0^{-1})$, $M_{j2}(\omega^2 k_0)$, $M_{j2}(k_0^{-1})$, $M_{j3}(\omega k_0)$, $M_{j3}(\omega^2 k_0^{-1})$. More precisely, with
\begin{align*}
A_{k_{0}}:= E_{k_{0}} \begin{pmatrix}
0 & 0 & 0 & \frac{-k_{0}^{-2}}{k_{0}-k_{0}^{-1}} & \frac{\omega}{k_{0}-\omega k_{0}} & 0 \\[0.2cm]
0 & 0 & 0 & \frac{-k_{0}^{-2}}{\omega k_{0}^{-1}-k_{0}^{-1}} & \frac{\omega}{\omega k_{0}^{-1}-\omega k_{0}} & 0 \\
\frac{1}{\omega^{2}k_{0}-k_{0}} & 0 & 0 & 0 & 0 & \frac{-\omega^{2}k_{0}^{-2}}{\omega^{2}k_{0}-\omega^{2}k_{0}^{-1}} \\
\frac{1}{k_{0}^{-1}-k_{0}} & 0 & 0 & 0 & 0 & \frac{-\omega^{2}k_{0}^{-2}}{k_{0}^{-1}-\omega^{2}k_{0}^{-1}} \\
0 & \frac{-\omega k_{0}^{-2}}{\omega k_{0}-\omega k_{0}^{-1}} & \frac{\omega^{2}}{\omega k_{0}-\omega^{2}k_{0}} & 0 & 0 & 0 \\[0.2cm]
0 & \frac{-\omega k_{0}^{-2}}{\omega^{2} k_{0}^{-1}-\omega k_{0}^{-1}} & \frac{\omega^{2}}{\omega^{2} k_{0}^{-1}-\omega^{2}k_{0}} & 0 & 0 & 0
\end{pmatrix},
\end{align*}
we have
\begin{multline*}
\begin{pmatrix}
[M(k_0)]_{1} & [M(\omega k_0^{-1})]_{1} & [M(\omega^2 k_0)]_{2} & [M(k_0^{-1})]_{2} & [M(\omega k_0)]_{3} & [M(\omega^2 k_0^{-1})]_{3} 
\end{pmatrix} \\
= \begin{pmatrix}
1 & 1 & 0 & 0 & 0 & 0 \\
0 & 0 & 1 & 1 & 0 & 0 \\
0 & 0 & 0 & 0 & 1 & 1
\end{pmatrix}(I-A_{k_{0}}^{T})^{-1}=: B_{k_{0}}.
\end{multline*}
Hence
\begin{align*}
& [M(k)]_{1} = [I]_{1} + \frac{ \omega E_{k_{0}} }{k - \omega k_0} [B_{k_{0}}]_{5} + \frac{- k_0^{-2} E_{k_{0}}}{k - k_0^{-1}}[B_{k_0}]_{4},
	\\
& [M(k)]_{2} = [I]_{2} + \frac{E_{k_{0}} }{k - k_0}[B_{k_0}]_{1} + \frac{-\omega^2 k_0^{-2} E_{k_{0}}}{k - \omega^2k_0^{-1}}[B_{k_0}]_{6},
	\\
& [M(k)]_{3} = [I]_{3} + \frac{\omega^2 E_{k_{0}} }{k -  \omega^2 k_0}[B_{k_0}]_{3} + \frac{-\omega k_0^{-2} E_{k_{0}} }{k - \omega k_0^{-1}}[B_{k_0}]_{2}.
\end{align*}
Since $n = (1,1,1)M$, we get
\begin{align*}
& n_1(k) = 1 + \frac{ \omega E_{k_{0}} }{k - \omega k_0}(1,1,1)[B_{k_0}]_{5} + \frac{- k_0^{-2} E_{k_{0}}}{k - k_0^{-1}}(1,1,1)[B_{k_0}]_{4},
	\\
& n_2(k) = 1 + \frac{E_{k_{0}} }{k - k_0}(1,1,1)[B_{k_0}]_{1} + \frac{-\omega^2 k_0^{-2} E_{k_{0}}}{k - \omega^2k_0^{-1}}(1,1,1)[B_{k_0}]_{6},
	\\
& n_3(k) = 1 + \frac{\omega^2 E_{k_{0}} }{k -  \omega^2 k_0}(1,1,1)[B_{k_0}]_{3} + \frac{-\omega k_0^{-2} E_{k_{0}} }{k - \omega k_0^{-1}}(1,1,1)[B_{k_0}]_{2},
\end{align*}
and the Boussinesq one-soliton is given by
$$u(x,t) = -i\sqrt{3}\frac{\partial}{\partial x}n_{3}^{(1)}(x,t),$$
where
$$n_{3}^{(1)}(x,t) := \lim_{k\to \infty} k (n_{3}(x,t,k) - 1) = \omega^2 E_{k_{0}} 
(1,1,1)[B_{k_0}]_{3} -\omega k_0^{-2} E_{k_{0}} (1,1,1)[B_{k_0}]_{2}.$$ 
A calculation gives
\begin{align}\label{uonezerodef}
u(x,t) = \frac{\frac{3}{2}(k_0 - k_0^{-1})^2}{\Big(f_{k_0} e^{-\frac{k_0^2 - 1}{4k_0}(x - ct)} + f_{k_0}^{-1} e^{\frac{k_0^2 - 1}{4k_0}(x - ct)}\Big)^2},
\end{align}
where
\begin{align}\label{cfk0def}
c := \frac{k_0 + k_0^{-1}}{2}, \qquad f_{k_0} := \sqrt{\frac{i \omega^2  (k_0^2 - \omega^2) c_{k_0}}{\sqrt{3}k_0(k_0^2 -1)}}.
\end{align}
It is not necessary to specify the branch of the square root in (\ref{cfk0def}) because the value of $u(x,t)$ does not depend on the choice of sign of $f_{k_0}$. 
If $f_{k_0} = 0$, then we interpret $u(x,t)$ in (\ref{uonezerodef}) as the zero solution $u(x,t) \equiv 0$.

\begin{lemma}\label{onesolitonsingularlemma}
Let $k_0\in (-1,0)\cup (1,\infty)$. The following are equivalent:
\begin{enumerate}[$(a)$]
\item The function $u(x,t)$ in (\ref{uonezerodef}) is smooth and real-valued for all $(x,t) \in \R \times [0,\infty)$.

\item $f_{k_0}^2 \geq 0$. 

\item $c_{k_0}$ obeys the positivity condition (\ref{ck0positivitycondition}), i.e., $i(\omega^{2}k_{0}^{2} - \omega)c_{k_{0}} \geq 0$.
\end{enumerate}
Moreover, $u(x,t)$ is singular if and only if $i(\omega^{2}k_{0}^{2} - \omega)c_{k_{0}} \in (-\infty,0)$.
\end{lemma}
\begin{proof}
Using the short-hand notation $z := f_{k_{0}} e^{-\frac{k_0^2 - 1}{4k_0}(x - ct)}$, we can write (\ref{uonezerodef}) as
 $$u(x,t) = \frac{\frac{3}{2}(k_0 - k_0^{-1})^2}{(z + \frac{1}{z})^2}.$$
Hence $u(x,t)$ is non-singular if and only if $z \neq \pm i$ for all $x,t$. Since the exponential $e^{-\frac{k_0^2 - 1}{4k_0}(x - ct)}$ takes on any value in $(0, \infty)$ as $x$ and $t$ vary, we see that $u(x,t)$ is singular if and only if $f_{k_{0}} \in i\R \setminus \{0\}$.

The function $u$ is real-valued if and only if $z + z^{-1} \in \R \cup i\R$ for all $x, t$. As $x \to +\infty$, we have 
$$z + z^{-1} \sim f_{k_{0}}^{-1} e^{\frac{k_0^2 - 1}{4k_0}(x - ct)},$$
so if $u$ is real-valued, then $f_{k_{0}} \in \R \cup i\R$. On the other hand, if $f_{k_{0}} \in \R \cup i\R$, then $z + z^{-1} \in \R \cup i\R$. So $u$ is real-valued if and only if $f_{k_{0}} \in \R \cup i\R$. 

We conclude that $u(x,t)$ is both real-valued and smooth for all $x,t$ if and only if $f_{k_{0}} \in \R$, or, in other words, if and only if $f_{k_{0}}^2 \geq 0$. 
This proves that $(a)$ and $(b)$ are equivalent. 

The equivalence of $(b)$ and $(c)$ as well as the last claim follow from the fact that
$$\frac{i \omega^2 (k_0^2 -\omega^2)}{ \sqrt{3} k_0 (k_0^2 -1)} \frac{1}{i(\omega^2 k_0^2 -\omega)} = \frac{1}{\sqrt{3}k_{0}(k_{0}^{2}-1)} > 0$$
for all $k_0 \in (-1,0) \cup (1, \infty)$.
\end{proof} 

Lemma \ref{onesolitonsingularlemma} shows that the general (non-singular) one-soliton is given by (\ref{uonezerodef}) for some $f_{k_0}$ with $f_{k_0}^2 \geq 0$. 
Using $x_0 \in \R$ to parametrize the allowed values of $f_{k_0}$ according to 
$$f_{k_0} = \pm e^{\frac{k_0^2 -1}{4k_0} x_0},$$ 
this leads to the following well-known expression for the one-soliton solution of (\ref{badboussinesq}) (see for example \cite[Eq. (1.5)]{DH1999}):
\begin{align}\label{onesoliton}
u(x,t) = A \sech^2(\sqrt{A/6}(x - x_0 -ct)),
\end{align}
where $A := \frac{3}{8}(k_0 - k_0^{-1})^2 \geq 0$ (here the branch of $\sqrt{A}$ is defined such that $\sqrt{A} \geq 0$ and we have used that $\sqrt{A} = \sqrt{\frac{3}{8}}(k_0 - k_0^{-1})$ because $k_0 \in (-1,0) \cup (1, +\infty)$). 

If $k_0 \in (1, \infty)$, then $c > 1$, while if $k_0 \in (-1, 0)$, then $c < -1$.
Thus, as claimed in the introduction, the zeros of $s_{11}$ in $(1, \infty)$ give rise to right-moving solitons, while the zeros of $s_{11}$ in $(-1,0)$ give rise to left-moving solitons. It is worth noting that both right- and left-moving solitons travel with speeds greater than $1$.

\subsection{Breather}
Let us now consider the case when $s_{11}$ has a single simple zero $k_0$ in $D_2 \setminus \R$, i.e., let $\mathsf{Z} = \{k_0\}$ with $k_0 \in D_2 \setminus \R$.
It is possible to proceed as we did above in the case of the one-soliton; however, the resulting calculations become cumbersome because one needs to solve an algebraic system with $12$ equations. The following approach makes better use of the $\mathcal{A}$- and $\mathcal{B}$-symmetries and is more efficient.

The residue conditions \eqref{Mresiduesk0 sol} together with the normalization condition $M(x,t,k) = I + O(k^{-1})$ as $k \to \infty$ imply that
\begin{align}\nonumber
[M(k)]_{3} = &\;  [I]_{3} + \frac{C_{k_0} [M(k_0)]_{1}}{k - k_0} + \frac{- k_0^{-2} C_{k_0} [M(k_0^{-1})]_{2}}{k - k_0^{-1}} 
	\\ \label{M3fromresidues}
& + \frac{\omega^2 D_{k_0} [M(\omega^2 \bar{k}_0)]_{1}}{k - \omega^2 \bar{k}_0} + \frac{-\omega \bar{k}_0^{-2} D_{k_0} [M(\omega \bar{k}_0^{-1})]_{2}}{k - \omega \bar{k}_0^{-1}}. 
\end{align}
The $\mathcal{A}$- and $\mathcal{B}$-symmeties (\ref{nsymm}) imply that
$$n_3(k) = n_2(\omega k) = n_1(\omega^2 k) = n_3(k^{-1}).$$
Hence (\ref{M3fromresidues}) yields
$$n_3(k) = 1 + \frac{C_{k_0} n_3(\omega k_0)}{k - k_0} 
+ \frac{- k_0^{-2} C_{k_0} n_3(\omega k_0)}{k - k_0^{-1}} 
+ \frac{\omega^2 D_{k_0} n_3(\bar{k}_0)}{k - \omega^2 \bar{k}_0} + \frac{-\omega \bar{k}_0^{-2} D_{k_0} n_3(\bar{k}_0)}{k - \omega \bar{k}_0^{-1}}$$
and $n_{3}^{(1)}$ can be written as
$$n_{3}^{(1)}(x,t)  = (1 - k_0^{-2}) C_{k_0} n_3(\omega k_0) + (1 - \omega^2 \bar{k}_0^{-2} ) \omega^2 D_{k_0} n_3(\bar{k}_0).$$ 
Moreover, since
$$\frac{\frac{i(k_0^2-1)}{2 \sqrt{3} k_0^2}}{l_3(k) - l_3(k_0)} = \frac{1}{k - k_0} 
+ \frac{- k_0^{-2} }{k - k_0^{-1}},$$
we obtain
\begin{align}\label{n3C1D1}
n_3(k) = 1 + \frac{\frac{i(k_0^2-1)}{2 \sqrt{3} k_0^2}}{l_3(k) - l_3(k_0)} C_{k_0} n_3(\omega k_0)
+ \frac{\frac{i(\bar{k}_0^2-\omega^2)}{2 \sqrt{3} \bar{k}_0^2}}{l_3(k) - l_3(\omega^2 \bar{k}_0)} \omega^2 D_{k_0} n_3(\bar{k}_0).
\end{align}
Let $\tilde{c}_{k_0} := \frac{i(k_0^2-1)}{2 \sqrt{3} k_0^2} c_{k_0}$ and $\tilde{d}_{k_0} := \frac{i(\bar{k}_0^2-\omega^2)}{2 \sqrt{3} \bar{k}_0^2} \omega^2 d_{k_0}$.
Evaluating (\ref{n3C1D1}) at $k = \omega k_0$ and $k = \bar{k}_0$, we find the system
\begin{align*}
n_3(\omega k_0) = &\; 1 + \frac{\tilde{c}_{k_0} e^{x(l_1(k_0) - l_3(k_0)) + t(z_1(k_0) - z_3(k_0))}}{l_1(k_0) - l_3(k_0)} n_3(\omega k_0)
	\\
& + \frac{\tilde{d}_{k_0} e^{x(l_3(\bar{k}_0) - l_2(\bar{k}_0))+ t(z_3(\bar{k}_0) - z_2(\bar{k}_0))}}{l_1(k_0) - l_2(\bar{k}_0)} n_3(\bar{k}_0),
	\\
n_3(\bar{k}_0) =&\; 1 + \frac{\tilde{c}_{k_0} e^{x(l_1(k_0) - l_3(k_0)) + t(z_1(k_0) - z_3(k_0))}}{l_3(\bar{k}_0) - l_3(k_0)} n_3(\omega k_0)
	\\
& + \frac{\tilde{d}_{k_0} e^{x(l_3(\bar{k}_0) - l_2(\bar{k}_0)) + t(z_3(\bar{k}_0) - z_2(\bar{k}_0))}}{l_3(\bar{k}_0) - l_2(\bar{k}_0)} n_3(\bar{k}_0).
\end{align*}
Multiplying the first equation by $e^{x l_1(k_0) + t z_1(k_0)}$ and the second by $e^{x l_3(\bar{k}_0) + t z_3(\bar{k}_0)}$, we obtain
\begin{align*}
\begin{pmatrix}
e^{x l_1(k_0) + t z_1(k_0)}n_3(\omega k_0) \\ e^{x l_3(\bar{k}_0) + t z_3(\bar{k}_0)} n_3(\bar{k}_0) \end{pmatrix}
= (I - B)^{-1}\begin{pmatrix} e^{x l_1(k_0) + t z_1(k_0)} \\ e^{x l_3(\bar{k}_0) + t z_3(\bar{k}_0)} \end{pmatrix}
\end{align*}
provided that the matrix $I - B$ is invertible, where 
$$B = B(x,t) := \begin{pmatrix} 
\frac{\tilde{c}_{k_0} e^{x(l_1(k_0) - l_3(k_0))+t(z_1(k_0) - z_3(k_0))}}{l_1(k_0) - l_3(k_0)}  &  \frac{\tilde{d}_{k_0} e^{x(l_1(k_0) - l_2(\bar{k}_0)) + t(z_1(k_0) - z_2(\bar{k}_0))}}{l_1(k_0) - l_2(\bar{k}_0)} 	\\
\frac{\tilde{c}_{k_0} e^{x (l_3(\bar{k}_0)- l_3(k_0)) + t (z_3(\bar{k}_0)- z_3(k_0))}}{l_3(\bar{k}_0) - l_3(k_0)} & \frac{\tilde{d}_{k_0} e^{x(l_3(\bar{k}_0) - l_2(\bar{k}_0)) + t(z_3(\bar{k}_0) - z_2(\bar{k}_0))}}{l_3(\bar{k}_0) - l_2(\bar{k}_0)} 
\end{pmatrix}.
$$
Since, by (\ref{n3C1D1}),
\begin{align}\nonumber
n_3^{(1)}(x,t) 
= & -2i\sqrt{3}\big(\tilde{c}_{k_0} e^{x(l_1(k_0) - l_3(k_0)) + t(z_1(k_0) - z_3(k_0))} n_3(\omega k_0) 
	\\ \label{n31breather}
& \qquad\qquad + \tilde{d}_{k_0} e^{x(l_3(\bar{k}_0) - l_2(\bar{k}_0)) + t(z_3(\bar{k}_0) - z_2(\bar{k}_0))} n_3(\bar{k}_0)\big),
\end{align}
we find
\begin{align}\nonumber
n_3^{(1)}(x,t) & =-2i\sqrt{3}\begin{pmatrix} \tilde{c}_{k_0} e^{-x l_3(k_0) - t z_3(k_0)} &  \tilde{d}_{k_0} e^{-x l_2(\bar{k}_0) - t z_2(\bar{k}_0)} \end{pmatrix} (I - B)^{-1}  \begin{pmatrix} e^{x l_1(k_0) + t z_1(k_0)} \\ e^{x l_3(\bar{k}_0) + t z_3(\bar{k}_0)} \end{pmatrix}
	\\\nonumber
& =-2i\sqrt{3}\tr\bigg\{ (I - B)^{-1}  \begin{pmatrix} e^{x l_1(k_0)+t z_1(k_0)} \\ e^{x l_3(\bar{k}_0)+t z_3(\bar{k}_0)} \end{pmatrix} \begin{pmatrix} \tilde{c}_{k_0} e^{-x l_3(k_0)-t z_3(k_0)} &  \tilde{d}_{k_0} e^{-x l_2(\bar{k}_0)-t z_2(\bar{k}_0)} \end{pmatrix}\bigg\}
	\\ \label{n31trB}
& =-2i\sqrt{3}\tr\big\{ (I - B)^{-1}  B_x\big\},
\end{align}
so the corresponding solution of (\ref{badboussinesq}) is
\begin{align}\label{ubreather}
u(x,t) = -i\sqrt{3}\frac{\partial}{\partial x}n_{3}^{(1)}(x,t)
= -6\frac{\partial}{\partial x} \tr\big\{ (I - B)^{-1}  B_x\big\}.
\end{align}
We refer to the solution (\ref{ubreather}) as a breather. Note that $u(x,t)$ is well-defined by (\ref{ubreather}) whenever $\det(I - B(x,t)) \neq 0$. The next lemma shows that the breather (\ref{ubreather}) is well-defined for all $x\in \R$ and $t\geq 0$ if $k_0 \in D_{\mathrm{reg}}$. It also shows that if $k_0 \in D_{\mathrm{sing}}$, then for each fixed $t\geq 0$, there is an $x \in \R$ such that $u(x,t)$ is not well-defined.

\begin{lemma}\label{breatherlemma}
Let $k_0\in D_2 \setminus \R = D_{\mathrm{reg}} \sqcup D_{\mathrm{sing}}$. 
\begin{enumerate}[$(a)$]
\item If $k_0 \in D_{\mathrm{reg}}$, then $\det(I-B) > 0$ for all $(x,t) \in \R \times [0, \infty)$. In particular, the function $u(x,t)$ in (\ref{ubreather}) is a smooth and real-valued solution of (\ref{badboussinesq}) for all $(x,t) \in \R \times [0, \infty)$ if $k_0 \in D_{\mathrm{reg}}$.

\item If $k_0 \in D_{\mathrm{sing}}$ and $c_{k_0} \in \C \setminus \{0\}$, then for each fixed $t\geq 0$, the function $x \mapsto \det(I-B)$ has at least one zero on $\R$. 

\end{enumerate}
\end{lemma}
\begin{proof}
Let
$$A :=  \begin{pmatrix} 
\frac{1}{l_1(k_0) - l_3(k_0)}  &  \frac{1}{l_1(k_0) - l_2(\bar{k}_0)} 	\\
\frac{1}{l_3(\bar{k}_0) - l_3(k_0)} & \frac{1}{l_3(\bar{k}_0) - l_2(\bar{k}_0)} 
  \end{pmatrix} 
 \begin{pmatrix} \tilde{c}_{k_0} e^{- \theta_{31}(x,t,k_0)} & 0 \\ 0 &  \tilde{d}_{k_0} e^{\theta_{32}(x,t,\bar{k}_0)} \end{pmatrix}.$$
Since
$$B = \begin{pmatrix} e^{xl_1(k_0) + t z_1(k_0)} & 0 \\ 0 & e^{xl_3(\bar{k}_0) + t z_3(\bar{k}_0)}  \end{pmatrix} A \begin{pmatrix} e^{xl_1(k_0) + t z_1(k_0)} & 0 \\ 0 & e^{xl_3(\bar{k}_0) + t z_3(\bar{k}_0)}  \end{pmatrix}^{-1},$$
we have $\det(I-B) = \det(I-A)$.
Using that $\bar{\tilde{c}}_{k_0} = \tilde{d}_{k_0}$ as a consequence of (\ref{dk0def}), we see that $A = A(x,t)$  has the form
$$A = \begin{pmatrix}
A_{11} & \frac{l_{3}(\bar{k}_{0})-l_{2}(\bar{k}_{0})}{l_{1}(k_{0})-l_{2}(\bar{k}_{0})} \bar{A}_{11} \\
\frac{l_{1}(k_{0})-l_{3}(k_{0})}{l_{3}(\bar{k}_{0})-l_{3}(k_{0})}A_{11} & \bar{A}_{11}
\end{pmatrix}.$$
Long but straightforward computations show that 
\begin{align}\label{detIminusA}
\det(I-A) = 1-2A_{11}^{R} + h(k_0) |A_{11}|^{2}, \quad h(k_0):=\frac{(k_{0}^{I}+\sqrt{3}k_{0}^{R})^{2}(1+|k_{0}|^{2}+|k_{0}|^{4})}{2k_{0}^{I}(\sqrt{3}k_{0}^{R}-k_{0}^{I})(|k_{0}|^{2}-1)^{2}},
\end{align}
where $(\cdot)^{R}:=\re (\cdot)$ and $(\cdot)^{I}:=\im (\cdot)$. 
Introducing polar coordinates $(r, \alpha)$ for $k_0$ via $k_0 = r e^{i\alpha}$, we can write
$$h(k_0) = f(r) g(\alpha),$$
where
$$f(r) := \frac{r^4+r^2+1}{ \left(r^2-1\right)^2}, \qquad g(\alpha) := \frac{(\sin{\alpha} + \sqrt{3} \cos{\alpha} )^2}{2\sin(\alpha ) (\sqrt{3} \cos{\alpha} -\sin{\alpha } )}.$$
The function $f(r)$ is a bijection from $(1, \infty)$ onto $(1, \infty)$. The function $g(\alpha)$ is a bijection from $(0,\pi/6)$ onto $(4, \infty)$, and a bijection from $(-\pi/6, 0)$ onto $(-\infty, -1/2)$. We deduce that $h(k_0) = f(r)g(\alpha)$ satisfies $h(k_0) > 4$ for all $k_0 \in D_{\mathrm{reg}}^R = \{r e^{i\alpha} \, | \, r > 1, \ \alpha \in (0, \pi/6)\}$ and $h(k_0) < -1/2$ for all $k_0 \in D_{\mathrm{sing}}^R = \{r e^{i\alpha} \, | \, r > 1, \ \alpha \in (-\pi/6, 0)\}$.
Using the symmetry $f(r)g(\alpha) = f(1/r)g(\alpha + \pi)$, we can extend these inequalities to $D_{\mathrm{reg}}^L$ and $D_{\mathrm{sing}}^L$, which gives
$$h(k_0) > 4 \quad \text{for $k_0 \in D_{\mathrm{reg}}$;} \qquad 
h(k_0) < - 1/2 \quad \text{for $k_0 \in D_{\mathrm{sing}}$}.$$

Suppose that $k_0 \in D_{\mathrm{reg}}$. Since $h(k_0) > 4$, (\ref{detIminusA}) yields
$$\det(I-A) \geq 1-2|A_{11}| + h(k_0) |A_{11}|^{2}
> 1-2|A_{11}| + 4 |A_{11}|^{2} = (1-2|A_{11}|)^{2} \geq 0,$$
showing that $\det(I-B) = \det(I-A) > 0$ for all $x,t$. By (\ref{n31trB}), $u(x,t)$ is a smooth function of $x,t$ as long as $\det(I-B) \neq 0$. The reality of $u(x,t)$ is a consequence of Theorem \ref{inverseth}, but can also be verified directly from (\ref{n31breather}) and (\ref{ubreather}) using the relations $\bar{\tilde{c}}_{k_0} = \tilde{d}_{k_0}$ and
\begin{align*}
& \begin{pmatrix} n_3(\omega k_0)  \\ n_3(\bar{k}_0) \end{pmatrix} = (I - A)^{-1} \begin{pmatrix} 1 \\ 1 \end{pmatrix}.
\end{align*}
This completes the proof of $(a)$.

Suppose now that $k_0 \in D_{\mathrm{sing}}$ and $c_{k_0} \in \C \setminus \{0\}$. 
In this case,
\begin{align*}
A_{11} = c_1 e^{(l_{1}(k_{0})-l_{3}(k_{0}))x+(z_{1}(k_{0})-z_{3}(k_{0}))t},
\end{align*}
where $c_1$ is a nonzero complex constant. Since $\re(l_{1}(k_{0})-l_{3}(k_{0})) \neq 0$, we see that for each fixed $t\geq 0$, $|A_{11}|$ takes on any value in $(0, \infty)$ as $x$ ranges over $\R$. 
Moreover, since $h(k_0) < -1/2$ is strictly negative, (\ref{detIminusA}) shows that $\det(I - A) \to 1$ as $|A_{11}| \to 0$ and that $\det(I - A) \to -\infty$ as $|A_{11}| \to \infty$. This shows that the function $x \mapsto \det(I-B) = \det(I-A)$ has at least one zero for any fixed $t \geq 0$.
\end{proof}

\subsection*{Acknowledgements}
Support is acknowledged from the Novo Nordisk Fonden Project, Grant 0064428, the European Research Council, Grant Agreement No. 682537, the Swedish Research Council, Grant No. 2015-05430, Grant No. 2021-04626, and Grant No. 2021-03877, and the Ruth and Nils-Erik Stenb\"ack Foundation.

\bibliographystyle{plain}
\bibliography{is}

\end{document}